\documentclass[a4paper,11pt]{article}
\usepackage{amsfonts}
\usepackage{amssymb}
\usepackage{amsmath}
\usepackage{amsthm}
\usepackage{verbatim}
\usepackage{algorithmic}
\theoremstyle{plain}
\newtheorem{theorem}{Theorem}[section]
\newtheorem{lemma}[theorem]{Lemma}

\newtheorem{proposition}[theorem]{Proposition}
\theoremstyle{definition}

\newcommand{\mA}{\mathfrak{A}}
\newcommand{\mB}{\mathfrak{B}}
\newcommand{\mM}{\mathfrak{M}}
\newcommand{\mfo}{\models_{\mathrm{FO}}}
\newcommand{\Uf}{\mathit{Uf}}

\newcommand{\Vg}{\mathit{Vg}}
\newcommand{\Vgp}{\mathit{Vg'}}

\newcommand{\depst}{=\!\!}
\begin{document}
\title{A Double Team Semantics for Generalized Quantifiers}
\author{Antti Kuusisto\footnote{
\noindent
Email:\ \ antti.j.kuusisto@uta.fi}\\
University of Wroc\l aw}

\date{}

\maketitle

\vspace{1cm}

\noindent
We investigate extensions of dependence logic with generalized quantifiers.
We also introduce and investigate the notion of a \emph{generalized atom}.
We define a system of semantics that can accommodate
variants of dependence logic, possibly extended with generalized quantifiers and generalized
atoms, under the same umbrella framework. The semantics is based on pairs of teams, or \emph{double teams}.
We also devise a game-theoretic semantics equivalent to the double team semantics.
We make use of the double team semantics by defining a
logic $\mathrm{DC}^2$, which canonically fuses together
\emph{two-variable dependence logic} $\mathrm{D}^2$ and 
\emph{two-variable logic with counting quantifiers} $\mathrm{FOC}^2$.
We establish that the satisfiability and finite satisfiability problems of $\mathrm{DC}^2$
are complete for $\mathrm{NEXPTIME}$.
%
%
%
%
%
%
%
%
\section{Introduction}
%

%
%

%
\emph{Independence-friendly logic} 
is an extension of first-order logic motivated
by issues concerning Henkin quantifiers
and game-theoretic semantics.
Indepen- dence-friendly logic, also known as $\mathrm{IF}$-logic,
was first defined in \cite{sandu}.
The logic extends first-order logic $\mathrm{FO}$ by quantifiers
of the type $\exists x/\{y_1,...,y_k\}$. The
background intuition concerning the interpretion of
these quantifiers is that when a
formula $\exists x/\{y_1,...,y_k\}\, \varphi$ is
evaluated game-theoretically, then the value
of $x$ is chosen in ignorance of the values of
the variables $y_1,...,y_k$.
While game-theoretic semantics of ordinary first-order logic
gives rise to a game of perfect information, the game for
$\mathrm{IF}$-logic is a game of imperfect information.
In \cite{hodges}, Hodges gave a compositional
semantics for $\mathrm{IF}$-logic.
While ordinary Tarskian semantics for first-order logic is
based on evaluating formulae with respect to single \emph{assignments}
(functions that give values to variables in the domain of a model),
the semantics of Hodges is based on \emph{sets} of assignments. 
In \cite{vaananen}, V\"{a}\"{a}n\"{a}nen introduced
\emph{dependence logic}, which provides a
novel alternative approach to issues concerning independence-friendly logic
and Henkin quantifiers.
Instead of quantifiers of the type 
$\exists x/\{y_1,...,y_k\}$,
%
%
dependence logic extends first-order logic
by novel atomic expressions $\depst(x_1,...,x_k)$, which state that the value
of $x_k$ is determined by the values of $x_1,...,x_{k-1}$.
The compositional
semantics of dependence logic is
similar to Hodges' semantics for $\mathrm{IF}$-logic.
The semantics is formulated in terms of sets of assignments.
V\"{a}\"{a}n\"{a}nen named such sets \emph{teams}, and 
since then, the related semantic framework
has been called \emph{team semantics}.
After the introduction of dependence logic, research
on team semantics has been very active, and a
notably large number of related papers has appeared
in the course of a relatively short period. In addition to
dependence logic, several related logics have been introduced
and studied.
\emph{Independence logic}, introduced in \cite{indep},
extends first-order logic with atoms of the type $x\, \bot\, y$.
The intuitive meaning of this atom is that 
$x$ and $y$ are independent of each other in the sense that 
nothing can be said about the value of $x$ based on the value of $y$,
and vice versa. \emph{Independence logic} 
even allows for atoms $\overline{x}\, \bot_{\overline{z}}\, \overline{y}$,
which state that the tuples $\overline{x}$
and $\overline{y}$ are independent when the values of
the variables in $\overline{z}$ are kept constant; see \cite{indep}
for the formal details.
In \cite{galliani}, Galliani introduces \emph{inclusion logic}.
This is yet a further variant of dependence logic. This logic
extends first-order logic by 
atoms of the type $\overline{x}\subseteq\overline{y}$, which state that
any tuple of values defined by $\overline{x}$
is also a tuple of values defined by $\overline{y}$.
The article \cite{galliani} also defines two separate
systems of team semantics, called \emph{strict}
and \emph{lax} semantics.
The systems differ from each other in their treatment of 
the existential quantifier and disjunction.
In strict semantics, the existential quantifier is treated in
the original way familiar from dependence logic.
A model $\mA$ and a team $X$ satisfy a 
formula $\exists x\, \varphi$ if and only if
it is possible to extend\footnote{Strictly speaking,
if the valuations in $X$ already give an interpretation for $x$,
then the team $X$ is \emph{modified} by altering the assignments
rather than extending them.} each valuation $s\in X$
with a pair $(x,a)$, where $a\in\mathit{Dom}(\mA)$,
such that the resulting extended
team satisfies $\varphi$.
The key issue here is that each valuation $s\in X$ is extended by
\emph{exactly one} pair $(x,a)$ that provides an interpretation of $x$.
In lax semantics, each assignment $s\in X$ can be extended by
more than one pair $(x,a)$, resulting in a whole set of
extensions of the valuation $s$.
For the technical difference between the strict and lax semantics
in their treatment of the disjunction, see \cite{galliani}
or Section \ref{deplog} below.
There are interesting and perhaps
surprizing differences between the lax and strict semantics.
It is shown in \cite{gallianihella} that
with lax semantics, inclusion logic
is equiexpressive with \emph{positive greatest fixed point logic},
and therefore captuers $\mathrm{PTIME}$
in restriction to linearly ordered finite models.
On the other hand, with strict semantics,
inclusion logic captures $\mathrm{NP}$,
as observed in \cite{hannula}.
In addition to extensions of first-order logic
with different kinds of atoms, also
\emph{generalized quantifiers} have been studied in the context of team semantics.
In \cite{engstrom}, Engstr\"{o}m defines a semantics
that can
%
%
accommodate generalized quantifiers in the 
framework of team semantics.
Inter alia, the article \cite{engstrom}
studies \emph{branching
quantifiers} consisting of partially ordered generalized quantifiers.
%
%
%
%
%
%
Investigations in the setting of \cite{engstrom} have been recently continued
for example in the articles \cite{jufe} and \cite{jufeva}.
In this article, we define a semantics that can deal with extensions of
dependence logic and its variants with generalized quantifiers.
Our semantics differs
from the semantics given \cite{engstrom}.
Our semantics is based on \emph{double teams}.
There are several reasons---discussed below---why we believe that the
double team semantics is particularly
\emph{natural, general, and useful}.
The double team sematics we shall define is \emph{fully symmetric}
in the sense that it respects obvious \emph{canonical duality principles}
concerning negation. The double team semantics is also \emph{compositional for negation} in a very natural way.
In investigations related to team semantics, the syntax of the logic investigated is usually given in 
\emph{negation normal form}. This means that
negations are only allowed in front of atomic formulae.\footnote{
In some cases non-first-order atoms cannot be negated at all.}
In the framework of double team semantics, such
syntactic limitations are avoided in a natural way.
In addition to the double team semantics,
we also define a corresponding \emph{canonical game-theoretic semantics},
and prove its equivalence to the double team semantics.
The double team semantics, and its game
theoretic counterpart, provide a suitable
setting for the definition of a
notion of a \emph{minor quantifier}.
This is a slight generalization of
Lindstr\"{o}m's
definition of a generalized
quantifier in \cite{lindstrom}.
The notion of a minor quantifier nicely enables the accommodation of
the lax and strict interpretations
of the existential quantifier under the same
umbrella framework.
The strict and lax
interpretations of the existential quantifier
canonically give rise to two corresponding minor quantifiers.
Furthermore, it turns out that
the ordinary existential quantifier
gives rise to a third minor quantifier different from the
strict and lax quantifiers.
The semantic framework based on double teams
provides a natural setting for the \emph{interpretation of
the meaning} of the strict and lax quantifiers.
In particular, the framework enables the investigation of 
the relationship between ordinary generalized quantifiers
and the strict and lax quantifiers, thereby providing novel \emph{insight} into
the nature of these formal tools that occupy an important role in the
current research in team semantics.
In addition to the notion of a minor quantifier,
we introduce the notion of a \emph{generalized atom}.
Generalized atoms can be used in order to declare properties of
(double) teams. The atoms $\depst(x_1,...,x_k)$, $\overline{x}\, \bot_{\overline{z}}\, \overline{y}$
and $\overline{x}\, \subseteq\, \overline{y}$
%
%
are examples of generalized atoms.
In addition to minor quantifiers,
the double team semantics and its game-theoretic counterpart accommodate generalized atoms
under the same general system of semantics.
Generalized atoms have previously been briefly mentioned in
\cite{kuusisto1} and defined in
the technical report \cite{kuusisto2}.
%

%
%
%
Recent research in team semantics has
%
%
revealed---as one perhaps could expect---that subtle changes in semantic choices, such as using
the lax semantics instead the strict semantics, can give rise to logics with different
expressivities. To understand 
related phenomena better, it definitely makes sense
to study team semantics based systems in a
general unified umbrella framework.
%

%
%

%
In order to make direct use of the generality of the double team semantics,
we define the logic $\mathrm{DC}^2$, which extends
\emph{two-variable dependence logic} $\mathrm{D}^2$ by
counting quantifiers $\exists^{\geq k}$.
%
%
%
We prove that the satisfiability problem of this logic is decidable.
In fact, we show that both the finite and standard satisfiability problems of
$\mathrm{DC}^2$ are $\mathrm{NEXPTIME}$-complete.
The logic $\mathrm{DC}^2$ is an extension of both
two-variable dependence logic $\mathrm{D}^2$ and
\emph{two-variable logic with counting} $\mathrm{FOC}^2$.
It was show in \cite{pratthartmann} that the satisfiability
and finite satisfiability problems of $\mathrm{FOC}^2$
are $\mathrm{NEXPTIME}$-complete.
In \cite{kontinen},
the corresponding problems for
$\mathrm{D}^2$ were shown to 
also be $\mathrm{NEXPTIME}$-complete.
%
%

%
Research on two-variable logics is
currently particularly active.
Recent articles in the field include for example
\cite{
IEEEonedimensional:kieronski,
IEEEonedimensional:charatonic,
IEEEonedimensional:kieronskimichaliszyn,
IEEEonedimensional:kieronskitendera,
IEEEonedimensional:zeume,
IEEEonedimensional:tendera},
and several others.
Mainly the related research has concerned decidability and complexity issues
in restriction to particlar classes of 
structures, and also questions
related to different built-in features and operators that 
increase the expressivity of the base language.
Team semantics has so far been discussed in
this context only in \cite{kontinen}.
The article \cite{kontinen} discusses ordinary
two-variable dependence logic $\mathrm{D}^2$,
which does not include counting quantifiers. In fact, when writing
\cite{kontinen}, no direct semantics for 
counting quantifiers was available in the team semantics
framework.\footnote{Counting quantifiers are first-order definable,
so indirect access to them would have been possible.}
The double team appoach provides an appropriate canonical system of
semantics, and furthermore, facilitates the $\mathrm{NEXPTIME}$-completeness
proof given below. Concerning the proof, our objective is not so much to study
the particular logic $\mathrm{DC}^2$.
Instead, we wish to demonstrate how the double team framework can
\emph{in practise} be used in order to study
fragments of team semantics based logics extended with generalized quantifiers.
%
%

%
Our double team semantics provides a general system 
that can deal with generalized quantifiers as well as
generalized atoms, but on the face of it, the move from single teams to double teams
may seem like an undesirable step towards a more complicated
framework. We claim that this issue is not so simple, for two
reasons. Firstly, the syntax of most variants of dependence logic is
currently given in negation normal form, leading to systems
with more connectives and quantifiers than necessary. Disjunction and conjunction
have to be \emph{both} included as primitive connectives in the logics,
and the same applies to the existential and universal quantifiers.
This leads to longer proofs.
Secondly, we shall in fact briefly discuss below in Section \ref{singleteam} a
semantics which is rather similar to our
double team semantics---facilitating investigations analogous to
those carried out in this article---but formulated in terms of single teams.
%
%
%

%
Finally, it is worth noting that while the double team semantics can be used in
investigations related to
dependence logic and its variants, it is also a
\emph{canonical semantics for ordinary extensions of first-order logic with generalized quantifiers},
i.e., extensions that do not include novel atomic formulae, such as
dependence atoms.
The structure of this article is as follows.
In Sections \ref{two} and \ref{deplog} we discuss the necessary background definitions.
In Section \ref{doubleteamsemantics} we define the double team
semantics and discuss some of its basic properties.
In Sections \ref{operators} and
\ref{minorquantifiers} we introduce and investigate generalized atoms
and minor quantifiers. In Section \ref{gamesemantics}
we define a game-theoretic counterpart
for the double team semantics. We also show that the two systems of
semantics are equivalent.
In Section \ref{complexity} we investigate the logic $\mathrm{DC}^2$.
In particular, we prove $\mathrm{NEXPTIME}$-completeness of
the satisfiability and finite satisfiability problems of the logic.
In Section \ref{singleteam} we briefly discuss a single team
semantics for generalized quantifiers.
\section{Preliminaries}\label{two}
Let $\mathbb{Z}_+$ denote the set of positive integers,
and let $\mathrm{VAR} = \{\, v_i\ |\ i\in\mathbb{Z}_+\ \}$ be
the set of  exactly all first-order \emph{variable symbols}.
We shall mainly use metavariables $x,y,z,x_1,x_2$, etc.,
in order to refer to variable symbols in $\mathrm{VAR}$.
We let $\overline{x},\overline{y},\overline{z},\overline{x}_1,\overline{x}_2$, etc.,
denote finite nonempty tuples of variable symbols, i.e., 
tuples in $\mathrm{VAR}^n$ for some $n\in\mathbb{Z}_+$.
Let $X\subseteq\mathrm{VAR}$ be a \emph{finite}, possibly empty set.
Let $\mathfrak{A}$ be a model with the domain $A$.
We do not allow for models to have an empty domain,
so $A\not=\emptyset$.
A function $f:X\rightarrow A$ is called an
\emph{assignment} for the model $\mathfrak{A}$.
Let $\overline{a}$ be any finite nonempty tuple.
We let $\overline{a}(k)$ denote the $k$-th member of the tuple:
for example $(a,b)(1) = a$ and $(a,b)(2) = b$.
When we write $u\in\overline{a}$, we mean that $u$ is a member of the tuple $\overline{a}$, i.e.,
if $\overline{a} = (a_1,...,a_n)$, then $u\in\overline{a}$ iff $u\in\{a_1,...,a_n\}$.
If $f$ is a function mapping into some set $S^k$ of tuples of the length $k\in\mathbb{Z}_+$,
then $f_i$ denotes the function with the same domain as $f$ defined such that
$$f_i(x)\ =\ \bigl(f(x)\bigr)(i),$$
i.e., $f_i$ is the \emph{$i$-th coordinate function} of $f$.
%
%

%
Let $s$ be an assignment with the domain $X$ and for the model $\mathfrak{A}$.
Let $n\in\mathbb{Z}_+$. Let $\overline{x}\in\mathrm{VAR}^n$ be a finite nonempty tuple of variables,
and let $\overline{a}\in A^n$.
Assume that if $\overline{x}$ repeats a variable, then $\overline{a}$ repeats the
corresponding value, i.e.,
if $\overline{x}(i)=  \overline{x}(j)$ for some $i,j\in\{1,...,n\}$,
then $\overline{a}(i)=  \overline{a}(j)$. We say
that $\overline{a}$ \emph{respects $\overline{x}$-repetitions}.
We let $s[\overline{x}/ \overline{a}]$ denote the variable assignment for $\mathfrak{A}$
with the domain $X\cup\{\ x\ |\ x\in\overline{x}\ \}$ defined as follows.
\begin{enumerate}
\item
$s[\overline{x}/ \overline{a}](y)\ =\ \overline{a}(k)$\ \ \ if\ \ \ $y = \overline{x}(k)$,
\item
$s[\overline{x}/ \overline{a}](y)\ =\ s(y)$\ \ \ if\ \ \ $y \not\in \overline{x}$.
\end{enumerate}
Let $T\in\mathcal{P}(A^n)$, where $\mathcal{P}$
denotes the power set operator.
Assume that each tuple in $T$ respects $\overline{x}$-repetitions.
We define 
$$s[\, \overline{x}/ T\, ]\ =\ \{\ s[\, \overline{x}/ \overline{a}\, ]\ |\ \overline{a}\in T\ \}.$$
Note that $s[\, \overline{x}/\emptyset\, ] = \emptyset$.
Let $S$ be a set and 
$\overline{z}$ a tuple of variables of the length $k\in\mathbb{Z}_+$.
If $T\subseteq S^k$ is a relation such each $\overline{u}\in T$
respects $\overline{z}$-repetitions, then we 
say that the relation $T$ \emph{respects  $\overline{z}$-repetitions.}
Let $X\subseteq\mathrm{VAR}$ be a finite, possibly empty set of
first-order variable symbols. Let $U$ be a set of assignments
$f:X\rightarrow A$. Such a set $U$ is a \emph{team}
with the \emph{domain} $X$ and for the
model $\mathfrak{A}$. The
domain $A$ of the model $\mathfrak{A}$ is a
\emph{codomain} of the team $U$.
Note that the empty set is a team for $\mathfrak{A}$, as is the set $\{\emptyset\}$
containing only the empty variable assignment.
The team $\emptyset$ does not have a unique
domain; any finite subset of $\mathrm{VAR}$ is a
domain of $\emptyset$.
The domain of the team $\{\emptyset\}$ is $\emptyset$.
A pair of teams $(U,V)$ is a \emph{double team} if $U$ and $V$ are teams
with the same domain;
the pairs $(U,\emptyset)$, $(\emptyset,V)$ are 
double teams when $U$ and $V$ are teams.
Let $V$ be a nonempty team with the domain $X$ and for the model $\mathfrak{A}$.
Let $n\in\mathbb{Z}_+$, and let $S\subseteq A^n$.
Let $f:V\rightarrow \mathcal{P}(S)$ be a function, where
$\mathcal{P}$ denotes the power set operator.
Let $\overline{x} = (x_1,...,x_n)$ be a tuple of variables.
Assume that for each $s\in V$, the relation $f(s)$ respects
$\overline{x}$-repetitions. Then we say that
\emph{$f$ respects $\overline{x}$-repetitions}.
We define
$$V[\, \overline{x}/f\, ]=\bigcup\limits_{s\, \in\, V}s[\, \overline{x}/f(s)\, ].$$
(Note that if we have $V=\emptyset$, then $V[\overline{x}/ f] = \emptyset$.)
Let $B$ denote the set $\{\ \overline{a}\in A^n\ |\ \overline{a}\text{ respects }
\overline{x}\text{-repetitions}\ \}.$
We let $f':V\rightarrow\mathcal{P}(B)$ denote
the function defined such that $f'(s) = B\setminus f(s)$
for all $s\in V$. Naturally
$$V[\, \overline{x}/f'\, ]=\bigcup\limits_{s\, \in\, V}s[\, \overline{x}/f'(s)\, ].$$
%
%
%

%
%

%
%

%
Let $V$ be a team with the domain $X$
and for the model $\mathfrak{A}$. Let $k\in\mathbb{Z}_+$.
Let $y_1,...,y_k$ be variable symbols.
Assume that $\{y_1,...,y_k\} \subseteq X$. Define 
$$\mathrm{Rel}\bigl(V,\mathfrak{A},(y_1,...,y_k)\bigr)\ =\ \{\ \bigl(s(y_1),...,s(y_k)\bigr)\ |\ s\in V\ \}.$$
%
%
%
%
If $V$ is empty, then the obtained
relation is the empty relation. Occasionally, when the model $\mA$ is
clear from the context, we simply write $\mathrm{Rel}\bigl(V,(y_1,...,y_k)\bigr)$
instead of $\mathrm{Rel}\bigl(V,\mathfrak{A},(y_1,...,y_k)\bigr)$.
Let $(i_1,...,i_n)$ be a non-empty sequence of positive integers.
A \emph{generalized quantifier} (cf. \cite{lindstrom}) of the type $(i_1,...,i_n)$ is a class $\mathcal{C}$ of structures
$(A,B_1,...,B_n)$ such that the following conditions hold.
\begin{enumerate}
\item
$A\not=\emptyset$.
\item
For each $j\in\{1,...,n\}$, we have $B_j \subseteq A^{i_j}$.
\item
If $(A',B_1',...,B_n')\in\mathcal{C}$ and if there is an isomorphism $f:A'\rightarrow A''$
from $(A',B_1',...,B_n')$ to another structure $(A'',B_1'',...,B_n'')$,
then we have $(A'',B_1'',...,B_n'')\in\mathcal{C}$.
\end{enumerate} 
Let $Q$ be a generalized quantifier of the type $(i_1,...,i_n)$.
We let $\overline{Q}$ denote the generalized quantifier of the type $(i_1,...,i_n)$
defined such that 
$$\overline{Q}\ =\ \{\ (A,C_1,...,C_n)\ |\ (A,C_1,...,C_n)\not\in Q\ \}.$$
Let $\mathfrak{A}$ be a model 
with the domain $A$.
We define $Q^{\mathfrak{A}}$ to be the set
$$\{\ (B_1,...,B_n)\ |\ (A,B_1,...,B_n)\in Q\ \}.$$
Similarly, we define 
$$\overline{Q}^{\mathfrak{A}}\ =
\ \{\ (B_1,...,B_n)\ |\ (A,B_1,...,B_n)
\in\overline{Q}\ \}.$$
If $\varphi$ is a formula of first-order logic,
possibly extended with generalized quantifiers,
we write $\mA,s\models_{\mathrm{FO}}\varphi$ when
the model $\mA$ satisfies $\varphi$  under the assignment $s$.
The related semantic clause for generalized quantifiers is as follows.
Let $Q$ denote a generalized quantifier of the type $(i_1,...,i_n)$.
Consider expressions of the type $Q\overline{x}_1,...,\overline{x}_n(\varphi_1,...,\varphi_n)$,
where $\overline{x}_j$ is a tuple of variables of the length $i_j$,
and $\varphi_j$ is a formula of first-order logic, possibly extended with generalized
quantifiers. Let $\mA$ be a model with domain $A$ and $s$ an assignment with codomain $A$.
If $\overline{x}$ is a tuple of variables of the length $k\in\mathbb{Z}_+$, we let
$A^k[\overline{x}]$ denote the set of exactly all tuples in $A^k$ that
respect $\overline{x}$-repetitions.
We define $\mA,s\models_{\mathrm{FO}} Q\overline{x}_1,...,\overline{x}_n(\varphi_1,...,\varphi_n)$ iff
$\bigl(\, A, S_1,...,S_n\, \bigr)\, \in\, Q^{\mA},$
where
$S_j\ =\ \{\ \overline{a}\in A^{i_j}[\overline{x}_j]\ |\ \mA,s[\overline{x}_j/
\overline{a}]\models_{\mathrm{FO}}\varphi_j\ \}.$
%
%
%
%
The quantifier $Q\overline{x}_1,...,\overline{x}_n$ binds the 
variables $\overline{x}_j$ in the formula $\varphi_j$.
We of course assume that $s$ interprets all the free variables in 
 the formula $Q\overline{x}_1,...,\overline{x}_n(\varphi_1,...,\varphi_n)$,
and that $\mA$ interprets the non-logical symbols that
appear in the formulae $\varphi_1,...,\varphi_n$.
Below, once we have defined the notion of a minor quantifier,
we occasionally call generalized quantifiers
\emph{ordinary generalized quantifiers}.
In the investigations below, each instance of a subformula of a formula $\varphi$
is considered to be a distinct subformula: for example, in the formula $(P(x)\vee P(x))$,
the left and right instances of the formula $P(x)$ are considered to be two \emph{distinct} subformulae
of the formula $(P(x)\vee P(x))$. It is not important how this distinction is achieved formally.
We let $\mathrm{SUB}_{\varphi}$ denote the set of subformulae of
$\varphi$. For example the set $\mathrm{SUB}_{(P(x)\vee P(x))}$
has three subformulae in it, the formula $(P(x)\vee P(x))$ and both
instances of $P(x)$.
We consider only models with a purely relational vocabulary,
without function symbols or constant symbols.
When we informally leave brackets unwritten in formulae,
the order of priority of binary connectives is such that
$\wedge$ is first, and then come
$\vee$, $\rightarrow$ and $\leftrightarrow$, in
the given order.
The notation $\mA,[x\mapsto a,y\mapsto b]\models\varphi$
means that $\mA,s\models\varphi$ when $s$ is an assignment
whose domain is $\{x,y\}$,
and it holds that $s(x) = a$ and $s(y) = b$.
\section{Dependence logic and its variants}\label{deplog}
%

%
%
%
%
%
Let $\tau$ be a vocabulary containing
relation symbols only.
Let $\mathcal{A}(\tau)$ be the 
smallest set $T$ such that the following conditions
are satisfied.
\begin{enumerate}
\item
Let $x_1$ and $x_2$ be (not necessarily distinct) variable symbols.
Then $x_1 = x_2\ \in\ T$.
\item
Let $k$ be a positive integer.
If $R\in\tau$ is a $k$-ary relation symbol
and $x_1,...,x_k$ are (not necessarily distinct) variable symbols,
then $R(x_1,...,x_k)\in T$.
\item
Let $k$ be a positive integer.
If $x_1,...,x_k$ are (not necessarily distinct) variable symbols,
then $\depst(x_1,...,x_k)\in T$.
\end{enumerate}
Formulae formed by the rules $1$ and $2$ above are 
called \emph{first-order atoms}.
The set of $\tau$-formulae of
\emph{dependence logic} $\mathrm{D}$ is
the smallest set $T$ such that the following conditions hold.
\begin{enumerate}
\item
$\mathcal{A}(\tau)\subseteq T$.
\item
If $\varphi\in \mathcal{A}(\tau)$, then $\neg\varphi\in T$.
\item
If $\varphi,\psi\in T$, then $(\varphi\vee\psi)\in T$.
\item
If $\varphi,\psi\in T$, then $(\varphi\wedge\psi)\in T$.
\item
If $\varphi\in T$ and $z\in\mathrm{VAR}$, then $\exists z\, \varphi\in T$.
\item
If $\varphi\in T$ and $z\in\mathrm{VAR}$, then $\forall z\, \varphi\in T$.
\end{enumerate}
\emph{Two-variable dependence logic} $\mathrm{D}^2$ is a fragment of $\mathrm{D}$.
Let $\sigma$ be a vocabulary containing relation symbols
only. Assume each symbol in $\sigma$ is either of the arity $1$ or $2$.
Fix two distinct variable symbols $x$ and $y$.
The set $\mathcal{A}(\sigma)$ of atomic $\sigma$-formulae of $\mathrm{D}^2$ is the
smallest set $T$ defined as follows.
\begin{enumerate}
\item
Assume $P\in\sigma$ and $R\in\sigma$ are unary and binary relation symbols, respectively.
Let $z,z'\in\{x,y\}$ be (not necessarily distinct) variables. Then $P(z)\in T$ and $R(z,z')\in T$.
\item
Let $z,z'\in\{x,y\}$ be (not necessarily distinct) variables.
Then we have $\depst(z)\in T$ and $\depst(z,z')\in T$.
Also $z = z'\in T$.
\end{enumerate}
The set of $\sigma$-formulae of $D^2$ is the smallest
set $T$ satisfying the following conditions.
\begin{enumerate}
\item
$\mathcal{A}(\sigma)\subseteq T$.
\item
If $\varphi\in \mathcal{A}(\sigma)$, then $\neg\varphi\in T$.
\item
If $\varphi,\psi\in T$, then $(\varphi\vee\psi)\in T$.
\item
If $\varphi,\psi\in T$, then $(\varphi\wedge\psi)\in T$.
\item
If $\varphi\in T$ and $z\in\{x,y\}$, then $\exists z\, \varphi\in T$.
\item
If $\varphi\in T$ and $z\in\{x,y\}$, then $\forall z\, \varphi\in T$.
\end{enumerate}
We next define the sematics of $\mathrm{D}$.
In the definition, $\mA$ denotes a model and $U$ a team. The domain of
the team $U$ is always assumed to contain the free variables in the formulae,
and the codomain of $U$ is of course assumed to be the domain $A$ of 
the model $\mA$. Furthermore, it is assumed that the vocabulary of the 
model $\mA$ contains the non-logical symbols in the formulae.
The following clauses define the semantics of $\mathrm{D}$.
\[
\begin{array}{lll}
 \mathfrak{A}\models_U x_1=x_2 & \ \Leftrightarrow \ & 
          \forall s\in U\bigl(\mathfrak{A},s\models_{\mathrm{FO}} x_1=x_2\bigr).\\ 
 \mathfrak{A}\models_U R(x_1,...,x_m)& \ \Leftrightarrow \ & 
          \forall s\in U\bigl(\mathfrak{A},s\models_{\mathrm{FO}} R(x_1,...,x_m)\bigr).\\
 \mathfrak{A}\models_U\ \depst(x_1,...,x_m)& \ \Leftrightarrow \ & 
          \text{if  there exist assignments }s,t\in U\text{ such that}\\
                 & &\text{}s(x_i) = t(x_i)\text{ for all }i\in\{1,...,m\}\setminus\{m\},\\
                 & &\text{then we have }s(x_m) = t(x_m).\\
 \mathfrak{A}\models_U \neg\, x_1=x_2 & \ \Leftrightarrow \ & 
          \forall s\in U\bigl(\mathfrak{A},s\not\models_{\mathrm{FO}} x_1=x_2\bigr).\\
 \mathfrak{A}\models_U \neg\, R(x_1,...,x_m)& \ \Leftrightarrow \ & 
          \forall s\in U\bigl(\mathfrak{A},s\not\models_{\mathrm{FO}} R(x_1,...,x_m)\bigr).\\
 \mathfrak{A}\models_U \neg \depst (y_1,...,y_m)& \ \Leftrightarrow \ & 
          U = \emptyset.\\
   \mathfrak{A}\models_U (\varphi\vee\psi) & \ \Leftrightarrow \ &
           \mathfrak{A}\models_{U_1}\varphi\text{ and }
          \mathfrak{A}\models_{U_2}\psi\text{ for some}\\
         & &U_1,U_2\subseteq U\text{ such that }
          U_1\cup U_2 = U.\\
   \mathfrak{A}\models_U (\varphi\wedge\psi) & \ \Leftrightarrow \ &
           \mathfrak{A}\models_{U_1}\varphi\text{ and }
          \mathfrak{A}\models_{U_2}\psi.\\
   \mathfrak{A}\models_U \exists z\, \varphi& \ \Leftrightarrow \ &
           \mathfrak{A}\models_{U[z/f]}\varphi\text{ for some }f:U\rightarrow A.\\
   \mathfrak{A}\models_U \forall z\, \varphi& \ \Leftrightarrow \ &
           \mathfrak{A}\models_{U[z/A]}\varphi.
\end{array}
\]
Notice that $\mA\models_{U}\, \depst(z)$ iff either $U= \emptyset$ or
$s(z) = s'(z)$ for all $s,s'\in U$. Formulae of $\mathrm{D}$ that do not contain
instances of atoms $\depst(x_1,...,x_k)$
are called \emph{first-order formulae}. It is well-known and easy to show that
for first-order formulae, $\mA\models_{U}\varphi$ iff
we have $\mA,s\models_{\mathrm{FO}}\varphi$ for all $s\in U$.
Variants of dependence logic studied in the current literature include
%
%
for example \emph{inclusion logic} \cite{galliani}.
The syntax of inclusion logic is the same as that of dependence logic, with
the exception that instead\,  of atomic expressions $\depst(x_1,...,x_m)$,
the non-first-order atoms in inclusion logic are \emph{inclusion atoms}
%
%
%
$(y_1,...,y_k)\, \subseteq\, (z_1,...,z_k)$, and negated inclusion atoms 
are not allowed.
Inclusion atoms are interpreted such that 
$\mA\models_{U} (y_1,...,y_k)\, \subseteq\, (z_1,...,z_k)$ iff
$\mathit{Rel}\bigl(U,(y_1,...,y_k)\bigr) \subseteq \mathit{Rel}\bigr(U,(z_1,...,z_k)\bigr)$.
The existential quantifier is interpreted such that
$\mA\models_{U}\exists z\, \varphi$ iff $\mA\models_{U[x/S]}\varphi$ for
some non-empty set $S\subseteq\mathit{Dom}(\mA)$.
Other semantic clauses are exactly the same as
the ones given for dependence logic above.
This results in the interpretation of inclusion logic with \emph{lax semantics}.
Inclusion logic can also be interpreted using \emph{strict semantics}. The difference with
lax semantics is the interpretation of the existential quantifier and disjunction. 
For the existential quantifier, the semantic clause is exactly the
same as that given for dependence logic above.
For the disjunction, the semantic clause
dictates that $\mA\models_{U}\varphi\vee\psi$ iff 
we have $\mA\models_{U_1} \varphi$ and $\mA\models_{U_2}\psi$
for some teams $U_1,U_2\subseteq U$ such that $U_1\cup U_2 = U$ and
$U_1\cap U_2 = \emptyset$.
%
%
%
%
%

%
It is established in \cite{gallianihella} that
with lax semantics, inclusion logic
is equiexpressive with the \emph{positive greatest fixed point logic},
and therefore captures $\mathrm{PTIME}$
in restriction to ordered finite models.
With strict semantics,
inclusion logic captures $\mathrm{NP}$,
as observed in \cite{hannula}.
Also \emph{independence logic} \cite{indep} is a widely studied variant
of dependence logic. For formal details related to independence logic,
see \cite{indep}.
\section{A double team semantics}\label{doubleteamsemantics}
In ordinary team semantics, the \emph{background intuition}\footnote{Intuition only!} concerning
the satisfaction of formulae is that a team
satisfies a formula $\varphi$ iff every member of the team satisfies $\varphi$.
In the double team semantics, the background intuition is that a double team $(U,V)$ satisfies a
formula iff every assignment in the team $U$ satisfies the formula,
and furthermore, every assignment in the team $V$ \emph{falsifies} the formula.
Both in ordinary and double team semantics, the intuition is
actually even formally valid when the investigated
formula is a first-order formula.
The truth definition for first-order atoms and connectives is as follows.
\[
\begin{array}{lll}
 \mathfrak{A},(U,V)\models y_1=y_2 & \ \Leftrightarrow \ & 
          \forall s\in U\bigl(\mathfrak{A},s\models_{\mathrm{FO}} y_1=y_2\bigr)\text{ and }\\
          & & \forall s\in V\bigl(\mathfrak{A},s\not\models_{\mathrm{FO}} y_1=y_2\bigr).\\ 
 \mathfrak{A},(U,V)\models R(y_1,...,y_m)& \ \Leftrightarrow \ & 
          \forall s\in U\bigl(\mathfrak{A},s\models_{\mathrm{FO}} R(y_1,...,y_m)\bigr)\text{ and }\\
          & & \forall s\in V\bigl(\mathfrak{A},s\not\models_{\mathrm{FO}} R(y_1,...,y_m)\bigr).\\ 
 \mathfrak{A},(U,V)\models \neg\varphi & \ \Leftrightarrow \ & 
 \mathfrak{A},(V,U)\models \varphi.\\ 

   \mathfrak{A},(U,V)\models(\varphi\vee\psi) & \ \Leftrightarrow \ &
%
%
%
           \mathfrak{A},(U_1,V)\models\varphi\text{ and }
          \mathfrak{A},(U_2,V)\models\psi\text{ for}\\
         & &\text{some }U_1,U_2\subseteq U\text{ such that }
          U_1\cup U_2 = U.
\end{array}
\]
The background intuition concerning the
satisfaction of a quantified formula $Qx\, \varphi(x)$
is based on the idea that the set of witnesses of
$Qx\, \varphi(x)$ is the set of \emph{exactly all} values $b$
such that $\varphi(b)$ holds. A proper subset will not do.
This intuition easily generalizes to concern generalized
quantfiers $Q$ of arbitrary types.
For a generalized quantifier $Q$ of the type $(i_1,...,i_n)$, we define
$$\mathfrak{A},(U,V)\models Q\overline{x}_1,...,\overline{x}_n(\varphi_1,...,\varphi_n)$$
if and only if there exist functions $f:U\rightarrow  Q^{\mathfrak{A}}$ and
$g:V\rightarrow \overline{Q}^{\mathfrak{A}}$ such that
\[
\begin{array}{c}
 \mathfrak{A},\bigl(\, U[\, \overline{x}_1/f_1]\cup V[\, \overline{x}_1/{g_1}\, ],
\, U[\, \overline{x}_1/{f_1}'\, ]\cup V[\, \overline{x}_1/{g_1}'\,  ]\, \bigr)\models\varphi_1,\\
 \vdots\\
 \mathfrak{A},\bigl(\, U[\, \overline{x}_n/f_n]\cup V[\, \overline{x}_n/{g_n}\, ],
\, U[\, \overline{x}_n/{f_n}'\, ]\cup V[\, \overline{x}_n/{g_n}'\, ]\, \bigr)\models\varphi_n.\\
\end{array}
\]
The functions $f$ and $g$ must have the property that for each $i\in\{1,...,n\}$,
the coordinate functions $f_i$ and $g_i$ (and thereby also the functions $f_i'$
and $g_i'$) respect $\overline{x}_i$-repetitions. 
Notice that if $W = \emptyset$ is the empty team
and $f:W\rightarrow Q^{\mA}$ the empty function ($f = \emptyset$),
then $W[\overline{x}/f] = \emptyset$.
\begin{proposition}\label{first-ordercorrespondence}
Let $\varphi$ be a formula of first-order logic, possibly extended with generalized
quantifiers. Let $(U,V)$ be a double team. Then
$$\mathfrak{A},(U,V)\models\varphi\
\text{ iff }\
\forall s\in U\forall t\in V\bigl(\mathfrak{A},s\models_{\mathrm{FO}}
\varphi\text{ and }\mathfrak{A},t\not\models_{\mathrm{FO}}\varphi\bigr).$$
\end{proposition}
\begin{proof}
The claim is established by a straightforward induction on the structure of formulae.
\end{proof}
When $\varphi$ is a sentence, we define $\mathfrak{A}\models\varphi$
iff $\mathfrak{A},(\{\emptyset\},\emptyset)\models\varphi$.
When $\mathfrak{A}$ is known from the context, we may write $(U,V)\models\psi$ instead
of $\mathfrak{A},(U,V)\models\psi$.
%
%
%
%
%

%
%
%
%
%

%
Note that the truth definition of disjunction could be
easily modified without sacrificing 
Proposition \ref{first-ordercorrespondence}.
For example we could define $\mA,(U,V)\models\varphi\vee\psi$ iff
$\mA,(U_1,V\cup\, U_1')\models \varphi$ and $\mA,(U_2,V\cup\, U_2')\models\psi$
for some $U_1,U_2\subseteq U$ such that $U_1\cup U_2 = U$;
here $U_1' = U\setminus U_1$ and $U_2' = U\setminus U_2$.
This definition would perhaps be a better match with 
our truth definition concerning generalized quantifiers.
For the sake of simplicity, we shall mostly ignore such
alternative definitions for connectives in this article.
However, let us define the connective $\vee^s$ such
that $\mA,(U,V)\models \varphi\vee^s\psi$ iff
$\mA,(U_1,V)\models\varphi$ and $\mA,(U_2,V)\models\psi$
for some $U_1,U_2\subseteq U$ such that $U_1\cup U_2 = U$
and $U_1\cap U_2 = \emptyset$.
\section{Generalized atoms}\label{operators}
Let $m$ and $n$ be non-negative integers such that $n+m > 0$.
Let $Q$ be a generalized quantifier of the type $(i_1,...,i_{n+m})$.
Consider atomic expressions of the type
$$A_{Q,n}(\overline{y}_1,...,\overline{y}_n\, ;\, \overline{y}_{n+1},...,\overline{y}_{n+m}),$$
where each $\overline{y}_j$ is a tuple of variables of the length $i_j$, and $A_{Q,n}$
is simply a symbol. 
Extend the double team semantics such that
$$\mathfrak{A},(U,V)\models A_{Q,n}(\overline{y}_1,...,
\overline{y}_n\, ;
\, \overline{y}_{n+1},...,\overline{y}_{n+m})$$
if and only if
$$\bigl(\mathrm{Rel}(U,\mathfrak{A},\overline{y}_1),...,
\mathrm{Rel}(U,\mathfrak{A},\overline{y}_n),\mathrm{Rel}(V,\mathfrak{A},\overline{y}_{n+1}),...,
\mathrm{Rel}(V,\mathfrak{A},\overline{y}_{n+m})\bigr)\in Q^{\mathfrak{A}}.$$
The generalized quantifier $Q$ and the number $n$ define a  
\emph{generalized atom}
%
%
of the type
$$\bigl((i_1,...,i_n),(i_{n+1},...,i_{n+m})\bigr).$$
Note that types of generalized quantifiers are tuples and
types of generalized atoms
are pairs of tuples; exactly one tuple of such a pair
of tuples can be the empty tuple. (We do not
bother ourselves with generalized atoms
of the type $(\emptyset;\emptyset)$ or
generalized quantifiers of the type $\emptyset$.)
We occasionally call generalized atoms
\emph{non-first-order atoms}, while other atoms are
first-order atoms.
\section{Minor quantifiers}\label{minorquantifiers}
In this section we generalize the notion of a generalized quantifier
by Lindstr\"{om} in \cite{lindstrom}.
This way we obtain a framework that can naturally accommodate
in a single umbrella framework the different kinds of semantics for the existential quantifier in
the literature on dependence logic and its later variants.
Minor quantifiers have a natural intuitive interpretation. The interpretation will be discussed in 
Section \ref{gamesemantics}, where we define the game-theoretic counterpart for the double
team semantics.
Let $Q$ be a generalized quantifier of the type $(1)$.
Let $\mathcal{C}$ be a class of
structures $(A,B_+,B_-)$ such that the following conditions hold.
\begin{enumerate}
\item
$A\not=\emptyset$.
\item
$B_+\subseteq A$ and $B_-\subseteq A$.
\item
$B_+\cap B_- = \emptyset$.
\item
If $(C,D_+,D_-)\in\mathcal{C}$ and if there is an isomorphism $f:C\rightarrow E$
from $(C,D_+,D_-)$ to another structure $(E,F_+,F_-)$,
then $(E,F_+,F_-)\in\mathcal{C}$.
\item
For each $(A,B_+,B_-)\in\mathcal{C}$, there exists a
pair $(A,H)\in Q$ such that $B_+\subseteq H$ and $B_-\subseteq A\setminus H$.
\item
If $(A,B_+,B_-)\in\mathcal{C}$, there does \emph{not} exist a
pair $(A,H)\in \overline{Q}$ such that $B_+\subseteq H$ and $B_-\subseteq A\setminus H$.
\item
For each $(A,H)\in Q$, there exists a
tuple $(A,B_-,B_-)\in \mathcal{C}$ such that $B_+\subseteq H$ and $B_-\subseteq A\setminus H$.
\end{enumerate} 
%
%
%
%
We say that \emph{$\mathcal{C}$ witnesses $Q$}.
%
%

%
Consider a pair $(\mathcal{C},\mathcal{D})$ such that 
$\mathcal{C}$ witnesses $Q$ and $\mathcal{D}$ witnesses $\overline{Q}$.
Here $Q$ is a quantifier 
of the type $(1)$.
The pair $(\mathcal{C},\mathcal{D})$
defines a \emph{minor quantifier} of the type $(1)$.
(For the sake of simplicity, we shall not 
define minor quantifiers of any other type.)
Let $M = (\mathcal{C},\mathcal{D})$.
We call $M$ a \emph{minor of} $Q$.
We write $M\leq Q$.
A possible \emph{intuitive} interpretation concerning
the relationship between $Q$ and
the minor quantifier $M$ is that in order to verify $Qx\, \varphi(x)$
in a model $\mA$, one does not necessarily
have to be able to find the set
$B\in Q^{\mA}$ such
that $b\in B$ iff $\varphi(b)$ holds in $\mA$.
Dependening on the quantifier $Q$, it
may be enough to find some smaller
set $B_+\subseteq B$ of values that
verify $\varphi(x)$, possibly 
together with a set $B_-\subseteq\mathit{Dom}(\mA)\setminus B$
of falsifying values. A tuple $(A,B_+,B_-)\in\mathcal{C}$ then
provides the sets $B_+$ and $B_-$.
On the other hand, to \emph{falsify} a formula $Qx\, \psi(x)$,
it suffices to find a tuple $(A,E_+,E_-)$ in $\mathcal{D}$,
where $E_+$ is a set of verifying and $E_-$ a set
of falsifying values for $\psi(x)$.
Therefore minor quantifiers provide a generalized perspective on
generalized quantifiers.
The perspective in some intuitive sense deals with
issues concerning the
\emph{constructive verification and falsification} of formulae.
%

%
%

%
The semantics of minor quantifiers will be highly analogous
to that of ordinary generalized quantifiers.
To make this issue explicit, let us fix
some notational conventions.
Let $M = (\mathcal{C},\mathcal{D})$ be a minor quantifier.
Let $\mA$ be a model with the domain $A$.
Define $M^{\mA} = \{\ (B_+,B_-)\ |\ (A,B_+,B_-)\in \mathcal{C}\ \}$
and $\overline{M}^{\mA} = \{\ (B_+,B_-)\ |\ (A,B_+,B_-)\in \mathcal{D}\ \}$.
Let $U$ be a team and $f:U\rightarrow M^{\mA}$ a function.
When discussing the semantics of minor quantifiers
$M = (\mathcal{C},\mathcal{D})$, we let
$U[ x/f ]$ denote the team $U[ x/f_1 ]$,
while $U[ x/f' ]$ denotes the team $U[ x/f_2 ]$.
Similarly, if $V$ is a team and $g:V\rightarrow \overline{M}^{\mA}$ a
function, we let $V[ x/g ]$ denote the team $U[ x/g_1 ]$
and $U[ x/g' ]$ the team $U[ x/g_2 ]$.
This convention makes the connection between ordinary generalized
quantifiers and minor quantifiers fully explicit. 
All related arguments will be carefully developed below,
so no notational confusion arises.
%
%
%
%
%
%

%
Let $M$ be a minor quantifier of the type $(1)$.
Consider expressions of the type $Mx\, \varphi$.
Extend the double team semantics such that 
$\mA,(U,V)\models Mx\, \varphi$ iff
there exists functions $f:U\rightarrow M^{\mA}$ and
$g:V\rightarrow \overline{M}^{\mA}$ such that
$$\mA,\bigl(U[\, x/f\, ]\cup V[\, x/g\, ],
\, U[\, x/f\, '\, ]\cup V[\, x/g\, '\, ]\bigr)\models \varphi.$$
Notice that the form of the above semantic clause is now the
same as in the case of ordinary quantifiers of the type $(1)$.
The following proposition is easy to establish.
\begin{proposition}\label{minorisnormal}
Let $Q$ be a generalized quantifier
and $M\leq Q$ a minor quantifier.
Let $\varphi$ be a formula of first-order logic extended with any
collection of minor quantifiers and
ordinary generalized quantifiers. Let $\varphi'$
be a formula obtained from $\varphi$ by
replacing any occurrence of $Q$ by $M$,
or alternatively, any occurrence of $M$ by $Q$.
Then
$\mathfrak{A},(U,V)\models\, \varphi\
\text{ iff }\
\mathfrak{A},(U,V)\models\, \varphi'.$
\end{proposition}
Let $Q$ be a generalized quantifier of the type $(1)$.
Notice that $Q$ canonically defines the minor quantifier
$$M_Q\ :=\ \Bigl(\{\ (A,B,A\setminus B)\ |\ (A,B)\in Q\ \},
\{\ (A,B,A\setminus B)\ |\ (A,B)\in\overline{Q}\ \}\Bigr),$$
whose semantics is equivalent to that of $Q$ in
the double team framework.
We can replace any instance of\, $Q$ by $M_Q$
(or vice versa) in any formula $\varphi$, and exactly
the same models and double teams will satisfy the
two formulae.\footnote{The formula $\varphi$ can indeed belong to any extension of
first order logic with ordinary generalized quantifiers,
minor quantifiers, and generalized atoms.}
We call $M_Q$ \emph{the minor quantifier defined by $Q$}.
Ordinary generalized quantifiers can therefore be
seen as special cases of minor quantifiers.
%
%
%

%
Define the \emph{strict existential quantifier} $\exists^s$
to be the minor quantifier $(\mathcal{C},\mathcal{D})$,
where $\mathcal{C}$ contains exactly all triples $(A,B,C)$ such
that $A$ is a nonempty set, $B\subseteq A$ is a singleton set, 
and $C = \emptyset$, while $\mathcal{D}$ contains
exactly all triples $(D,E,F)$ such that $D$ is a nonempty set,
$E = \emptyset$, and $F = D$.
Define the \emph{lax existential quantifier} $\exists^l$
to be the minor quantifier $(\mathcal{C},\mathcal{D})$,
where $\mathcal{C}$ contains exactly all triples $(A,B,C)$ such
that $A$ is a nonempty set, $B\subseteq A$ is a nonempty set, 
and $C = \emptyset$, while $\mathcal{D}$ contains
exactly all triples $(D,E,F)$ such that $D$ is a nonempty set,
$E = \emptyset$, and $F = D$.
Note that neither $\exists^s$ nor $\exists^l$ is
equal to the minor quantifier $M_{\exists}$
defined by the ordinary existential quantifier.
%

%
%
%
%
%
%
%
%
%
%
%
%
%
%
%
%
%

%
\section{Game-theoretic semantics}\label{gamesemantics}
In this section we define a natural game-theoretic semantics for first-order logic
extended with all ordinary generalized
quantifiers of type $(1)$, all minor quantifiers of type $(1)$,
and all generalized atoms.
We only deal with
quantifiers of the type $(1)$ in the
rest of the article for
the sake of simplicity.
Strictly speaking, we could of course avoid
discussing ordinary generalized quantifiers here,
%
%
%
%
%
but we shall discuss them anyway since it makes the exposition of the
background intuitions behind the game-theoretic semantics particularly transparent.
%
%
%
%
%
%

%
Let $\mathfrak{A}$ be a model with the domain $A$.
Let $s$ be an assignment that maps a finite set of first-order variable symbols into $A$.
We define a semantic game $G(\mathfrak{A},s,\#,\varphi)$, where $\#\in\{+,-\}$ is a symbol
and $\varphi$ a formula.
Here we assume that the assignment $s$ interprets all the free variables in $\varphi$.
The game is played by an \emph{agent} $\mathcal{A}$ against an \emph{interrogator} $\mathcal{I}$.
The intuition is that the interrogator poses questions, and the agent tries to
answer them. In a game $G(\mathfrak{A},s,+,\varphi)$,
the agent's task is to maintain that $\varphi$ holds,
while in a game $G(\mathfrak{A},s,-,\varphi)$,
the agent's task is to maintain that $\varphi$ does not hold.
%

%
%
A \emph{play} of the game $G(\mathfrak{A},s,\#,\varphi)$
begins from the \emph{position} $(\mathfrak{A},s,\#,\varphi)$.
All positions of the game are tuples of the form $(\mathfrak{A},t,\#,\psi)$,
where $t$ is a finite assignment for $\mathfrak{A}$, $\#\in\{+,-\}$, and $\psi$ is a
subformula of $\varphi$.
Assume that we have reached a position $(\mathfrak{A},t,\#,\neg\psi)$ in a play of the game.
The play of the game continues from the position $(\mathfrak{A},t,\overline{\#},\psi)$,
where $\overline{\#}\in\{+,-\}\setminus\{\#\}$.
Assume a position $(\mathfrak{A},t,+,\psi\vee \psi')$ has been reached.
Then the player $\mathcal{A}$ chooses exactly one
of the sets $\{\psi,\psi'\}$, $\{\psi\}$, $\{\psi'\}$.
If $\mathcal{A}$ chooses $\{\psi,\psi'\}$, then
$\mathcal{I}$ chooses a formula $\chi\in\{\psi,\psi'\} $,
and the play continues from the position
$(\mathfrak{A},t,+,\chi)$.
If $\mathcal{A}$ chooses $\{\psi\}$, then
the play of the game continues from the position
$(\mathfrak{A},t,+,\psi)$.
If $\mathcal{A}$ chooses $\{\psi'\}$, then
the play continues from the position
$(\mathfrak{A},t,+,\psi')$.\footnote{
Consider the connective $\vee^s$
defined in Section \ref{doubleteamsemantics}.
We can add this connective into the language considered.
The rules for a position $(\mathfrak{A},t,+,\psi\vee^s \psi')$
are  exactly as for $(\mathfrak{A},t,+,\psi\vee\psi')$,
but with the exception that the choice $\{\psi,\psi'\}$
by $\mathcal{A}$ is not allowed.
The rules for a position $(\mathfrak{A},t,-,\psi\vee^s\psi')$
are the same as for a position $(\mathfrak{A},t,-,\psi\vee\psi')$. 
As the reader can easily check, Theorem \ref{gamecorrespondence}
below goes through even when the language is extended by $\vee^s$.
We could consider further connectives and even define a natural notion of a
\emph{minor connective}, but we shall not do that for the sake of brevity.}
The background \emph{intuition} concerning the
disjunction rule is that 
$\mathcal{A}$ makes one of the following three claims.
\begin{enumerate}
\item
Both $\psi$ and $\psi'$ hold.
\item
At least $\psi$ holds.
\item
At least $\psi'$ holds.
\end{enumerate}
If a position $(\mathfrak{A},t,-,\psi\vee\psi')$ has been reached,
%
%
%
the player $\mathcal{I}$ chooses one of the positions $(\mathfrak{A},t,-,\psi)$
and $(\mathfrak{A},t,-,\psi')$.
The play of the game then continues from the position chosen by $\mathcal{I}$.
Assume we have reached a position $(\mathfrak{A},t,+,Qx\, \psi)$ in the game,
where $Q$ is an ordinary generalized quantifier. The
play of the game continues as follows.
\begin{enumerate}
\item
In the case $Q^{\mathfrak{A}}$ is empty, the play ends in the position $(\mathfrak{A},t,+, Qx\, \psi)$,
and we say that the player $\mathcal{A}$ \emph{does not survive} the play of the game.
Otherwise, the player $\mathcal{A}$ chooses a set $S\in Q^{\mathfrak{A}}$. The
background \emph{intuition} is that $\mathcal{A}$ claims
that $S$ is the set of \emph{exactly all} values for $x$ in $A$ that verify $\psi$. 
\item
Then the player $\mathcal{I}$ chooses either the set $S$ chosen by $\mathcal{A}$, or its complement $A\setminus S$.
\begin{enumerate}
\item
If $\mathcal{I}$ chooses $S$, then $\mathcal{I}$ also chooses an element $b\in S$, and the 
play of the game continues from the position $(\mathfrak{A},t[x/b],+,\psi)$. In this case the 
intuition is that the player $\mathcal{I}$ is opposing the claim that $b$ verifies $\psi$.
If $S=\emptyset$ and $\mathcal{I}$ chooses $S$,
the play of the game ends in the position $(\mathfrak{A},t,+, Qx\, \psi)$, and
the player $\mathcal{A}$ \emph{survives} the play of the game
\item
If $\mathcal{I}$ chooses $A\setminus S$, then $\mathcal{I}$ also chooses an element $b\in A\setminus S$.
The play of the game continues from the position $(\mathfrak{A},t[x/b],-,\psi)$.
The intuition is that 
the player $\mathcal{I}$ is opposing the claim that $b$ \emph{falsifies} $\psi$.
If $\mathcal{I}$ chooses $A\setminus S $ and 
$A\setminus S =\emptyset$,
the play of the game ends in the position $(\mathfrak{A},t,+, Qx\, \psi)$,
and the player $\mathcal{A}$ \emph{survives} the play of the game.
\end{enumerate}
\end{enumerate}
Assume we have reached a position $(\mathfrak{A},t,-,Qx\, \psi)$ in a play of the game,
where $Q$ is an ordinary generalized quantifier. The
play continues as follows.
\begin{enumerate}
\item
In the case $\overline{Q}^{\mathfrak{A}}$ is empty, the play of the game ends in the
position $(\mathfrak{A},t,-, Qx\, \psi)$,
and the player $\mathcal{A}$ does \emph{not survive} the play of the game.
Otherwise, the player $\mathcal{A}$ chooses a set $S\in \overline{Q}^{\mathfrak{A}}$. The \emph{intuition} is that
the player $\mathcal{A}$ claims
that $S$ is the set of \emph{exactly all} values for $x$
that verify $\psi$, while $S\not\in Q^{\mathfrak{A}}$.
\item
The player $\mathcal{I}$ then chooses either the set $S$ chosen by $\mathcal{A}$ or its complement $A\setminus S$.
\begin{enumerate}
\item
If $\mathcal{I}$ chooses $S$, then $\mathcal{I}$ also chooses an element $b\in S$, and the
play of the game continues from the position $(\mathfrak{A},t[x/b],+,\psi)$. In this case the 
intuition is that the player $\mathcal{I}$ is opposing the claim that $b$ verifies $\psi$.
If $\mathcal{I}$ chooses $S$ and $S=\emptyset$,
the play of the game ends ends in the position $(\mathfrak{A},t,-, Qx\, \psi)$,
and the player $\mathcal{A}$ survives the play of the game.
\item
If $\mathcal{I}$ chooses $A\setminus S$, then $\mathcal{I}$ also chooses an element $b\in A\setminus S$.
The game continues from the position $(\mathfrak{A},t[x/b],-,\psi)$.
The intuition is that 
the player $\mathcal{I}$ is opposing the claim that $b$ falsifies $\psi$.
If $\mathcal{I}$ chooses $A\setminus S$ and $A\setminus S=\emptyset$, 
the play ends in the position $(\mathfrak{A},t,-, Qx\, \psi)$, and the player
$\mathcal{A}$ survives the play of the game.
\end{enumerate}
\end{enumerate}
Assume we have reached a position $(\mathfrak{A},t,+, Mx\, \psi)$ in the game,
where $M$ is a minor quantifier. The
play of the game continues as follows.
\begin{enumerate}
\item
In the case $M^{\mathfrak{A}}$ is empty, the play ends in the position $(\mathfrak{A},t,+, Mx\, \psi)$,
and we say that the player $\mathcal{A}$ \emph{does not survive} the play of the game.
Otherwise, the player $\mathcal{A}$ chooses a pair $(S,T)\in M^{\mA}$.
The \emph{intuition} is that $S$ and $T$ are sets of values
for $x$, witnessing and falsifying $\psi$,
respectively. In other words, the player $\mathcal{A}$
claims that assignments in $t[x/S]$ satisfy $\psi$, while
assignments in $t[x/T]$ falsify $\psi$.
A further piece of the background intuition of course is that providing
such a pair $(S,T)$ is sufficient for the verification of $Mx\, \psi$.
\item
Then the player $\mathcal{I}$ chooses either the set $S$ or the set $T$.
\begin{enumerate}
\item
If $\mathcal{I}$ chooses $S$, then $\mathcal{I}$ also chooses an element $b\in S$, and the
play of the game continues from the position $(\mathfrak{A},t[x/b],+,\psi)$. In this case the 
intuition is that the player $\mathcal{I}$ is opposing the claim that $b$ verifies $\psi$.
If $S=\emptyset$ and $\mathcal{I}$ chooses $S$,
the game ends in the position $(\mathfrak{A},t,+,Mx\, \psi)$, and
the player $\mathcal{A}$ survives the play of the game.
\item
If $\mathcal{I}$ chooses $T$, then $\mathcal{I}$ also chooses an element $b\in T$.
The play of the game continues from the position $(\mathfrak{A},t[x/b],-,\psi)$.
The intuition is that 
the player $\mathcal{I}$ is opposing the claim that $b$ falsifies $\psi$.
If $T =\emptyset$ and $\mathcal{I}$ chooses $T$,
the game ends in the position $(\mathfrak{A},t,+,Mx\, \psi)$, and
the player $\mathcal{A}$ survives the play of the game.
\end{enumerate}
\end{enumerate}
Assume we have reached a position $(\mathfrak{A},t,-,Mx\, \psi)$ in a play of the game,
where $M$ is a minor quantifier. The
play continues as follows.
\begin{enumerate}
\item
In the case $\overline{M}^{\mathfrak{A}}$ is empty, the play of the game ends in the
position $(\mathfrak{A},t,-, Mx\, \psi)$,
and the player $\mathcal{A}$ does not survive the play of the game.
Otherwise, the player $\mathcal{A}$ chooses a
pair $(S,T)\in \overline{M}^{\mathfrak{A}}$. The intuition is that
that $S$ and $T$ are sets of values witnessing and falsifying $\psi$, respectively,
and supplying such a pair $(S,T)$ is enough to falsify $Mx\, \psi$.
\item
The player $\mathcal{I}$ then chooses either the set $S$ or the set $T$.
\begin{enumerate}
\item
If $\mathcal{I}$ chooses $S$, then $\mathcal{I}$ also chooses an element $b\in S$, and the
play of the game continues from the position $(\mathfrak{A},t[x/b],+,\psi)$. 
The intuition is that the player $\mathcal{I}$ is opposing the claim that $b$ verifies $\psi$.
If $\mathcal{I}$ chooses $S$ and $S = \emptyset$,
the play of the game ends in the position $(\mathfrak{A},t,-,Mx\, \psi)$,
and the player $\mathcal{A}$ survives the play of the game.
\item
If $\mathcal{I}$ chooses $T$, then $\mathcal{I}$ also chooses an element $b\in T$.
The game continues from the position $(\mathfrak{A},t[x/s],-,\psi)$.
The intuition is that 
the player $\mathcal{I}$ is opposing the claim that $b$ falsifies $\psi$.
If $\mathcal{I}$ chooses $T$ and $T = \emptyset$,
the play of the game ends in the position $(\mathfrak{A},t,-,Mx\, \psi)$,
and the player $\mathcal{A}$ survives the play of the game.
\end{enumerate}
\end{enumerate}
If $\psi$ is an atomic first-order formula,
and a position $(\mathfrak{A},t,+,\psi)$ is reached in a play of the game,
then $\mathcal{A}$ survives the play of the game if $\mathfrak{A},t\models_{\mathrm{FO}}\psi$.
If  $\mathfrak{A},t\not\models_{\mathrm{FO}}\psi$, then $\mathcal{A}$ does not
survive the play.
If a position $(\mathfrak{A},t,-,\chi)$ is reached, where $\chi$ is an atomic first-order formula,
then $\mathcal{A}$ survives the play of the game if $\mathfrak{A},t\not\models_{\mathrm{FO}}\chi$.
If $\mathfrak{A},t\models_{\mathrm{FO}}\chi$, then $\mathcal{A}$
does not survive the play.
%
%
If a position $(\mathfrak{A},t,+,\psi)$ or $(\mathfrak{A},t,-,\psi)$
is reached, where $\psi$ is a generalized atom, then $\mathcal{A}$ survives
the play. When a position with an atomic formula is reached, the 
play of the game ends.
Let $U$ and $V$ be teams with the same domain. Assume the domain
contains the free variables of $\varphi$.
A play of the game $G(\mathfrak{A},U,V,\varphi)$ is played by $\mathcal{A}$ and $\mathcal{I}$
such that $\mathcal{I}$ picks a beginning position $(\mathfrak{A},s,+,\varphi)$ or
$(\mathfrak{A},t,-,\varphi)$, where $s\in U$ and $t\in V$.
The play then proceeds according to the rules discussed above.
If $U = V = \emptyset$, and therefore $\mathcal{I}$ cannot choose a
beginning position, then $\mathcal{A}$ survives
the unique play of the game. In this case no \emph{end position} in the play
of the game is generated.
Let $F$ be a strategy of $\mathcal{A}$ for
the game $G(\mathfrak{A},U,V,\varphi)$; a strategy of $\mathcal{A}$ is
simply a function that provides a unique
choice for $\mathcal{A}$ in every possible position of the game that requires a choice.
The domain of $F$ is the set of positions in the game $G(\mathfrak{A},U,V,\varphi)$
that can be reached in some play of the game, and require a choice by $\mathcal{A}$.
In a position of the type $(\mathfrak{A},t,\#,Kx\, \psi)$, if $K^{\mA}$ is
empty, then the function $F$ is undefined on
the input $(\mathfrak{A},t,\#,Kx\, \psi)$.
Hence $F$ does not provide any move for $\mathcal{A}$ in
such a position. Here $K$ can be a minor quantifier or an ordinary generalized quantifier.

Let $S$ be the set of assignments $t$ such that some play, where $\mathcal{A}$
plays according to the strategy $F$, ends in the position $(\mathfrak{A},t,+,\chi)$.
The set $S$ is the \emph{team of positive final assignments} of the formula $\chi$ in the
game $G(\mathfrak{A},U,V,\varphi)$, when $\mathcal{A}$ plays according to $F$.
Similarly, let $T$ be the set of assignments $t$ such that some play, where $\mathcal{A}$
plays according to $F$, ends in the position $(\mathfrak{A},t,-,\chi)$.
The set $T$ is the \emph{team of negative final assignments} of the formula $\chi$ in the
game $G(\mathfrak{A},U,V,\varphi)$, when $\mathcal{A}$ plays according to $F$.
A \emph{survival strategy} of $\mathcal{A}$ in a game $G(\mathfrak{A},U,V,\varphi)$
is a strategy that guarantees, in every play of the game
where $\mathcal{A}$ follows $F$, a survival for $\mathcal{A}$.
Let $F$ be a survival strategy for $\mathcal{A}$ in $G(\mathfrak{A},U,V,\varphi)$.
Let $S(\chi)$ and $T(\chi)$ denote, respectively, the teams of positive and negative final assignments of
the generalized atom $\chi$ in the game $G(\mathfrak{A},U,V,\varphi)$, when $\mathcal{A}$ plays
according to $F$. The survival strategy $F$ is a \emph{uniform survival strategy} for $\mathcal{A}$, if for every
generalized atom $\chi$ in $\varphi$, we have $\mathfrak{A},\bigl(S(\chi),T(\chi)\bigr)\models\chi$.
Recall that all occurrences of a subformula in a formula $\varphi$
are considered to be distinct subformulae of $\varphi$.
Therefore, for example, if $\varphi$ is a generalized atom and a game
$G(\mathfrak{A},U,V,\varphi\vee\varphi)$  is played according to some strategy, the
teams of final assignments for the different instances of $\varphi$ may turn out different.
When $\mathfrak{A}$ is known from the context, we
may write $G(U,V,\psi)$ instead of $G(\mathfrak{A},U,V,\psi)$.
Also, we may write $(s,\#,\psi)$ instead of $(\mathfrak{A},s,\#,\psi)$.
\begin{theorem}\label{gamecorrespondence}
$\mathfrak{A}, (U,V)\models\varphi$ iff there exists a uniform survival
strategy for $\mathcal{A}$ in the game $G(\mathfrak{A},U,V,\varphi)$.
\end{theorem}
\begin{proof}
The claim is proved by induction on the structure of $\varphi$.
The case for atomic formulae is trivial.
Assume that $(U,V)\models\neg\psi$. Therefore 
$(V,U)\models\psi$. By the induction hypothesis, $\mathcal{A}$ has a
uniform survival strategy $F$ in $G(V,U,\psi)$.
The strategy $F$ provides a uniform survival strategy in $G(U,V,\neg\psi)$.
Assume that $\mathcal{A}$ has  a uniform survival strategy in $G(U,V,\neg\psi)$.
Therefore $\mathcal{A}$ has a uniform survival strategy in $G(V,U,\psi)$.
By the induction hypothesis, $(V,U)\models\psi$.
Therefore $(U,V)\models\neg\psi$.
Assume that $(U,V)\models \psi\vee\psi'$. 
Thus we have $(U_1,V)\models \psi$ and $(U_2,V)\models\psi'$
for some $U_1,U_2\subseteq U$ such that $U_1\cup U_2 = U$.
By the induction hypothesis, the player $\mathcal{A}$ has a
uniform survival strategy $F_1$ in
the game $G(U_1,V,\psi)$ and $F_2$ in the game $G(U_2,V,\psi')$.
Define a strategy $F$ for $G(U,V,\varphi\vee\psi)$ such that 
\[
 F\bigl(\, (s,+,\psi\vee\psi')\, \bigr)\ =\ \begin{cases}
    \{\psi,\psi'\} & \text{  if }\  s\in U_1\cap U_2\\
    \{\psi\} & \text{ if  }\ s\in U_1\setminus U_2\\
    \{\psi'\} & \text{ if  }\ s\in U_2\setminus U_1\\
 \end{cases}
\]
%
%
%
%
%
%
for each $s\in U$.
On other positions, $F$ agrees with $F_1$ or $F_2$, depending on whether
the input position contains a subformula of $\psi$ or $\psi'$.
%
%
The strategy $F$ gives the same final teams of
assignments as $F_1$ and $F_2$,
and therefore $F$ is a uniform survival strategy for $\mathcal{A}$ in $G(U,V,\psi\vee\psi')$.
Assume there exists a uniform survival strategy $F$ for $G(U,V,\psi\vee\psi')$.
Define $U_1\subseteq U$ to be the set of assignments $s\in U$ such that
$F\bigl(\, (s,+,\psi\vee\psi')\, \bigr)\, =\, \{\psi,\psi'\}$ or
$F\bigl(\, (s,+,\psi\vee\psi')\, \bigr)\, =\, \{\psi\}$.
Similarly, define $U_2\subseteq U$ to be the set of assignments $s\in U$ such that
$F\bigl(\, (s,+,\psi\vee\psi')\, \bigr)\, = \, \{\psi,\psi'\}$ or 
$F\bigl(\, (s,+,\psi\vee\psi')\, \bigr)\, =\, \{\psi'\}$.
Now, $F$ provides uniform survival strategies for $G(U_1,V,\psi)$
and for $G(U_2,V,\psi')$.
By the induction hypothesis,
$(U_1,V)\models \psi$ and $(U_2,V)\models\psi'$.
Since $U_1\cup\, U_2 = U$,
we have $(U,V)\models \psi\vee\psi'$.
We shall not discuss the argument for ordinary generalized
quantifiers, since the related details are essentially provided
by the argument for minor quantifiers.
Assume that $(U,V)\models Mx\, \psi$.
Thus there exists functions
$f:U\rightarrow M^{\mathfrak{A}}$ and $g:V\rightarrow \overline{M}^{\mathfrak{A}}$
such that
$$\bigl(U[x/f]\cup V[x/g] ,U[x/f\, '\, ]\cup V[\, x/g'\, ]\bigr)\models\psi.$$
By the induction hypothesis, there exists a uniform survival strategy $F$ in
$$G\bigl(U[x/f]\cup V[x/g] ,U[x/f\, '\, ]\cup V[\, x/g'\, ],\psi\bigr).$$
Extend the strategy $F$ to a strategy $F^+$ such that 
$F^+\bigl(\, (s,+,Mx\, \psi)\, \bigr)\, =\, f(s)$
for each $s\in U$ and $F^+\bigl(\, (t,-,Mx\, \psi)\, \bigr)\, =\, g(t)$
for each $t\in V$.
The strategy $F^+$ gives the same final teams of assignments as $F$,
and hence $F^+$ is a uniform survival strategy for $\mathcal{A}$ in $G(U,V,Mx\, \psi)$.
Assume $F$ is a uniform survival strategy in $G(U,V,Mx\, \psi)$.
%
%
%
%
Define the function $f:U\rightarrow M^{\mathfrak{A}}$ such that
$f(s)\ =\ F\bigl(\, (s,+,Mx\, \psi)\, \bigr)$
%
%
%
%
%
%
%
for all $s\in U$.
Define also the function $g:V\rightarrow \overline{M}^{\mathfrak{A}}$ such that
$g(s) = F\bigl(\, (s,-,Mx\, \psi)\, \bigr)$
%
%
%
%
%
%
%
for all $s\in V$. Now, $F$ provides a uniform survival strategy for 
$$G\Bigl(U[x/f]\cup V[x/g],U[x/f\, '\, ]\cup V[\, x/g\, '\, ],\psi\Bigr).$$
By the induction hypothesis,
$$\bigl(U[x/f]\cup V[x/g],U[x/f\, '\, ]\cup V[\, x/g\, '\,  ]\bigr)\models\psi.$$
Therefore $(U,V)\models Mx\, \psi$.
\end{proof}
\section{Interpreting dependence logic with double team semantics}\label{interpreting}
In this section we discuss a simple canonical way of
conservatively interpreting variants of dependence logic with
double team semantics.
We also address some issues concerning the interpretation of
dependence logic and  its variants.
Let $k$ be a positive integer and $T$ a non-empty set.
Let $R\subseteq T^k$ be a relation.
We say that $R$ \emph{is a partial function}, if
the following conditions hold.
\begin{enumerate}
\item
If $k = 1$, then $| R | \leq 1$.
\item
If $k>1$, and if we have $(s_1,...,s_{k-1},t)\in R$ and $(s_1,...,s_{k-1},u)\in R$,
then $t = u$.
\end{enumerate}
%
%
%
%
%

%
For each positive integer $k$, let $\mathcal{D}_k$ denote the generalized
quantifier that contains the triples $(A,R,S)$ such that
the following conditions hold.
\begin{enumerate}
\item
$A$ is a nonempty set.
\item
$R\subseteq A^k$ and $S\subseteq A^k$.
\item
$R$ is a partial function and $S = \emptyset$.
\end{enumerate}
Let $\Delta$ be the class $\{\ \mathcal{D}_k\ |\ k\in\mathbb{Z}_+\ \}$.
We next define a translation of formulae of dependence logic $\mathrm{D}$
into a logic with the minor quantifier $\exists^s$ and
generalized atoms
$$D_k(x_1,...,x_k;x_1,...,x_k)$$
for each $k\in\mathbb{Z}_+$; the semantics of the atom $D_k(x_1,...,x_k;x_1,...,x_k)$
is given by the generalized quantifier $\mathcal{D}_k\in \Delta$.
Define the following translation function $T$:
\begin{enumerate}
\item
If $\varphi$ is a first-order atom, then $T(\varphi) = \varphi$
and $T(\neg\varphi) = \neg\varphi$.
\item
$T\bigl(\depst(x_1,...,x_k)\bigr)\ =\ D_k(x_1,...,x_k;x_1,...,x_k)$ and
$T\bigl(\neg \depst(x_1,...,x_k)\bigr)\ =\ \neg D_k(x_1,...,x_k;x_1,...,x_k)$.
\item
$T(\varphi\vee\psi) = \bigl(T(\varphi)\vee T(\psi)\bigr)$.
\item
$T(\varphi\wedge\psi) = \neg \bigl(\neg T(\varphi)\vee \neg T(\psi)\bigr)$.
\item
$T(\exists z\, \varphi) = \exists^s z\, \varphi$.
\item
$T(\forall z\, \varphi) = \neg\, \exists^s z\, \neg\, T(\varphi)$.
\end{enumerate}
The following proposition is immediate.
\begin{proposition}\label{translationofdependence}
Let $\varphi$ be a formula of 
dependence logic. Then $\mA\models_{U}\varphi$ iff\, $\mA,(U,\emptyset)\models T(\varphi)$.
\end{proposition}
Obviously inclusion logic with strict semantics can be similarly translated into a logic
with double team semantics. A different class of generalized quantifiers is needed
in order to define the atoms that inclusion atoms translate to,
and the alternative disjunction $\vee^s$
defined in Section \ref{doubleteamsemantics} is used in the
target language. Also inclusion logic with lax
semantics can be analogously translated. Standard disjunctions
are used in the target language, and existential quantifiers
translate to the lax quantifier $\exists^l$.
\subsection{Interpreting different existential quantifiers}
It is interesting to note that neither the strict nor the lax existential
quantifier is the same as the
minor quantifier $M_{\exists}$ defined by the existential quantifier.
It is natural to consider the three different
existential quantifiers as \emph{epistemic variants} of each other.
Let us briefly discuss what this perspective means.
Consider the game-theoretic semantics for minor quantifiers.
Let $\varphi(x)$ be a first-order formula.
To show that the formula $\exists^s x\,\varphi(x)$
is true, the agent $\mathcal{A}$ simply has to find a single witness $b$ such that
the formula $\varphi(b)$ holds. It is enough that the agent
\emph{knows} one suitable witness $b$ for $\varphi(x)$.
Let $\exists^t$ denote the minor quantifier $M_{\exists}$, and
call it the \emph{total existential quantifier}.
Establishing that $\exists^t x\, \varphi(x)$ holds
is rather different from showing that $\exists^s x\, \varphi(x)$ holds.
This time it is not enough for the agent to know a single witness
for $\varphi(x)$. Instead, the agent has to be able to say, for each element $b$
in the domain of the model under investigation, whether $\varphi(b)$ holds or not.
Therefore the agent has to have an
\emph{epistemically complete understanding} of
which elements of the domain
%
%
satisfy $\varphi(x)$ and which do not.
Indeed, the strict existential quantifier seems to resemble the intuitive understanding of
\emph{ordinary} existence claims better than the total existential
quantifier. But of course $\exists^t$
may be more appropriate than the $\exists^s$ in some non-standard context.  
Establishing that $\exists^l x\, \varphi(x)$ is similar to showing that $\exists^s x\, \varphi(x)$,
but here the agent can provide more than one witness to be taken into
account in the rest of the semantic game.
In the light of Propositions \ref{minorisnormal} and \ref{first-ordercorrespondence},
the three existential quantifiers are interchangeable in the context of
ordinary first-order logic. But it is possible to conceive
natural non-classical logics---possibly dealing with epistemic
considerations, and not necessarily involving generalized
atoms---where different epistemic modes of
existential quantification make a crucial difference.
And obviously it is rather trivial to invent ad hoc atoms $A(x;x)$ such that,
say, $\exists^s x\, A(x;x)$ and $\exists^t x\, A(x;x)$
are not equivalent.
Let $T$ denote the trivial generalized quantifier of the type $(1)$
defined such that $\mA\models_{\mathrm{FO}} Tx\, \varphi$ always holds.
In the double team framework, the statement
$\mA, (\{\emptyset\}, \emptyset) \models  Tx\, P(x)$
means that the player $\mathcal{A}$ can \emph{classify}
all elements $b\in\mathit{Dom}(\mA)$ according to whether $P(b)$
holds or not, i.e., $\mathcal{A}$ can point out exactly the set of values $b$
such that $P(b)$. The statement $\mA, (\{\emptyset\}, \emptyset) \models  \exists^t\, x\, P(x)$
means that the player $\mathcal{A}$ can classify
all elements $b$ of the domain of $\mA$ according to whether $P(b)$
holds or not, and the set of values such that $P(b)$ holds, is nonempty.
These are constructive statements that clearly \emph{differ from the ordinary reading}
of the generalized quantifiers $T$ and $\exists$.
The notion of a minor quantifier provides a novel way of 
generalizing the notion of a generalized quantifier by
providing a fine-grained picture of constructive issues
related to verification of quantified formulae. A possible future research direction could
include considering semantic games, where choosing (sets of) witnesses would be
associated with a \emph{cost}, and  of course the player(s) involved would have
limited amounts of resources with which to meet the costs.
For example, in a very simple case, each element of the domain of a model could be associated with a unit cost.
Such games could help in the analysis of \emph{proving or verifying
theorems with limited resources}.
A tentative approach to first-order logic with a \emph{resource consicious semantics} is
given in \cite{kuusisto6}.
For the sake of entertainment, let us consider the following (naive)
thought experiment.
Flip a coin once in a half a minute period.
Flip the coin again in the next fifteen seconds.
In the next $7.5$ seconds, flip the coin again the third time.
Keep doing this, always halving the duration of the previous period.
Do this so that for at least the last third of each period,
the coin is in rest, so that no angular momentum is preserved
from one period to another.
Keep doing this for one minute, and after that, do
nothing for at least three minutes. Under sufficiently naive
and idealized classical assumptions, this
experiment can be carried out. It is then a
rather puzzling question what the state 
of the coin is when two minutes has passed.
Is it heads or tails? Is it something else?
A truly annoying state!
Of course we do not care about Planck's time 
and all that here. This is an entirely classical 
paradox. There are of course several ways of 
adding constraints that make the experiment impossible.
For example, we can stipulate that each
flipping of the coin consumes at least some unit amount $r$ 
of \emph{resources}, and the amount
of available resources is not infinite.
\subsection{Observations concerning atoms}
Above we translated dependence atoms $\depst(x_1,...,x_k)$ into
atomic expressions $D_k(x_1,...,x_k;x_1,...,x_k)$. This
creates an unnecessary syntactic complication: it seems rather pointless to
write $x_1,...,x_k$ twice. We can of course 
avoid such complications in similar translations 
by simply allowing for syntactic atomic expressions $A(x_1,...,x_k)$,
whose semantics is defined by a generalized quantifier of the type $(k,k)$,
and more generally, atoms $B(\overline{x}_1,...,\overline{x}_k)$ 
defined by quantifiers of the type
$(i_1,...,i_k,i_1,...,i_k)$.
Atomic expressions with the simple syntactic form
$B(\overline{x}_1,...,\overline{x}_n)$, where
the symbol $;$ does not appear, may perhaps be more appropriate
for example from the point of view of issues in natural language analysis.
%

%
%
%

%
Let $(Q,P)$ be a pair of generalized quantifiers of type $(i_1,...,i_k)$.
Consider atomic expressions of the type
$B(\overline{x}_1,...,\overline{x}_k),$
where each tuple $\overline{x}_j$ is of the length $i_j$.
Extend the double team semantics such
that $\mA,(U,V)\models B(\overline{x}_1,...,\overline{x}_k)$ iff
\begin{align*}
& \bigl(\mathrm{Rel}(\mathfrak{A},U,\overline{x}_1),...,
\mathrm{Rel}(\mathfrak{A},U,\overline{x}_k)\bigr)\in Q^{\mathfrak{A}}\\
\text{ and } & \\
& \bigl(\mathrm{Rel}(\mathfrak{A},V,\overline{x}_{1}),...,
\mathrm{Rel}(\mathfrak{A},V,\overline{x}_{k})\bigr)\in P^{\empty\, \mathfrak{A}}.
\end{align*}
If $P = \overline{Q}$, we call the atom 
defined by $(Q,P)$ a \emph{symmetric atom}.
It is interesting to note that above
it would \emph{not} have been possible to translate
atoms $\depst(x_1,...,x_k)$ to
symmetric atoms $B(x_1,...,x_k)$.
The truth definitions of the
dependence atom $\depst(x_1,...,x_k)$
and its negated counterpart $\neg \depst(x_1,...,x_k)$
are not related in a way that
would lead to the required symmetry.
Currently, there does not seem to be an account in the dependence logic
literature that thoroughly analyzes issues related to the choice of the definition
$\mA\models_{U}\neg\depst(x_1,...,x_k)$ iff $U = \emptyset$.
It is well known that dependence logic is 
downwards closed, i.e., if $\mA\models_U\varphi$ and $V\subseteq U$,
then $\mA\models_V\varphi$.
The definition
$\mA\models_{U}\neg\depst(x_1,...,x_k)\ \Leftrightarrow\
\mA\not\models_{U}\, \depst(x_1,...,x_k)$ 
would lead to a logic that is not downwards closed.
Downwards closure is a natural
intuitive property of dependence logic.
Downwards closure
reflects the background \emph{intuition} that a team satisfies a
formula if all assignments in it satisfy the formula.\footnote{
We of course recall that this is nothing more than the background intuition.}
%
%
%
With the semantics $\mA\models_U\neg \depst(x_1,...,x_k)\ \Leftrightarrow\ U =\emptyset$
for negated dependence atoms, dependence logic is downwards closed, but
still this choice of definition may seem intuitively somewhat arbitrary. At least
the definition calls for further reflection.
In inclusion logic \cite{galliani}, negated atoms are not
allowed, and thereby no analogous problem of interpretation arises.
%
%
%
But the possibility of negating atomic formulae---a
syntactically natural feature---is compromised.\footnote{Of 
course it should be kept in mind here that
negation in the context of team semantics is
not the contradictory negation on the level of teams.}
We shall not attempt to analyze the issue concerning negated atoms further,
but we wish to point out that the double team framework can perhaps help in
advancing the interpretation
of formalisms in the family of dependence logic, for
at least the following three reasons.
Firstly, the double team semantics provides a
\emph{general} framework for interpreting various different variants of dependence logic.
How exactly generality is related to elucidation is an interesting question itself,
and obviously we shall not attempt to analyze this issue in this article,
but a general framework does offer a setting for interpreting and \emph{comparing}
different systems embeddable in the framework. 
For example, we have above given possible interpretations for the strict and
lax existential quantifiers, and also observed that neither of these quantifiers is
the same as the minor quantifier defined by the
ordinary existential quantifier.
Secondly, the double team semantics has obvious
\emph{symmetric} duality properties concerning the interpretation
of negation.\footnote{Issues
related to different modes of negation seem to lead to notable issues concerning
the intuitive interpretation of formulae in various systems
based on team semantics.
Related issues are likely to arise also in the framework presented in
the current article. A  rather obvious framework for the analysis
of different negations would involve systems based on
sets of teams, or possibly pairs of sets of teams, or
something similar. Such a framework
would allow for a more direct access to
different uses of the contradictory negation on
different levels of type theory.}
%
%
%
%
%
%
How exactly symmetries 
lead to elucidation is an interesting
question that we shall not attempt to analyze in this article.
But whatever their explanatory power may be,
at least symmetric duality properties have an obvious
mathematical appeal.\footnote{Symmetries, as well as
presentations in a more general well understood framework,
seem to play an important role in \emph{explanations}
in the mathematical and analytic realms.
Of course also for example analogies play a role.}
%
%
%
%
%

%
Finally, the double team semantics has a very natural
game-theoretic counterpart.
A game-theoretic semantics can---at least in some
reasonable sense---be seen as \emph{fundamental}
in relation to other approaches, because it provides an
\emph{action based} account of the meaning of formulae.
On the face of it, semantic games can seem rather 
far removed from contexts where natural
language is learned, but 
it is not difficult to invent action-based scenarios
described by semantic games, where the
meaning of the words \emph{all} and \emph{exists}
%
%
becomes at least elucidated to an agent.
Tarski's semantics for first-order logic essentially gives simply a \emph{translation} of symbols into
their natural language counterparts.\footnote{Of course Tarski's
semantics also ties truth of first-order formulae to the notion of a model,
and additionally provides an inductive method for computing 
truth values of formulae based on the truth values of the atoms.}
This resembles translating a
language into another. An interpreter has to be familiar with the
target language in order to understand the truth definition.
The situation seems different in the context of
action-based truth definitions.\footnote{Of course
game-theoretic truth definitions are still usually \emph{described} in natural language.}
%
%
%
In fact, it seems to even make reasonable sense to  consider action-based approaches
in attempts to define semantics for natural languages.\footnote{
%
%
%
%
%
%
A person's first language is learned via action-based situations.
But it seems appealing to
think that logical understanding is also,
up to some extent, hard-wired in the brain or
physically somehow forced. For example it is easy to conceive a person
learning the meaning of the word \emph{all} in situations involving rather small
collections of objects. It is interesting that the person still 
learns the correct meaning of the word \emph{all}, instead of associating the
word with some exotic quantifier that is equivalent to $\forall$ in models of size
less than, say, $400$, or $2^{1000}$. There seems to be a natural cognitive and inductive generalization process
involved here.}
Action-based language acquisition is discussed for
example in \cite{steels}.
Wittgenstein's \emph{language games}, described in \cite{wittgenstein},
are a classical example of related considerations.
Of course the claim about fundamentality of action-based
approaches to semantics is highly debatable, and obviously
we do not wish to engage in that debate here.
We simply wish to point out that the game-theoretic
counterpart of double team semantics
does provide a description of an action-based approach to the meaning
of generalized quantifiers and atoms. 
In this context it is worth noting that the game-theoretic semantics is also a
novel canonical semantics for ordinary extensions
of first-order logic with generalized quantifiers---extensions
that do not involve generalized atoms.
\section{Complexity of $\mathrm{DC^2}$}\label{complexity}
\subsection{The logic $\mathrm{DC}^2$}
In this section we define the logic $\mathrm{DC}^2$. This logic
extends both ordinary two variable dependence logic $\mathrm{D}^2$
and two-variable logic with counting $\mathrm{FOC}^2$,
as we shall see. 
%

%
%
%
%
%
%
%
%
%
%
%
%
%
%
%
%
%
%
%
%
%
%
%
%

%
Let $k$ be a positive integer.
Define the classes
$$\mathcal{E}\, :=\, \{\, (A,B,\emptyset)\ |\
A\text{ is a non-empty set and }B\subseteq A\text{ satisfies }| B |\geq k\ \}$$
and
$$\mathcal{F}\, :=\, \{\, (A,\emptyset,B)\ |\ A\text{ is a non-empty set, }
B\subseteq A\text{ and }| A\setminus B | < k\ \}.$$
The pair $(\mathcal{E},\mathcal{F})$ defines the
\emph{minor counting quantifier} $\exists^{\mathit{\geq k}}$.
Notice that $\exists^{\mathit{\geq k}}$ is a minor
of the generalized quantifier
$\{\ (A,B)\ |\ A\not=\emptyset,\ |B|\geq k\ \}$.
%
%

%
Let $\tau$ be a relational vocabulary consisting of
the union of a countably infinite set of
unary relation symbols and a countably infinite set of binary relation symbols.
Fix two distinct first-order variable symbols $x$ and $y$.
Define $\mathcal{A}(\tau)$ to be the smallest set $T$ such that
the following conditions hold.
\begin{enumerate}
\item
If $P\in\tau$  and $z\in\{x,y\}$, then $P(z)\in T$.
\item
If $R\in\tau$, and $z,z'\in\{x,y\}$, then $R(z,z')\in T$.
\item
If $z,z'\in\{x,y\}$, then $z = z'\in T$.
\end{enumerate}
Define \emph{two-variable first-order logic with counting} ($\mathrm{FOC}^2$)
to be the smallest set $T$ such that the following conditions are satisfied.
\begin{enumerate}
\item
$\mathcal{A}(\tau)\subseteq T$.
\item
If $\varphi\in T$, then $\neg\varphi\in T$.
\item
If $\varphi,\psi\in T$, then $(\varphi\vee\psi)\in T$.
%
%
\item
If $\varphi\in T$, $z\in\{x,y\}$, and 
$k$ is a positive integer, then $\exists^{\geq k} z\, \varphi\in T$.
\end{enumerate}
Here $\exists^{\geq k}$ denotes
the minor quantifier $\bigl(\mathcal{E},\mathcal{F}\bigr)$.
The syntax of $\mathrm{FOC}^2$ contains only first-order atoms,
and in the light of Propositions
\ref{minorisnormal} and \ref{first-ordercorrespondence},
it makes no difference whether we use ordinary
Tarskian semantics or double team semantics in the interpretation
of $\mathrm{FOC}^2$-formulae; if $\varphi$ is a
formula of $\mathrm{FOC}^2$, and $\varphi'$
denotes the formula obtained from $\varphi$ by
replacing each symbol $\exists^{\geq k}$ by a symbol
that denotes the corresponding ordinary 
generalized quantifier, 
then $\mA,s\models_{\mathrm{FO}}\varphi'$
iff $\mA,\bigl(\{s\},\emptyset\bigr)\models\varphi$.
Define $\mathcal{A}^+(\tau)$ to be the smallest set $T$ such that
the following conditions hold.
\begin{enumerate}
\item
If $\varphi\in\mathcal{A}(\tau)$, then $\varphi\in T$.
%
%
\item
If $z,z'\in\{x,y\}$, then $\depst(z,z')\in T$ and $\depst(z)\in T$.
Here we assume that $z\not= z'$, i.e., $z$ and $z'$
are different variable symbols.
\end{enumerate}
The set of formulae of $\mathrm{DC}^2$ is
the smallest set $T$ such that the following conditions hold.
\begin{enumerate}
\item
$\mathcal{A}^+(\tau)\subseteq T$.
\item
If $\varphi\in T$, then $\neg\varphi\in T$.
\item
If $\varphi,\psi\in T$, then $(\varphi\vee\psi)\in T$.
\item
If $\varphi\in T$ and $z\in \{x,y\}$, then $\exists^{\mathit{s}} z\, \varphi\in T$.
\item
If $\varphi\in T$, $z\in\{x,y\}$, and 
$k$ is a positive integer, then $\exists^{\geq k} z\, \varphi\in T$.
\end{enumerate}
%
%
%
%
%
%

%
%

%
Let $z,z'\in\{x,y\}$ be variables.
The semantics of the atom $\depst(z)$ in is
defined in $\mathrm{DC}^2$ such that $\mA,(U,V)\models\, \depst(z)$ iff
$\mA,(U,V)\models D_1(z;z)$. Similarly,
$\mA,(U,V)\models\, \depst(z,z')$ iff
$\mA,(U,V)\models D_2\bigl((z,z');(z,z')\bigr)$.
The following lemma is trivial.
\begin{lemma}\label{downwards}
Let $(U,V)$ and $(S,T)$ be a double teams such that $S\subseteq U$ and $T\subseteq V$.
Let $\varphi\in\mathcal{A}^+(\tau)$ be any atomic formula of $\mathrm{DC}^2$.
If $(U,V)\models\, \varphi$, then $(S,T)\models\, \varphi$.
%
%
%
%
\end{lemma}
Obviously $\mathrm{FOC}^2$ is contained in $\mathrm{DC}^2$,
but also $\mathrm{D}^2$ is essentially contained in $\mathrm{DC}^2$
via the translation $T$ defined in
Section \ref{interpreting} (see Proposition \ref{translationofdependence}).
%

%
%

%
%
%
%
We have somewhat blindly copied the atoms of $\mathrm{D}$
into $\mathrm{DC}^2$; it is an interesting question what these atoms
\emph{exactly mean} in $\mathrm{DC}^2$,
and what other kinds of atoms and quantifiers
should be considered. We leave such questions 
for the future. Our objective in the rest of the current article is simply to 
show how the double team semantics nicely
facilitates the $\mathrm{NEXPTIME}$-completeness proof of
the logic $\mathrm{DC}^2$, and other sufficiently similar logics.
%

%
%

%
\subsection{$\mathrm{DC}^2$ is 
NEXPTIME-complete}
An input to the satisfiability or finite satisfiabilily problem of $\mathrm{DC}^2$ is any
sentence $\varphi$ of $\mathrm{DC}^2$.
Note that the set of non-logical symbols of $\varphi$
is limited to unary and binary relation symbols only.
The satisfiability problem asks whether there
exists a model $\mA$ such that $\mA,\bigl(\{\emptyset\},\emptyset\, \bigr)\models\varphi$,
while the finite satisfiability problem asks whether there exists a
\emph{finite} model $\mB$ such that $\mB,\bigl(\{\emptyset\},\emptyset\, \bigr)\models\varphi$.
An input to the satisfiability or finite satisfiabilily problem of $\mathrm{FOC}^2$ is any
sentence $\varphi$ of $\mathrm{FOC}^2$;
the set of non-logical symbols of $\varphi$
is limited to unary and binary relation symbols only.
The satisfiability problem asks whether there
exists a model $\mA$ such that $\mA\models_{\mathrm{FO}}\varphi$,
while the finite satisfiability asks whether there exists a
finite model $\mB$ such that $\mB\models_{\mathrm{FO}}\varphi$.
Below we show that the
satisfiability and finite satisfiability problems of $\mathrm{DC}^2$ are $\mathrm{NEXPTIME}$-complete.
Our proof uses the fact
that the satisfiability and finite satisfiability problems of $\mathrm{FOC}^2$
are $\mathrm{NEXPTIME}$-complete (see \cite{pratthartmann}).
We translate $\mathrm{DC}^2$ formulae into equisatisfiable formulae of
$\mathrm{FOC}^2$ with a polynomial cost in the formula length;
the translation can be carried out in logarithmic space.
A formula $\varphi$ translates to a formula
$$\varphi^*\ :=\ \psi_{\mathit{initial}}\
\wedge\ \bigwedge\limits_{\chi\ \in\ \mathrm{SUB}_{\varphi}}\psi_{\chi}.$$
%
%
%
%
Each conjunct $\psi_{\chi}$ contains two
fresh relation symbols $S_{\chi}$ and $T_{\chi}$.
Intuitively, the pair $(S_{\chi},T_{\chi})$
encodes the double team $(U_{\chi},V_{\chi})$ that
satisfies $\chi$, when $\varphi$ is evaluated in a model
where $\varphi$ holds. 
If $\chi$ is not an atom, the formula $\psi_{\chi}$ also contains auxiliary formulae
that describe how double teams
evolve, when $\varphi$ is evaluated.
For example, if $\chi = \exists^{s} x\, \alpha$, then
$\psi_{\chi}$ describes how the double team $(U_{\chi},V_{\chi})$
gives rise to a double team $(U_{\alpha},V_{\alpha})$ that satisfies $\alpha$.
%

%
%
%
%
In addition to relation symbols $S_{\chi}, T_{\chi}$ corresponding to
double teams, further fresh variable symbols are used in $\psi_{\chi}$
when $\chi$ is a formula whose main connective is a quantifier.
The fresh symbols $E_{\alpha}^{\mathit{Uf}},\,  
E_{\alpha}^{\mathit{Vg'}}$ correspond to the 
teams $U[z/f],\, V[z/g']$ needed in the
truth definition of quantified formulae.\footnote{It turns out that
there is no need for symbols $E_{\alpha}^{\mathit{Uf'}},E^{\Vg}$.
In fact, even the symbols $E_{\alpha}^{\mathit{Uf}}$ and  
$E_{\alpha}^{\mathit{Vg'}}$ could  be eliminated,
but we keep them for the sake of presentation.
The reader may consider further minor quantifiers for 
which the proofs in this section go trough. In doing so,
using  extra predicates  $E_{\alpha}^{\mathit{Uf}},\,  
E_{\alpha}^{\mathit{Vg'}},\, E_{\alpha}^{\mathit{Uf'}},E^{\Vg}$
may help.}
The logic $\mathrm{FOC}^2$ uses only two variables,
and this creates some obstacles that need to be overcome
when writing the formulae $\psi_{\chi}$.
Due to the expressivity limitations of $\mathrm{FOC}^2$, we
need to control the evaluation of double teams $(U_{\chi},V_{\chi})$.
For example, if $\chi = \exists^{s} x\, \alpha$ and the
domain of $U_{\chi}$ contains $x$, then we need to 
ensure that the new values of $x$ in $U[ x/f ]$ are 
in a sense \emph{independent} of the old values of $x$ in $U$;
the related definitions are formally discussed below.
Lemma \ref{downwardslemma} ensures that
we can indeed control the evaluation of the teams $(U_{\chi},V_{\chi})$
in the desired way, and therefore the two-variable logic $\mathrm{FOC}^2$
is sufficiently expressive for our purposes.
While formulae $\psi_{\chi}$ describe double teams
corresponding to subformulae of $\varphi$, the formula
$\psi_{\mathit{initial}}$ simply sets the stage by
asserting that the team satisfying $\varphi$ itself
corresponds to the team $\bigl(\{\emptyset\},\emptyset\, \bigr)$.
We are now ready for the formal details of the proof that
the logic $\mathrm{DC}^2$ is complete for $\mathrm{NEXPTIME}$.
We begin by some auxiliary definitions and the auxiliary
Lemmata \ref{downwardslemma} and \ref{emptysetlemma}.
We then formally define the conjuncts of $\varphi^*$
and show that $\varphi$ and $\varphi^*$ are equisatisfiable.
Let $U$ be a team for a model $\mA$. Let $A$ be the domain of $\mA$.
Assume the domain of $U$ contains the variable $x$.
Let $s,t\in U$ be assignments such that $s(z) = t(z)$
for all $z\in\mathit{Dom}(U)\setminus\{x\}$.
Then $t$ is called an
\emph{$x$-variant of $s$} (in $U$).
Note that $s$ is an $x$-variant of itself.
Let $M$ be a minor quantifier, and
let $N\in\{\, M^{\mA},\overline{M}^{\mA}\, \}$.
Let $f:U\rightarrow N$ be a function.
Assume that we have we have $f(s) = f(t)$
for all valuations $s,t\in U$ such
that $t$ is an $x$-variant of $s$.
Then we say that $f$ is 
\emph{$x$-independent}.
Let $g:U\rightarrow N$ be a
function. Assume $g_0:U\rightarrow N$ is an
$x$-independent function such that for each $s\in U$,
there exists an $x$-variant $t\in U$ of $s$ such that
$g_0(s) = g(t)$. Then $g_0$ is an
\emph{$x$-independent minor of $g$}.
Let $U$ be a team with the domain $\{x,y\}$
and for a model $\mA$,
where $x$ and $y$ are the variables used in $\mathrm{DC}^2$ and $\mathrm{FOC}^2$.
We let $\mathit{Rel}(U)$
denote the relation $\mathit{Rel}\bigl(U,\mA,(x,y)\bigr)$, as
opposed to $\mathit{Rel}\bigl(U,\mA,(y,x)\bigr)$. This means that
we in a sense nominate $x$ as the first variable and $y$ as the second
one. This convention will simplify the notation below.
If $U$ is a team with the domain $\{z\}$, where $z\in\{x,y\}$,
then we let $\mathit{Rel}(U)$ denote $\mathit{Rel}\bigl(U,\mA,z\bigr)$.
%

%
%

%
\begin{lemma}\label{downwardslemma}
Let $\psi$ be a formula of\, $\mathrm{DC}^2$.
Let $M\in \{\exists^{\mathit{s}},\exists^{\geq k}\}$, where $k$ is a positive
integer. Let $z\in\{x,y\}$ be a variable.
Let $f:U\rightarrow M^{\mA}$ and $g:V\rightarrow \overline{M}^{\mA}$
be functions, and let $f_0$ and $g_0$ be $z$-independent minors of $f$ and $g$,
respectively. If
$$\mA,\bigl(U[z/f]\cup V[z/g],\, U[z/f\, '\, ]\cup V[\, z/g\, '\,  ]\bigr)\models\psi,$$
then
$$\mA,\bigl(U[z/f_0]\cup V[z/g_0],\, U[z/f_0\, '\, ]\cup V[\, z/g_0\, '\,  ]\bigr)\models\psi.$$
\end{lemma}
\begin{proof}
%
%
Assume that
\begin{align}\label{eqn1}
\bigl(U[z/f]\cup V[z/g],\, U[z/f\, '\, ]\cup V[\, z/g\, '\,  ]\bigr)\models\psi.
\end{align}
It is clear that $U[z/f_0]\subseteq U[z/f]$ and $V[z/g_0\, '\, ]\subseteq V[\, z/g\, '\, ]$.
It is also clear that $V[z/g_0\, ] =  V[\, z/g\, ] =
U[x/f_0\, '\, ] =  U[x/f\, '\, ] = \emptyset$.
Therefore
\begin{align}\label{eqn2}
U[z/f_0]\cup V[z/g_0]\subseteq U[z/f]\cup V[z/g]
\end{align}
and
\begin{align}\label{eqn3}
U[z/f_0\, '\, ]\cup V[\, z/g_0\, '\,  ] \subseteq U[z/f\, '\, ]\cup V[\, z/g\, '\,  ].
\end{align}  
We define a strategy for the player $\mathcal{A}$ in the game
$$G^*\ :=\ G\bigl(\mA,U[z/f_0]\cup V[z/g_0],U[z/f_0\, '\, ]\cup V[\, z/g_0\, '\,  ],\psi\bigr).$$
Due to Equation \ref{eqn1}, player $\mathcal{A}$ has a uniform survival strategy $F$ in the game
$$G\ :=\ G\bigl(\mA,U[z/f]\cup V[z/g],U[z/f\, '\, ]\cup V[\, z/g\, '\,  ],\psi\bigr).$$
Due to Equations \ref{eqn2} and \ref{eqn3}, the strategy $F$ can be
canonically restricted to a
strategy $H$ for the game $G^*$.
We need to show that $H$ is a \emph{uniform survival strategy} for $\mathcal{A}$ in $G^*$.
Since $H$ is a restriction of the uniform survival strategy $F$, the player $\mathcal{A}$
survives each play of the game $G^*$ played according to $H$.
To see that $H$ is a uniform survival strategy, consider the sets $S^*(\chi)$ and $T^*(\chi)$
of positive and negative final assignments for an atomic subformula $\chi$ of $\psi$,
when $\mathcal{A}$ follows $H$ in $G^*$.
Let $S(\chi)$ and $T(\chi)$ be the corresponding sets in the game $G$,
when $\mathcal{A}$ follows $F$.
It is clear that $S^*(\chi)\subseteq S(\chi)$ and $T^*(\chi)\subseteq T(\chi)$.
Due to Equation \ref{eqn1}, we have $\bigl(S(\chi),T(\chi)\bigr)\models\chi$.
By Lemma \ref{downwards}, we have $\bigl(S^*(\chi),T^*(\chi)\bigr)\models\chi$,
and therefore $H$ is a uniform survival strategy
for $\mathcal{A}$ in the game $G^*$.
\end{proof}
It turns out that we do not actually need Lemma \ref{downwardslemma} in
full generality. The essential part of the Lemma is that
functions $f:U\rightarrow{\exists^s}^{\text{ }\mA}$ can be assumed 
to be $z$-independent; see the proof of Lemma \ref{lemmafour}
for further details. 
Let $\psi$ be a sentence of $\mathrm{DC}^2$.
Define $\mathit{Dom}_{\psi}(\psi) = \emptyset$.
Assume then that we have defined $\mathit{Dom}_{\psi}(\chi)$ for $\chi\in\mathrm{SUB}_{\psi}$.
\begin{enumerate}
\item
If $\chi = \exists^{\geq k} x\, \alpha$ or $\chi = \exists^{s} x\, \alpha$, define
$\mathit{Dom}_{\psi}(\alpha) = \mathit{Dom}_{\psi}(\chi)\cup\{x\}$.
\item
If $\chi = \exists^{\geq k} y\, \alpha$ or $\chi = \exists^{s} y\, \alpha$, define
$\mathit{Dom}_{\psi}(\alpha) = \mathit{Dom}_{\psi}(\chi)\cup\{y\}$.
\item
If $\chi = \chi_1\vee\chi_2$, define $\mathit{Dom}_{\psi}(\chi_1)
= \mathit{Dom}_{\psi}(\chi_2) = \mathit{Dom}_{\psi}(\chi)$.
\item
If $\chi = \neg\alpha$, define $\mathit{Dom}_{\psi}(\alpha)= \mathit{Dom}_{\psi}(\chi)$.
\end{enumerate}
\begin{lemma}\label{emptysetlemma}
Let $\psi$ be a sentence of\, $\mathrm{DC}^2$
and $U$ a team with exactly one assignment.
Then $\mA,\bigl(U,\emptyset\bigr)\models\psi$ iff
$\mA,\bigl(\{\emptyset\},\emptyset\, \bigr)\models\psi$.
\end{lemma}
\begin{proof}
%
%
Let $s$ be the unique assignment in $U$.
Assume that $\mA,(U,\emptyset)\models\psi$.
The player $\mathcal{A}$ has a uniform survival strategy $F$ in the game
$G\bigl(\{s\},\emptyset,\psi\bigr)$. (Recall that we
may write $G\bigl(U,\emptyset,\psi\bigr)$ instead of $G\bigl(\mA,U,\emptyset,\psi\bigr)$.)

Now, let $F'$ be the strategy for $G(\{\emptyset\},\emptyset,\psi)$,
where $\mathcal{A}$ canonically copies the moves determined by $F$
in $G(\{ s\},\emptyset,\psi)$. This means that for
each position $(\mA,t,\#, \alpha)$ in $G(\{\emptyset\},\emptyset,\psi)$,
we define $F'(\mA,t,\#, \alpha) := F(\mA, t',\#, \alpha)$,
where $t = t' \upharpoonright \mathit{Dom}_{\psi}(\alpha)$,
i.e., $t$ is the restriction of $t'$ to 
the set $\mathit{Dom}_{\psi}(\alpha)$.
It is easy to show that $F'$ is well-defined.
Let $\chi$ be an arbitrary atom of $\psi$,
and let $S(\chi)$ and $T(\chi)$ be the sets of
positive and negative final assignments for $\chi$ in
the game $G\bigl(\{s\},\emptyset,\psi\bigr)$, when $\mathcal{A}$
follows the strategy $F$. Recalling that $\psi$ is a
sentence, it is easy to see that the teams of positive and negative final assignments
$S^{*}(\chi)$ and $T^{*}(\chi)$
that arise in $G(\{\emptyset\},\emptyset,\psi)$ when $\mathcal{A}$ follows $F'$, are exactly
the same teams as those that arise in $G(\{ s\},\emptyset,\psi)$ when $\mathcal{A}$
follows $F$, i.e., $S^{*}(\chi) = S^{}(\chi)$ and $T^{*}(\chi) = T^{}(\chi)$.
Thus $\mA,(\{\emptyset\},\emptyset\, )\models\psi$.
The converse implication is similar.
Assume that $\mA,(\{\emptyset\},\emptyset)\models\psi$.
Thus $\mathcal{A}$ has a uniform survival strategy $H$ in
the game $G(\{\emptyset\},\emptyset,\psi)$.
let $H'$ be the strategy for $G(\{\, s \},\emptyset,\psi)$,
where $\mathcal{A}$ canonically copies the moves determined by $H$
in $G(\{ \emptyset \},\emptyset,\psi)$. This means that for
each position $(\mA,t,\#, \alpha)$ in $G(\{\, s \},\emptyset,\psi)$,
we define $H'(\mA,t,\#, \alpha) := H(\mA, t',\#, \alpha)$,
where $t' = t\upharpoonright \mathit{Dom}_{\psi}(\alpha)$.
Let $\chi$ be an arbitrary atom of $\psi$,
and let $S(\chi)$ and $T(\chi)$ be the sets of
positive and negative final assignments for $\chi$ in
the game $G\bigl(\{\emptyset\},\emptyset,\psi\bigr)$, when $\mathcal{A}$
follows the strategy $H$.
It is easy to see that the teams of positive and negative final assignments
$S^{*}(\chi)$ and $T^{*}(\chi)$
that arise in $G(\{ s \},\emptyset,\psi)$ when $\mathcal{A}$ follows $H'$, are exactly
the same teams as those that arise in $G(\{ \emptyset \},\emptyset,\psi)$ when $\mathcal{A}$
follows $H$, i.e., $S^{*}(\chi) = S^{}(\chi)$ and $T^{*}(\chi) = T^{}(\chi)$.
Thus $\mA,(\{ s \},\emptyset\, )\models\psi$.
\end{proof}
Now \emph{fix} a sentence $\varphi$ of $\mathrm{DC}^2$.
Our next aim is to define the $\mathrm{FOC}^2$ sentence $\varphi^*$
and then prove that $\varphi$ and $\varphi^*$ are equisatisfiable.
Let $\psi$ be an arbitrary subformula of $\varphi$.
Having fixed the sentence $\varphi$, we shall write $\mathit{Dom}(\psi)$
instead of $\mathit{Dom}_{\varphi}(\psi)$ in the rest of 
the article.
Let $\sigma$ be the set of relation symbols that occur in $\varphi$. 
As discussed above, $\varphi^*$
contains extra relation symbols that encode information concerning subformulae of $\varphi$.
%
%
Let $\mathrm{QSUB}_{\varphi}$ denote
the set of formulae $\alpha\in\mathrm{SUB}_{\varphi}$
such that there exists another subformula $\psi = Qz\,\alpha \in\mathrm{SUB}_{\varphi}$,
where $Q\in \{\exists^{\geq k}, \exists^s\}$.
For each formula $\alpha \in\mathrm{QSUB}_{\varphi}$, define the
fresh relation symbols $E_{\alpha}^{\mathit{Uf}}$
and $E_{\alpha}^{\mathit{Vg'}}$. The arity of
each of these symbols is $|\mathit{Dom}(\alpha)|$, i.e., the number of
variables in $\mathit{Dom}(\alpha)$.
Additionally, for each formula $\chi\in\mathrm{SUB}_{\varphi}$,
define fresh relation symbols $S_{\chi}$ and $T_{\chi}$. The arity of the symbols $S_{\chi}$
and $T_{\chi}$ is equal to $|\mathit{Dom}(\chi)|$. The set of relation symbols in $\varphi^*$
is the set 
\begin{multline*}
\sigma\ \cup\ \{\, E_{\alpha}^{\mathit{Uf}}\ |\ \alpha\, \in\, \mathrm{QSUB}_{\varphi}\, \}\
\cup\ \{\, E_{\alpha}^{\mathit{Vg'}}\ |\ \alpha\, \in\, \mathrm{QSUB}_{\varphi}\, \}\\
\cup\ \{\, S_{\chi}\ |\ \chi\in\mathrm{SUB}_{\varphi}\ \}\
\cup\ \{\, T_{\chi}\ |\ \chi\in\mathrm{SUB}_{\varphi}\ \}.
\end{multline*}
Let $\sigma^*$ denote this set.
%

%
%
%
%
%
%

%
Define $\psi_{\mathit{initial}}\, :=\, \exists^{=1} x\, S_{\varphi}(x)\wedge\neg\exists x T_{\chi}(x)$.
Here $\exists^{=1}x$ is the $\mathrm{FOC}^2$-expressible quantfier that states that
there exists exactly one $x$ satisfying the quantified formula.
To fully define $\varphi^*$, we still need to define
the formulae $\psi_{\chi}$ for each formula $\chi\in\mathrm{SUB}_{\varphi}$.
%
%

%
Let $\chi\in\mathrm{SUB}_{\varphi}$. If $\chi = \chi_1\vee\chi_2$
and $\mathit{Dom}(\chi) = \{x,y\}$, 
then $\psi_{\chi}$ is the conjunction of the formulae
\begin{align*}
&\psi_{\chi}^1\ := \forall x\forall y \Bigl(\,  S_{\chi}(x,y)\
\leftrightarrow\ \bigl(S_{\chi_1}(x,y)\, \vee\, S_{\chi_2}(x,y)\bigr)\Bigr),\\
&\psi_{\chi}^2\ := \forall x\forall y \Bigl(\,  T_{\chi}(x,y)\
\leftrightarrow\ T_{\chi_1}(x,y)\Bigr),\\
&\psi_{\chi}^3\ := \forall x\forall y \Bigl(\,  T_{\chi}(x,y)\
\leftrightarrow\ T_{\chi_2}(x,y)\Bigr).
\end{align*}
%
%
%
%
%
 If $\chi = \exists^{\geq k}y\, \alpha$
and $\mathit{Dom}(\chi) = \{x,y\}$, 
then $\psi_{\chi}$ is the conjunction of the formulae
\begin{align*}
&\psi_{\chi}^1\ := \ \forall x\forall y\bigl(\ S_{\chi}(x,y)\
\rightarrow\ \exists^{\geq k} y\, E_{\alpha}^{\mathit{Uf}}(x,y)\, \bigr),\\
%
%
%
&\psi_{\chi}^2\ := \ \forall x\forall y\bigl(\, E_{\alpha}^{\Uf}(x,y)\
\rightarrow\ \exists y\, S_{\chi}(x,y)\, \bigr),\\
&\psi_{\chi}^3\ := \ \forall x\forall y\bigl(\ T_{\chi}(x,y)\
\rightarrow\ \neg \exists^{\geq k} y\, \neg E_{\alpha}^{\Vg'}(x,y)\, \bigr),\\
&\psi_{\chi}^4\ := \ \forall x\forall y\bigl(\, E_{\alpha}^{\Vgp}(x,y)\
\rightarrow\ \exists y\, T_{\chi}(x,y)\, \bigr),\\
&\psi_{\chi}^5\ := \forall x \forall y\bigl( S_{\alpha}(x,y)\
\leftrightarrow\ E_{\alpha}^{\Uf}(x,y)\, \bigr),\\
&\psi_{\chi}^6\ := \forall x \forall y\bigl(\, T_{\alpha}(x,y)\
\leftrightarrow\  E_{\alpha}^{\Vgp}(x,y)\, \bigr).
\end{align*}
If $\chi$ is the atomic formula $\depst(x,y)$,
and thus necessarily $\mathit{Dom}(\chi) = \{x,y\}$,
then $\psi_{\chi}$ is the conjunction of the formulae
\begin{align*}
&\psi_{\chi}^1\ :=\ \neg\exists x\exists^{\geq 2}y\, S_{\chi}(x,y),\\
&\psi_{\chi}^2\ :=\ \neg\exists x\exists y\, T_{\chi}(x,y).
\end{align*}
%
%
%
%
%

%
The structure of each formula $\psi_{\chi}$, where $\chi\in\mathrm{SUB}_{\varphi}$, depends on
$\chi$ and $\mathit{Dom}(\chi)$.
A complete list of these formulae is given in the Appendix.
\begin{lemma}\label{lemmafour}
Assume $\mA$ is a $\sigma$-model such that $\mA,\bigl(\{\emptyset\},\emptyset\, \bigr)\models\varphi$.
Let $A$ be the domain of\, $\mA$.
Then there exists a $\sigma^*$-model $\mA^*$ with the same domain $A$
such that $\mA^*\models_{\mathrm{FO}}\varphi^*$.
%
%
\end{lemma}
\begin{proof}
%
%
%
%
%
%
The relation symbols $R\in\sigma$
are interpreted in $\mA^*$ such that $R^{\mA^*} := R^{\mA}$.
The interpretations of the relation symbols in $\sigma^*\setminus\sigma$
are given below.
Let $U$ is an team with the domain $\{x\}$
and for the model $\mA$.
Assume $U$ contains exactly one assignment.
Since $\mA,(\{\emptyset\},\emptyset)\models\psi$,
we have $\mA,(U,\emptyset)\models\varphi$ by Lemma \ref{emptysetlemma}.
We shall next recursively define a double team $(U_{\chi},V_{\chi})$
for each subformula $\chi\in\mathrm{SUB}_\varphi$
such that $\mA, (U_{\chi},V_{\chi})\models\chi$ holds.
We shall simultaneously define the interpretations 
of the symbols in $\sigma^*\setminus\sigma$,
thereby completing the definition of the model $\mA^*$.
First define $(U_{\varphi},V_{\varphi}) := (U,\emptyset)$.
Define $S_{\varphi}^{\mA^*} = \mathit{Rel}(U_{\varphi})$.
Also define $T_{\varphi}^{\mA^*} :=\emptyset$.
Now consider a formula $\chi\in\mathrm{SUB}_{\varphi}$,
and assume that we have defined $U_{\chi}$ and $V_{\chi}$
such that $\mA,(U_{\chi},V_{\chi})\models\chi$.
Assume first that $\chi\, =\, \exists^{\geq k} x\, \alpha$.
As $\mA,(U_{\chi},V_{\chi})\models\exists^{\geq k} x\, \alpha$,
there exist functions $f: U_{\chi}\rightarrow {\exists^{\geq k}}^{\empty\, \mA}$
and $g:V_{\chi}\rightarrow{\overline{\exists^{\geq k}}}^{\empty\, \mA}$ such that
$$\bigl(\, U_{\chi}[\, x/f\, ]\cup V_{\chi}[\, x/g\, ],\ U_{\chi}[\, x/f'\, ]
\cup V_{\chi}[\, x/g'\, ]\, \bigr)\models\, \alpha.$$
Furthermore, by Lemma \ref{downwardslemma},
we assume, w.l.o.g., that the functions $f$ and $g$ are $x$-independent.
We make the following definitions.
\begin{enumerate}
\item
$U_{\alpha}\, :=\, U_{\chi}[\, x/f\, ]\, \cup\, V_{\chi}[\, x/g\, ]\, =\,
U_{\chi}[\, x/f\, ]$
\item
$V_{\alpha}\, :=\, U_{\chi}[\, x/f'\, ]\, \cup\, V_{\chi}[\, x/g'\, ]\, =\,
V_{\chi}[\, x/g'\, ]$
\item
$S_{\alpha}^{\mA^*} := \mathit{Rel}\bigl(U_{\alpha}\bigr)$
\item
$T_{\alpha}^{\mA^*}  := \mathit{Rel}\bigl(V_{\alpha}\bigr)$
\item
${E_{\alpha}^{\mathit{Uf}}}^{\mA^*}  := \mathit{Rel}\bigl(U_{\chi}[x/f]\bigr) = S_{\alpha}^{\mA^*}$
%
%
\item
${E_{\alpha}^{\mathit{Vg'}}}^{\mA^*}  := \mathit{Rel}\bigl(V_{\chi}[x/g']\bigr) = T_{\alpha}^{\mA^*}$
\end{enumerate}
The cases where $\chi$ is a formula of any of the types
$\exists^{\geq k} y\, \alpha,\ \exists^{s} x\, \alpha,\ \exists^{s} y\, \alpha$,
are treated analogously. It is essential---as we shall see---that the
function $f$ is $x$-independent in the case $\chi = \exists^s x\, \alpha$,
and $y$-independent when $\chi = \exists^s y\, \alpha$.
Consider then the case where $\chi$ is $\alpha\vee\beta$.
Since $(U_{\chi},V_{\chi})\models \alpha\vee\beta$,
we have $(U_1,V_{\chi})\models \alpha$ and
$(U_2,V_{\chi})\models \beta$ for some $U_1,U_2\subseteq U_{\chi}$
such that $U_1\cup U_2 = U_{\chi}$.
We define $(U_{\alpha},V_{\alpha}) := (U_1, V_{\chi})$
and $(U_{\beta},V_{\beta}):= (U_2, V_{\chi})$.
We also define $S_{\alpha}^{\mA^*} := \mathit{Rel}(U_{\alpha})$,
$T_{\alpha}^{\mA^*} := \mathit{Rel}(V_{\alpha})$,
$S_{\beta}^{\mA^*} := \mathit{Rel}(U_{\beta})$,
and $T_{\beta}^{\mA^*} := \mathit{Rel}(V_{\beta})$.
In the case where $\chi$ is $\neg\alpha$,
we define $U_{\alpha} := V_{\chi}$ and $V_{\alpha} := U_{\chi}$.
We also define $S_{\alpha}^{\mA^*} := \mathit{Rel}(U_{\alpha})$
and $T_{\alpha}^{\mA^*} := \mathit{Rel}(V_{\alpha})$.
We have now defined the teams $U_{\chi}$ and $V_{\chi}$
for each $\chi\in\mathrm{SUB}_{\varphi}$ such
that we have $\mA,(U_{\chi},V_{\chi})\models\chi$.
We have also fully defined a $\sigma^*$-model $\mA^*$.
We shall next show that $\mA^*\mfo\varphi^*$.
While it is clear that $\mA^*\mfo\psi_{\mathit{initial}}$, we
must show that $\mA^*\models_{\mathrm{FO}} \psi_{\chi}$
for each $\chi\in\mathrm{SUB}_{\varphi}$.
Let us first consider the case where $\chi$ is
of the form $\exists^{s}y\, \alpha$ for some $\alpha\in\mathrm{SUB}_\varphi$.
This case divides into further subcases, depending on $\mathit{Dom}(\chi)$.
We assume first that $\mathit{Dom}(\chi) = \{x,y\}$.
%
%
%
We know that there exist $y$-independent functions $f:U_{\chi}\rightarrow {\exists^{s}}^{\empty\, \mA}$ and
$g:V_{\chi}\rightarrow \overline{\exists^{s}}^{\empty\, \mA}$
such that 
$$\mA,\bigl(\, U_{\chi}[\, y/f\, ]\cup V_{\chi}[\, y/g\, ],\ U_{\chi}[\, y/f\, '\, ] 
\cup V_{\chi}[\, y/g\, '\, ]\, \bigr)\models\, \alpha.$$
We\, have\, $\mathit{Rel}(U_{\chi}[\, y/f\, ]) = {E_{\alpha}^{\Uf}}^{\empty\, \mA^*}$,\, \,
$\mathit{Rel}(V_{\chi}[\, y/g\, ]) = \emptyset$,\, \, 
$\mathit{Rel}(U_{\chi}[\, y/f\,'\, ]) = \emptyset$\, \, 
and\, $\mathit{Rel}(V_{\chi}[\, y/g\,'\, ]) = {E_{\alpha}^{\Vgp}}^{\empty\, \mA^*}$.
We shall first show that $\mA^*\models_{\mathrm{FO}}\psi_{\chi}^1$.
Here it is essential that the function $f$ is $y$-independent.
Assume that $\mA^*,[x\mapsto a,y\mapsto b]\models_{\mathrm{FO}} S_{\chi}(x,y)$.
Thus $(a,b)\in S_{\chi}^{\mA^*} = \mathit{Rel}(U_{\chi})$.
%
%
Since $f$ is $y$-independent,
there exists exactly one element $b'\in A$ such
that $(a,b')\in \mathit{Rel}(U_{\chi}[ y/f]) = {E_{\alpha}^{\Uf}}^{\mA^*}$.
Therefore we have $\mA^*,[x\mapsto a]\models_{\mathrm{FO}} \exists^{=1}y\ {E_{\alpha}^{\Uf}}(x,y)$,
as required.
%

%
%
%
We have $\mA^*\models_{\mathrm{FO}}\psi_{\chi}^2$
since for every assignment $s\in U_{\chi}[\, y/f  ]$ such that
$s(x) = a$, there must exist an assignment $s'\in U_{\chi}$ such that $s'(x) = a$.
We can similarly show that $\mA^*\models_{\mathrm{FO}} \psi_{\chi}^3\wedge\psi_{\chi}^4$.
%
%
%

%
The fact that $\mA^*\models_{\mathrm{FO}} \psi_{\chi}^5\wedge\psi_{\chi}^6$
follows immediately since
$U_{\alpha} = U_{\chi}[ y/f ]\, \cup\, V_{\chi}[ y/g ]$
and $V_{\alpha} = U_{\chi}[\, y/f\, '\,  ]\, \cup\, V_{\chi}[\, y/g\, '\,  ]$.
The cases where $\mathit{Dom}(\chi)$ is $\{x\}$, $\{y\}$, or $\emptyset$, are similar,
as are the cases where $\chi := \exists^{s}x\, \alpha$.
Also all cases where $\chi := \exists^{\geq k} y\, \alpha$ or $\exists^{\geq k} x\, \alpha$ are similar;
we shall discuss the details of the case where $\chi := \exists^{\geq k} x\, \alpha$ and
$\mathit{Dom}(\chi) = \{x\}$.
We know that there exist functions $f:U_{\chi}\rightarrow{\exists^{\geq k}}^{\text{ }\mA}$
and $g:V_{\chi}\rightarrow{\overline{\exists^{\geq k}}}^{\text{ }\mA}$
such that
$$\mA,\bigl(\, U_{\chi}[\, x/f\, ]\cup V_{\chi}[\, x/g\, ],\ U_{\chi}[\, x/f\, '\, ] 
\cup V_{\chi}[\, x/g\, '\, ]\, \bigr)\models\, \alpha.$$
We\, have\, $\mathit{Rel}(U_{\chi}[\, x/f\, ]) = {E_{\alpha}^{\Uf}}^{\empty\, \mA^*}$,\, \,
$\mathit{Rel}(V_{\chi}[\, x/g\, ]) = \emptyset$,\, \, 
$\mathit{Rel}(U_{\chi}[\, x/f\,'\, ]) = \emptyset$\, \, 
and\, $\mathit{Rel}(V_{\chi}[\, x/g\,'\, ]) = {E_{\alpha}^{\Vgp}}^{\empty\, \mA^*}$.
Let us show that $\mA^*\models_{\mathrm{FO}}\psi_{\chi}^1$.
Assume that $\mA^*,[x\mapsto a]\models_{\mathrm{FO}} S_{\chi}(x)$ for some $a\in A$.
Thus $S_{\chi}^{\mA^*} = \mathit{Rel}(U_{\chi}) \not= \emptyset$,
whence $U_{\chi}\not=\emptyset$ . 
Therefore there exist at least $k$ elements $b\in A$ such
that $b\in \mathit{Rel}(U_{\chi}[ x/f ]) = {E_{\alpha}^{\Uf}}^{\mA^*}$.
Therefore $\mA^*\models_{\mathrm{FO}} \exists^{\geq k}x\, {E_{\alpha}^{\Uf}}(x)$, as required.
%

%
%
%
%
We have $\mA^*\models_{\mathrm{FO}}\psi_{\chi}^2$ since if $U_{\chi}[\, x/f\, ] \not= \emptyset$,
then $U_{\chi}\not=\emptyset$.
%
%
%
To show that $\mA^*\models_{\mathrm{FO}}\psi_{\chi}^3$,
assume that $\mA^*,[x\mapsto a]\models_{\mathrm{FO}} T_{\chi}(x)$ for some $a\in A$.
Thus $T_{\chi}^{\mA^*} = \mathit{Rel}(V_{\chi})$ is not empty. Let $s\in V_{\chi}$.
Recall that $g_2$ denotes the second coordinate function of $g$.
By the definition of the minor quantifier $\exists^{\geq k}$,
there are at most $k-1$ elements 
in the set $A \setminus g_2(s)$. Thus there are at most $k-1$
elements in $A\setminus\mathit{Rel}\bigl(V_{\chi}[ x/g\, '\, ]\bigr)$.
Therefore we have $\mA^*\models_{\mathrm{FO}} \neg \exists^{\geq k} x \neg {E_{\alpha}^{\Vgp}}(x)$,
and hence $\mA^*\models_{\mathrm{FO}}\psi_{\chi}^3$.
We have $\mA^*\models_{\mathrm{FO}} \psi_4$ since if $V_{\chi}[\, x/g\, '\,  ]$ is not empty,
then $V_{\chi}$ cannot be empty.
We have $\mA^*\models_{\mathrm{FO}}\psi_5\wedge\psi_6$
since $U_{\alpha} = U_{\chi}[ x/f ]$
and $V_{\alpha} = V_{\chi}[\, x/g\, '\,  ]$.
The cases where $\chi = \chi_1\vee\chi_2$ and $\chi = \neg\alpha$ are straightforward, so we
omit them and move directly to the cases where $\chi$ is an atomic formula.
Assume first that $\chi = R(y,x)$ for some relation symbol $R$.
We must show that
$\mA^*\models_{\mathrm{FO}} \forall x\forall y\bigl(S_{R(y,x)}(x,y)\rightarrow R(y,x)\bigr).$
(Notice indeed the order of all tuples of variables.)
Assume that
$\mA^*,[x\mapsto a, y\mapsto b]\models_{\mathrm{FO}} S_{R(y,x)}(x,y).$
By the definition of the relation $S_{R(y,x)}^{\mA^*}$, this means that
$(a,b)\in \mathit{Rel}(U_{R(y,x)})$.
We have $\mA,(U_{R(y,x)},V_{R(y,x)})\models_{\mathrm{FO}}R(y,x)$,
and therefore $\mA^*,s\models_{\mathrm{FO}} R(y,x)$
for all $s\in U_{R(y,x)}$.
Thus $\mA^*,[x\mapsto a,y\mapsto b]\models_{\mathrm{FO}}R(y,x)$.
All the remaining arguments for the cases where $\chi$ is an
atomic first-order formula, are similar.
Assume then that $\chi$ is the atom $\depst(x,y)$.
We must establish that we have 
$\mA^*\models_{\mathrm{FO}} \neg \exists x\exists^{\geq 2} y\ S_{=(x,y)}(x,y)$.
Assume $\mA^*,[x\mapsto a,y\mapsto b]\models_{\mathrm{FO}} S_{=(x,y)}(x,y)$
for some $a,b\in A$.
Therefore $(a,b)\in\mathit{Rel}(U_{=(x,y)})$.
We have
$\mA,\bigl(U_{=(x,y)},V_{=(x,y)}\bigr)\models\ \depst(x,y),$
and therefore
$s(y) = s'(y)$ for all $s,s'\in U_{=(x,y)}$ such that $s(x) = s'(x)$.
Hence there is no pair
$(a,b')\in\mathit{Rel}(U_{=(x,y)}) = S_{=(x,y)}^{\mA^*}$
such that $b\not= b'$.
Thus $\mA^*\models_{\mathrm{FO}} \neg \exists x\exists^{\geq 2} y\, S_{=(x,y)}(x,y)$,
as required. 
All remaining arguments concerning non-first-order atoms are similar.
\end{proof}
\begin{lemma}
Let $\mB^*$ be a $\sigma^*$-model such that $\mB^*\mfo \varphi^*$.
Let $B$ be the domain of\, $\mB^*$.
Then there exists a $\sigma$-model $\mB$ with the same domain $B$ such
that $\mB,\bigl(\{\emptyset\},\emptyset\, \bigr)\models\varphi$.
%
%
\end{lemma}
\begin{proof}
Assume that $\mB^*\mfo \varphi^*$.
%
%
Let $\mB$  be the reduct of $\mB^*$ to the
vocabulary $\sigma$, i.e.,
the domain of $\mB$ is $B$, and
each relation symbol $R\in\sigma$ is interpreted
such that $R^{\mB} := R^{\mB^*}$.
%

%
%
We shall next define a
double team $(U_{\chi},V_{\chi})$ for each $\chi\in\mathrm{SUB}_{\varphi}$.
We shall then establish that $\mB,\bigl(U_{\chi},V_{\chi}\big)\models\chi$ for
each $\chi\in\mathrm{SUB}_{\varphi}$.
If $\mathit{Dom}(\chi)$ is any of the sets $\{x\}, \{y\}, \{x,y\}$,
we let $U_{\chi}$ and $V_{\chi}$ be the teams with the domain $\mathit{Dom}(\chi)$
and codomain $B$ such that $\mathit{Rel}(U_{\chi}) = S_{\chi}^{\mB^*}$
and $\mathit{Rel}(V_{\chi}) = T_{\chi}^{\mB^*}$.
If $\mathit{Dom}(\chi)$ is $\emptyset$,
we let $U_{\chi}$ and $V_{\chi}$ be the teams with the domain $\{x\}$
and codomain $B$ such that $\mathit{Rel}(U_{\chi}) = S_{\chi}^{\mB^*}$
and $\mathit{Rel}(V_{\chi}) = T_{\chi}^{\mB^*}$.
We shall prove by induction on the structure of $\varphi$
that $\mB,\bigl(U_{\chi},V_{\chi}\bigr)\models\chi$
for each $\chi\in\mathrm{SUB}_{\varphi}$.
We shall then establish that $\mB,\bigl(\{\emptyset\},\emptyset\, \bigr)\models\varphi$.
Assume first that $\chi$ is the atomic formula $R(y,x)$.
Let $s\in U_{R(y,x)}$ be an assignment.
Thus $\mB^*,s\models_{\mathrm{FO}} S_{R(y,x)}(x,y)$.
Since $\mB^*\mfo \psi_{R(y,x)}$, 
we have $\mB^*,s\models_{\mathrm{FO}} R(y,x)$.
We show similarly that if $t\in V_{R(y,x)}$, then $\mB^*,t\not\models R(y,x)$.
Therefore $\mB,\bigl(U_{R(y,x)},V_{R(y,x)}\bigr)\models R(y,x)$.
The corresponding argument for other first-order atoms is similar. 
Let $\chi$ be the atom $\depst(x,y)$.
Since $\mB^*\models_{\mathrm{FO}}\psi_{\chi}$,
there exist no pairs $(a,b),(a,b')\in S_{\chi}^{\mB^*}$ such that
$b\not=b'$. Furthermore, $T_{\alpha}^{\mB^*} = \emptyset$. 
Therefore $\mB,(U_{\chi},V_{\chi})\models\chi$.
The corresponding arguments for other
non-first-order atoms of $\mathrm{DC}^2$ are similar.
For the sake of induction, 
let $\chi\, :=\, \exists^{\geq k} y\, \alpha$ be a subformula of $\varphi$,
and assume that $\mB,(U_{\alpha},V_{\alpha})\models\alpha$.
We need to show that $\mB,(U_{\chi},V_{\chi})\models\chi$.
Let us consider the case where $\mathit{Dom}(\chi) = \{x,y\}$.
We define a function $f:U_{\chi}\rightarrow{\exists^{\geq k}}^{\empty\, \mB}$
as follows. Assume $s\in U_{\chi}$ is an 
assignment such that $s(x) = a$ and $s(y) = b$
for some $a,b\in B$. Thus $(a,b)\in S_{\chi}^{\mB^*}$.
Since $\mB^*\models_{\mathrm{FO}}\psi_{\chi}^1$, the set
\begin{equation}\label{uniqueset}
B_s\ :=\ \{\, c\in A\ |\ \mB^*,[\, x\mapsto a,\, y\mapsto c\, ]
\models E_{\alpha}^{\Uf}(x,y)\ \}
\end{equation}
has at least $k$ elements.
%
%
Define $f:U_{\chi}\rightarrow{\exists^{\geq k}}^{\empty\, \mB}$ such that
$f(s) := (B_s,\emptyset)$ for each $s\in U_{\chi}$.
Thus $\mathit{Rel}( U_{\chi}[\, y/f\, ]) \subseteq {E_{\alpha}^{\Uf}}^{\mB^*}$.
%
%
%

%
Let us then similarly define a function $g:V_{\chi}\rightarrow{\overline{\exists^{\geq k}}}^{\empty\, \mB}$.
Let $s\in V_{\chi}$ be an
assignment such that $s(x) = a$ and $s(y) = b$ for some $a,b\in B$.
Thus $(a,b)\in T_{\chi}^{\mB^*}$.
Since $\mB^*\models_{\mathrm{FO}}\psi_{\chi}^3$, the number of elements in the set
\begin{equation}\label{another}
C_s\ :=\ \{\, c\in A\ |\ \mB^*,[\, x\mapsto a,\, y\mapsto c\, ]\models E_{\chi}^{\Vg'}(x,y)\ \}
\end{equation}
satisfies the condition $|B\setminus C_{s}| < k$.
Define $g:V_{\chi}\rightarrow\overline{\exists^{\geq k}}^{{}\, \mB}$
such that $g(s) := (\emptyset,C_s)$ for each $s\in V_{\chi}$.
Thus $\mathit{Rel}( V_{\chi}[\, y/g\, '\, ])
\subseteq {E_{\alpha}^{\Vgp}}^{\empty\, \mB^*}$.
As $U[\, y/f\, '\, ] = V[\, y/g\, ] = \emptyset$, we now know that
\begin{equation}\label{first}
\mathit{Rel}\bigl(\, U_{\chi}[\, y/f\, ]\, )\, \cup\,
\mathit{Rel}(\, V_{\chi}[\, y/g\, ]\, \bigr)\
\subseteq\ {E_{\alpha}^{\Uf}}^{\mB^*} 
\end{equation}
and
\begin{equation}\label{second}
\mathit{Rel}\bigl(\, U_{\chi}[\, y/f\, ' \, ]\, )\, \cup\,
\mathit{Rel}(\, V_{\chi}[\, y/g\, ' \, ]\, \bigr)\
\subseteq\ {E_{\alpha}^{\Vgp}}^{\mB^*}.
\end{equation}
We then show that also the converse inclusion of Equation \ref{first} holds.
Assume that $(a,c)\in{E_{\alpha}^{\Uf}}^{\mB^*}$.
As $\mB^*\models_{\mathrm{FO}}\psi_{\chi}^2$, there exists some $b\in B$
such that $(a,b)\in\mathit{Rel}(U_{\chi})$.
Let $s\in U_{\chi}$ be the assignment such that $s(x) = a$ and $s(y) = b$.
Now, by the definition of $f$ (see Equation \ref{uniqueset}),
we observe that since $(a,c)\in{E_{\alpha}^{\Uf}}^{\mB^*}$,
we have $c\in f_1(s)$;
recall here that $f_1$ denotes the first coordinate function of $f$.
Thus $(a,c)\in\mathit{Rel}(U_{\chi}[\, y/f\, ])$.
Therefore the converse inclusion of Equation \ref{first} holds.
We then establish that also the converse inclusion of Equation \ref{second} holds.
Assume that $(a,c)\in{E_{\alpha}^{\Vgp}}^{\mB^*}$.
As $\mB^*\models_{\mathrm{FO}}\psi_{\chi}^4$, there exists some $b\in B$
such that 
%
%
$(a,b)\in\mathit{Rel}(V_{\chi})$.
Let $s\in V_{\chi}$ be the assignment such that $s(x) = a$ and $s(y) = b$.
By the definition of the function $g$ (Equation \ref{another}),
we  observe that $c\in g_2(s)$. Thus $(a,c)\in\mathit{Rel}(V_{\chi}[\, y/g\, ' \, ])$.
Hence the converse inclusion of Equation \ref{second} holds.
As $\mB^*\models_{\mathrm{FO}}\psi_{\chi}^5\wedge\psi_{\chi}^6$,
we conclude that $U_{\chi}[\, y/f\, ]\cup V_{\chi}[\, y/g\, ] = U_{\alpha}$
and $V_{\chi}[\, y/f\, ' \,  ]\cup V_{\chi}[\, y/g\, ' \, ] = V_{\alpha}$.
As $\mB,(U_{\alpha},V_{\alpha})\models \alpha$, we
therefore conclude that $\mB, (U_{\chi},V_{\chi})\models\chi$.
%

%
%

%
The remaining cases where $\chi = \exists^{\geq k} y\, \alpha$ or $\,\chi = \exists^{\geq k} x\, \alpha$,
are similar. We next deal with the strict
existential quantifier $\exists^s$.
Let $\chi\, :=\, \exists^{s} x\, \alpha$, and
assume $\mB,(U_{\alpha},V_{\alpha})\models\alpha$.
Let us consider the details of case where $\mathit{Dom}(\chi) = \{y\}$.
We define a function $f:U_{\chi}\rightarrow{\exists^{s}}^{\empty\, \mB}$
as follows. Assume $s\in U_{\chi}$ is an 
assignment such that $s(y) = a$ for some $a\in B$.
Thus $a\in S_{\chi}^{\mB^*}$.
Since $\mB^*\models_{\mathrm{FO}}\psi_{\chi}^1$, the size of the set
%
%
%
%
\begin{equation}\label{there}
B_s\ :=\ \{\, c\in A\ |\ \mB^*,[\, x\mapsto c,\, y\mapsto a\, ]
\models E_{\alpha}^{\Uf}(x,y)\, \}
\end{equation}
is exactly one.
Define $f:U_{\chi}\rightarrow{\exists^{s}}^{\empty\, \mB}$ such that
$f(s) := (B_s,\emptyset)$ for each $s\in U_{\chi}$.
Thus $\mathit{Rel}( U_{\chi}[\, x/f\, ]) \subseteq {E_{\alpha}^{\Uf}}^{\mB^*}$.
We also of course have $\mathit{Rel}( U_{\chi}[\, x/f\, '\, ]) = \emptyset$.
%
%
%

%
Let us then define the function $g:V_{\chi}\rightarrow{\overline{\exists^{s}}}^{\empty\, \mB}$
such that $g(s) = (\emptyset,B)$ for each $s\in V_{\chi}$.
%
%
%
Assume $s\in V_{\chi}$ is an assignment such
that $s(y) = a$. Thus $a\in T_{\chi}^{\mB^*}$.
%
%
Since $\mB^*\models_{\mathrm{FO}}\psi_{\chi}^3$, we have $(c,a) \in {E_{\alpha}^{\Vgp}}^{\empty\, \mB^*}$
for each $c\in A$. 
Thus $\mathit{Rel}( V_{\chi}[\, x/g\, '\, ])
\subseteq {E_{\alpha}^{\Vgp}}^{\empty\, \mB^*}$.
As $\mathit{Rel}( V_{\chi}[\, x/g\, ])$
and $\mathit{Rel}( V_{\chi}[\, x/f\, ' \, ])$ are empty,
we have 
\begin{equation}\label{third}
\mathit{Rel}\bigl(\, U_{\chi}[\, x/f\, ]\, )\, \cup\,
\mathit{Rel}(\, V_{\chi}[\, x/g\, ]\, \bigr)\
\subseteq\ {E_{\alpha}^{\Uf}}^{\mB^*}
\end{equation}
and
\begin{equation}\label{fourth}
\mathit{Rel}\bigl(\, U_{\chi}[\, x/f\, ' \, ]\, )\, \cup\,
\mathit{Rel}(\, V_{\chi}[\, x/g\, ' \, ]\, \bigr)\
\subseteq\ {E_{\alpha}^{\Vgp}}^{\mB^*}.
\end{equation}
We then show that the converse inclusion of Equation \ref{third} holds.
Assume that $(a,b)\in{E_{\alpha}^{\Uf}}^{\mB^*}$.
As $\mB^*\models_{\mathrm{FO}}\psi_{\chi}^2$, we have $b\in\mathit{Rel}(U_{\chi})$.
Let $s\in U_{\chi}$ be the assignment such that $s(y) = b$.
Now, by the definition of $f$ (see Equation \ref{there}),
since $(a,b)\in {E_{\alpha}^{\Uf}}^{\mB^*}$,
we have $(a,b)\in\mathit{Rel}\bigl(U[\, x/f\, ]\bigr)$.
Therefore the converse inclusion of Equation \ref{third} holds.
It is easy to establish that also the converse inclusion of Equation \ref{fourth} holds.
Therefore, as $\mB^*\models_{\mathrm{FO}} \psi_{\chi}^5\wedge\psi_{\chi}^6$,
we infer that $U_{\chi}[\, y/f\, ]\cup V_{\chi}[\, y/g\, ] = U_{\alpha}$
and $V_{\chi}[\, y/f\, ' \,  ]\cup V_{\chi}[\, y/g\, ' \, ] = V_{\alpha}$.
As $\mB,(U_{\alpha},V_{\alpha})\models \alpha$, we
therefore conclude that $\mB, (U_{\chi},V_{\chi})\models\chi$.
We have now discussed the cases where $\chi = \exists^{\geq k} z\, \alpha$
or $\chi = \exists^{s} z\, \alpha$; here $z\in\{x,y\}$.
The arguments for the cases where $\chi = \alpha\vee\beta$
or $\chi = \neg\alpha$, are straightforward.
We conclude that $\mB,(U_{\varphi},V_{\varphi})\models\varphi$.
Since $\mB^*\models_{\mathrm{FO}}\psi_{\mathit{initial}}$, we have 
$\mathit{Rel}(U_{\varphi}) = S_{\varphi}^{\mB^*} = \{b\}$
for some $b\in B$ and $\mathit{Rel}(V_{\varphi}) = T_{\varphi}^{\mB^*} =\emptyset$.
Hence $\mB,\bigl(\{\emptyset\},\emptyset\, \bigr)\models\varphi$
by Lemma \ref{emptysetlemma}.
\end{proof}
\begin{theorem}
The satisfiability and finite satisfiability problems of\, $\mathrm{DC}^2$
are complete for $\mathrm{NEXPTIME}$.
\end{theorem}
\begin{proof}
The satisfiability and finite satisfiability problems of  $\mathrm{DC}^2$ are
in $\mathrm{NEXPTIME}$ due to the translation
from $\mathrm{DC}^2$ into $\mathrm{FOC}^2$ defined above;
it is shown in \cite{pratthartmann} that the satisfiability
and finite satisfiability problems for $\mathrm{FOC}^2$ are
$\mathrm{NEXPTIME}$-complete.
Furthermore, the satisfiability and finite satisfiability
problems for $\mathrm{DC}^2$ are 
$\mathrm{NEXPTIME}$-hard, since $\mathrm{DC}^2$ contains
$\mathrm{FOC}^2$.
%
%
%
\end{proof}
\section{A semantics for single teams}\label{singleteam}
In this section we define a semantics for variants of
dependence logic with generalized quantifiers based on single teams.
We also simplify the notion of a generalized atom in a trivial way so that it works
naturally in this context.
Let us first define the following semantics with two
semantic turnstiles $\models^+$ and $\models^-$  instead of one.
%
%

%
\[
\begin{array}{lll}
 \mathfrak{A},U\models^+ y_1=y_2 & \ \Leftrightarrow \ & 
          \forall s\in U\bigl(\mathfrak{A},s\models_{\mathrm{FO}} y_1=y_2\bigr).\\
 \mathfrak{A},U\models^- y_1=y_2 & \ \Leftrightarrow \ & 
          \forall s\in U\bigl(\mathfrak{A},s\models_{\mathrm{FO}} y_1\not=y_2\bigr).\\
%
%
 \mathfrak{A},U\models^+ R(y_1,...,y_m)& \ \Leftrightarrow \ & 
          \forall s\in U\bigl(\mathfrak{A},s\models_{\mathrm{FO}} R(y_1,...,y_m)\bigr).\\
%
%
 \mathfrak{A},U\models^- R(y_1,...,y_m)& \ \Leftrightarrow \ & 
          \forall s\in U\bigl(\mathfrak{A},s\not\models_{\mathrm{FO}} R(y_1,...,y_m)\bigr).\\
 \mathfrak{A},U\models^+ \neg\varphi & \ \Leftrightarrow \ & 
 \mathfrak{A},U\models^- \varphi.\\ 
 \mathfrak{A},U\models^- \neg\varphi & \ \Leftrightarrow \ & 
 \mathfrak{A},U\models^+ \varphi.\\ 
   \mathfrak{A},U\models^+(\varphi\vee\psi) & \ \Leftrightarrow \ &
%
%
%
           \mathfrak{A},U_1\models^+\varphi\text{ and }
          \mathfrak{A},U_2\models^+\psi\text{ for}\\
         & &\text{some }U_1,U_2\subseteq U\text{ such that }
          U_1\cup U_2 = U.\\
   \mathfrak{A},U\models^-(\varphi\vee\psi) & \ \Leftrightarrow \ &
%
%
%
           \mathfrak{A},U\models^-\varphi\text{ and }
          \mathfrak{A},U\models^-\psi.\\
\end{array}
\]
For a generalized quantifier $Q$ of the type $(i_1,...,i_n)$, we define
$$\mathfrak{A},U\models^+ Q\overline{x}_1,...,\overline{x}_n(\varphi_1,...,\varphi_n)$$
if and only if there exists a function $f:U\rightarrow  Q^{\mathfrak{A}}$
such that 
\[
\begin{array}{c}
 \mathfrak{A},U[\, \overline{x}_1/f_1]\models^+\varphi\text{ and }
\, U[\, \overline{x}_1/{f_1}'\, ]\models^-\varphi_1,\\
 \vdots\\
 \mathfrak{A},U[\, \overline{x}_n/f_n]\models^+\varphi\text{ and }
\, U[\, \overline{x}_n/{f_n}'\, ]\models^-\varphi_n.\\
\end{array}
\]
We also define
$$\mathfrak{A},U\models^-Q\overline{x}_1,...,\overline{x}_n(\varphi_1,...,\varphi_n)$$
if and only if there exists a function $g:U\rightarrow  \overline{Q}^{\mathfrak{A}}$
such that 
\[
\begin{array}{c}
 \mathfrak{A},U[\, \overline{x}_1/g_1]\models^+\varphi\text{ and }
\, U[\, \overline{x}_1/{g_1}'\, ]\models^-\varphi_1,\\
 \vdots\\
 \mathfrak{A},U[\, \overline{x}_n/g_n]\models^+\varphi\text{ and }
\, U[\, \overline{x}_n/{g_n}'\, ]\models^-\varphi_n.\\
\end{array}
\]
It is straightforward to establish the following proposition.
\begin{proposition}\label{first-ordercorres}
Let $\varphi$ be a formula of first-order logic, possibly extended with generalized
quantifiers. Let $U$ be a team. Then the equivalences 
$\mathfrak{A},U\models^+\varphi\ \Leftrightarrow\
\forall s\in U(\mathfrak{A},s\models_{\mathrm{FO}}\varphi)$
and $\mathfrak{A},U\models^-\varphi\ \Leftrightarrow\
\forall s\in U(\mathfrak{A},s\not\models_{\mathrm{FO}}\varphi)$ hold.
\end{proposition}
For a minor quantifier $M$, we define
$\mathfrak{A},U^+\models M x\, \varphi$
if and only if there exists a function $f:U\rightarrow M^{\mathfrak{A}}$
such that 
\[
\begin{array}{c}
 \mathfrak{A},U[\, x/f]\models^+\varphi\text{ and }
\, \mA,U[\, x/f'\, ]\models^-\varphi.\\
%
%
%
%
\end{array}
\]
We also define
$\mathfrak{A},U\models^-M x\, \varphi$
if and only if there exists a function $g:U\rightarrow \overline{M}^{\mathfrak{A}}$
such that 
\[
\begin{array}{c}
 \mathfrak{A},U[\, x/g]\models^+\varphi\text{ and }
\, \mA,U[\, x/g'\, ]\models^-\varphi.\\
%
%
%
%
\end{array}
\]
If $Q$ is a generalized quantifier and $M\leq Q$ its minor, we can
replace $Q$ by $M$ or vice versa, without affecting the satisfaction of
formulae. Note, however, that this interchangeability does not  generally hold
if we add generalized atoms into the picture.
Indeed, we can naturally extend the single team framework with a suitable notion of a
generalized atom.
Let $(Q,P)$ be a pair of generalized
quantifiers, each of the type $(i_1,...,i_k)$.
Consider syntactic atomic expressions 
of the type $A(\overline{y}_1,...,\overline{y}_k)$, where
each $\overline{y}_j$ is of the length $i_j$.
We define
$$\mA,U\models^+A(\overline{y}_1,...,\overline{y}_k)\
\Leftrightarrow\ \bigl(\mathit{Rel}(U,\mA,\overline{y}_1),...,\mathit{Rel}(U,\mA,\overline{y}_k)\bigr)
\in Q^{\mA}$$
and
$$\mA,U\models^-A(\overline{y}_1,...,\overline{y}_k)\
\Leftrightarrow\ \bigl(\mathit{Rel}(U,\mA,\overline{y}_1),...,\mathit{Rel}(U,\mA,\overline{y}_k)\bigr)
\in P^{\mA}.$$
Of course the functions $f$ and $g$ need to respect the repetitions of
the tuples $\overline{y}_i$.
We do not claim that the single team semantics is
somekind of a \emph{counterpart} of the double team semantics.
There are interesting subtleties related to
differences between the single team semantics and the double
team semantics.
%
%
For example, let $B$ denote the atom of the type $(1;1)$
such that $\mA,(U,V)\models B(x)$ iff both
$\mathit{Rel}(U,\mA,x) = A$ and $\mathit{Rel}(V,\mA,x) = A$, where $A$ is the domain of $\mA$.
Let $B^*$ denote the atom for single team semantics such that 
$\mA,U\models^+ B^*(x)$ iff $\mathit{Rel}(U,\mA,x) = A$ and
$\mA,U\models^- B^*(x)$ iff $\mathit{Rel}(U,\mA,x) = A$.
Let $\exists^t$ denote the minor quantifier $M_{\exists}$.
Let $\mB$ be a model whose domain contains two elements.
Now $\mB,\bigl(\{\emptyset\},\emptyset\, \bigr)\models \exists^t x\, \exists^t x \, B(x)$,
while $\mB,\{\emptyset\}\not\models^+\, \exists^t x\, \exists^t x \, B^*(x)$ and
$\mB,\emptyset\not\models^-\, \exists^t\, x\exists^t x\, B^*(x)$.
It is not difficult to devise a corresponding symmetric game-theoretic semantics
for single teams, but we shall not do this in the current article for the sake of brevity.
The uniformity condition here seems to be---in a subtle way---quite
different from the uniformity condition of the game semantics corresponding to the
double team semantics.
But as said, we shall not attempt to provide an
account of the game corresponding to the single team semantics in this article.
\section{Reflections on general perspectives}
In this section we briefly discuss the interpretation
of team semantics by considering a rather general approach
to related technical issues.
The investigations are based on the use 
of a semantics that resembles
Scott-Montague semantics, as
suggested in \cite{kuusisto3}.
The findings may perhaps elucidate
issues related to team semantics and double
team semantics, and provide insight
into the differences of the two approaches.
The investigations are also of interest
independently of team semantics.
Let $\tau$ be a vocabulary.
Consider structures of the type $$(\mA,S_1,...,S_n),$$
where $\mA$ is a $\tau$-structure with the domain $A$,
and for each $j\in\{1,...,n\}$,
$S_j \subseteq A^{i_j}$ is a relation of the arity $i_j$.
Thus the relations $S_1,...,S_n$ have the arities $i_1,...,i_n$, respectively.
Note that the relations $S_j$ are not 
part of $\mA$, but are instead extra relations.
Define an \emph{operation} of the type
$(\tau,i_1,...,i_{n+1})$ 
to be a class function $F$ (too large to be a set)
that maps any structure $(\mA,S_1,...,S_n)$ of the appropriate type to a
relation $R\subseteq A^{i_{n+1}}$
of the arity $i_{(n+1)}$.
The operator $F$ satisfies the constraint that if $\mA$ and $\mB$
are $\tau$-models with the domains $A$ and $B$,
respectively, and if $f:A\rightarrow B$ is an isomorphism from $(\mA,S_1,...,S_n)$ to 
$(\mB,S_1',...,S_n')$, then $f(F(\, (\mA,S_1,...,S_n)\, )\, ) =
F(\, (\mB,S_1',...,S_n')\, )$.\footnote{Obviously the isomorphism
takes into account relations in $\tau$ as well as the external relations.
Also note that if $S\subseteq A^k$, then $f(S) = \{\ (f(a_1),...,f(a_k))\ |\ (a_1,...,a_k)\in S\ \}$.}
A very important class of operators is the class where
$\tau$ is the empty signature. In the elaborations below,
it may help to always first consider the special case of
such operators.
Fix a possibly infinite index set $I$.
A logic that deals with the above operations can be based on a grammar of the type
$$\varphi\ ::=\ P_i\ |\ \langle F\rangle(\varphi_1,...,\varphi_n),$$
where $P_i$ is a relation symbol such that $i\in I$.
The relation symbols $P_i$ may have different arities.
Of course we can have more than one operator $F$
in the logic; for example, if $G$ and $H$ are operators of the
types $(\tau,i_1',...,i_{m+1}')$ and
$(\tau,i_1'',...,i_{k+1}'')$, respectively,
then we can define a logic given by the grammar
$$\varphi\ ::=\ P_i\ |\ \langle F\rangle(\varphi_1,...,\varphi_n)\
|\ \langle G\rangle(\varphi_1,...,\varphi_m)\ |\ \langle H\rangle(\varphi_1,...,\varphi_k).$$
Note that the signature $\tau$ in the type of each operator is the same.
Each formula is associated with an arity. The arity
of an atomic formula $P_i$ is the arity of the relation symbol $P_i$.
If $H$ is an operator of the type $(\tau,i_1'',...,i_{k+1}'')$,
then the arity of $\langle H\rangle(\varphi_1,...,\varphi_k)$
is $i_{(k+1)}''$. Let $\mathit{Ar}(\varphi)$ denote the arity of $\varphi$.
Importantly, in the grammars above, we need the extra condition that
for each $j$, the arity $\mathit{Ar}(\varphi_j)$ of the
formula $\varphi_j$ in the formula
$\langle F\rangle(\varphi_1,...,\varphi_n)$ 
is $i_j$; recall that $F$ is an operator of the type $(\tau,i_1,...,i_{n+1})$.
A similar convention obviously concerns the operators $G$ and $H$ as well. 
The semantics of the logic is defined with respect to
pointed models $\bigl((\mA,\{P_i\}_{i\in I}\bigr),\overline{a})$,
where $\mA$ is a $\tau$-model with the domain $A$,
the objects $P_i\subseteq A^{\mathit{Ar}(P_i)}$ are relations,
and $\overline{a}$ is a tuple of elements of $A$.
Call $M := (\mA,\{P_i\}_{i\in I}\bigr)$.
Let $\overline{a} \in A^{\mathit{Ar}(P_i)}$ and 
$\overline{b} \in A^{\mathit{Ar}\bigl(\langle F \rangle(\varphi_1,...,\varphi_n)\bigr)}$.
The semantics of atomic formulae asserts
that $(M,\overline{a})\models P_i$ iff $\overline{a}\in P_i$.
The semantics of compound formulae
asserts that $(M,\overline{b})\models \langle F \rangle(\varphi_1,...,\varphi_n)$
iff we have
$$\overline{b}\in F\bigl(\, \mA,||\varphi_1||^M,...,||\varphi_n||^M\, \bigr),$$
where $||\varphi_i||^M = \{\ \overline{v}\in A^{\mathit{Ar(\varphi_i)}}
\ |\ (M,\overline{v})\models \varphi_i\ \}$.
This system bears some resemblance to the Scott-Montague semantics of modal logic.
This approach to logic is very general. To see why, consider
operators of the type $(\emptyset,1,...,1)$.
The related logic is interpreted by pointed models of the type $((W,\{P_i\}_{i\in I}),w)$,
where $W$ is a nonempty set, $w\in W$ and $P_i\subseteq W$ for each $i$.
Call $M := (W,\{P_i\}_{i\in I})$.
We may consider $W$ to be an abstract set
of atomic semantic objects, and the subset $||\varphi||^M$ of $W$ is the semantic
value (or meaning) of the formula $\varphi$ in the model $M$.
Importantly, there is great freedom in the choice of operators $F$ considered.
However, each operator $F$ is \emph{compositional} in the sense that the
semantic value $||\langle F \rangle(\varphi_1,...,\varphi_n)||^M$
of the formula $\langle F \rangle(\varphi_1,...,\varphi_n)$ is
functionally determined by $F$ from the semantic values
$||\varphi_1||^M,...,||\varphi_n||^M$ of 
the formulae $\varphi_1,...,\varphi_n$.
Our framework provides a general approach
to compositional operators.
Let us next consider operators 
of the type $(\tau,1,...,1)$, where
$\tau$ is no more necessarily the empty signature.
We have formulae of the type $\langle F\rangle(\varphi_1,...,\varphi_n)$.
Can we consider the semantics of our logic from the point of
view of Kripke semantics? We can indeed.
Let $\mM$ be a $\tau$-model with the domain $W$.
Call $M := \bigl(\mM,\{P_i\}_{i\in I}\bigr)$.
Our semantics dictates that $(M,w)\models \langle F \rangle(\varphi_1,...,\varphi_n)$
iff 
$$w\in F\bigl(\, \mM,||\varphi_1||^M,...,||\varphi_n||^M\, \bigr).$$
Let $X \subseteq W$.
Define
\begin{enumerate}
\item
$M,X\models P_i$ iff $X =  P_i$.
\item
$M,X\models \langle F \rangle(\varphi_1,...,\varphi_n)$
if and only if there exist sets $Y_1,...,Y_n\subseteq W$
such that $X = F(\mM,Y_1,...,Y_n)$ and $M,Y_i\models\varphi_i$ for each $i$.\footnote{Notice
that here this condition means simply that $M,X\models \langle F \rangle(\varphi_1,...,\varphi_n)$
if and only if $X = F(\mM,||\varphi_1||^M,...,||\varphi_n||^M)$.}
\end{enumerate} 
We call this the \emph{canonical lift} of the logic to the level of teams.
We can now (consistently)
redefine the satisfaction of $\langle F \rangle(\varphi_1,...,\varphi_n)$
such that $(M,w)\models \langle F \rangle(\varphi_1,...,\varphi_n)$
iff there exists some set $X\subseteq W$ such that $w\in X$
and $M,X\models \langle F \rangle(\varphi_1,...,\varphi_n)$.
Let $\mathcal{W}$ be the power set of $W$.
Define the $(n+1)$-ary relation $\mathcal{R}\subseteq\mathcal{W}^{(n+1)}$
such that $(X,Y_1,...,Y_n)\in\mathcal{R}$ iff $X = F( (\mM,Y_1,...,Y_n) )$.
Now define the model $\mathcal{M} = (\mathcal{W},\mathcal{R},\{\mathcal{S}_i\}_{i\in I})$,
where each $\mathcal{S}_i$ is
the set $\{\ P_i\ \}\subseteq\mathcal{W}$.
The model $\mathcal{M}$ is (essentially) a Kripke model
with the $(n+1)$-ary accessibility relation $\mathcal{R}$ and
proposition symbols $\mathcal{S}_i$. Consider pointed
models $(\mathcal{M},X)$, where $X\in\mathcal{W}$.
Following standard Kripke semantics of (polyadic) modal logic,
we make the following definition.
\begin{enumerate}
\item
$(\mathcal{M},X)\models \mathcal{S}_i$ iff $X_i\in\mathcal{S}_i$.
\item
$(\mathcal{M},X)\models \langle\mathcal{R}\rangle(\varphi_1,...,\varphi_n)$
iff there exist $Y_1,...,Y_n\in\mathcal{W}$ such that
$(X,Y_1,...,Y_n)\in\mathcal{R}$ and $(\mathcal{M},Y_i)\models\varphi_i$
for each $i$.
\end{enumerate}
Define a translation such that $P_i^* = \mathcal{S}_i$ 
and $$(\langle F \rangle(\varphi_1,...,\varphi_n))^*
= \langle\mathcal{R}\rangle(\varphi_1^*,...,\varphi_n^*).$$
We have $M,X\models P_i$
iff $(\mathcal{M},X)\models\mathcal{S}_i$
and $M,X\models \langle F \rangle(\varphi_1,...,\varphi_n)$
iff $(\mathcal{M},X)\models\langle\mathcal{R}\rangle(\varphi_1^*,...,\varphi_n^*)$.
Every operator $F$ gives rise to an accessibility relation $\mathcal{R}$ in the
canonical way defined above, and each unary predicate $P_i$
gives rise to the related unary predicate $\mathcal{S}_i$.
This way we \emph{lift} the general compositional
semantics to the realm of Kripke semantics.
This way typical compositional frameworks can be viewed
from the point of view of Kripke semantics.\footnote{Notice
that various logical equivalence-related similarity relations can be nicely lifted
to suitable bisimulations. Note also that related logics can easily be
associated with natural recursive capacities (for example) in
the way done in \cite{kuusistocsl2013}. Now, the Kripke structures obtained in
this way may not contain all Kripke structures. But, nevertheless, it is intereting
that if all structures were in the class, genuine use of the recursive capacities
(non-locality) would necessarily
lead to indeterminacy of truth
values somewhere, as shown in \cite{kuusistodistributed}.
What then, suffices to guarantee determinacy? This is an intriguing question.}
For the sake of an example concerning the canonical lift,
let $F$ be defined such that $F(\mM,S,T) = S\, \cup\, T$, i.e.,
$F$ is the disjunction. Let $X\subseteq W$.
Then $M,X\models \langle F\rangle(\varphi_1,\varphi_2)$
iff there exist sets $Y_1,Y_2\subseteq W$ such
that $X = Y_1\cup Y_2$
and we have $M,Y_1\models\varphi_1$ and $M,Y_2\models\varphi_2$. 
This is the truth defintion for the disjunction in modal dependence logic.
For the sake of another example, let $G$ be defined such that $G(\mM,S) = W\setminus S$, i.e.,
$G$ is the negation. Let $X\subseteq W$.
Then $M,X\models \langle G\rangle\varphi$
iff there exists a set $Y\subseteq W$ such
that $X = W\setminus Y$,
and we have $M,Y\models\varphi$.
%

%
%

%
Now consider a logic with only monotone operations $F$, i.e., 
if we have $X_1\subseteq Y_1,\, X_2\subseteq Y_2,\, ..., X_k\subseteq Y_k$,
then $F(\mM,X_1,...,X_k)\subseteq F(\mM,Y_1,...,Y_k)$.
We make the following definition.
\begin{enumerate}
\item
$M,X\models P_i$ iff $X \subseteq  P_i$.
\item
$M,X\models \langle F \rangle(\varphi_1,...,\varphi_n)$
if and only if there exist sets $Y_1,...,Y_n\subseteq W$
such that $X \subseteq F(\mM,Y_1,...,Y_n)$ and $M,Y_i\models\varphi_i$ for each $i$.
%
%
%
%
%
\end{enumerate} 
We call this the \emph{monotone canonical lift} to the level of teams.
We can again (consistently)
redefine the satisfaction of $\langle F \rangle(\varphi_1,...,\varphi_n)$
such that $(M,w)\models \langle F \rangle(\varphi_1,...,\varphi_n)$
iff there exists some set $X\subseteq W$ such that $w\in X$
and $M,X\models \langle F \rangle(\varphi_1,...,\varphi_n)$.
(The consistency is easy to show by first noticing that 
the sets $Z$ such that $M,Z\models\chi$,
satisfy $Z\subseteq ||\chi||^M$.)
As above, we shall interpret this semantics in the
style of Kripke. Let $\mathcal{W}$ be the power set of $W$.
Define the $(n+1)$-ary relation $\mathcal{R}\subseteq\mathcal{W}^{(n+1)}$
such that $(X,Y_1,...,Y_n)\in\mathcal{R}$ iff $X \subseteq F( (\mM,Y_1,...,Y_n) )$.
Define the model $\mathcal{M} = (\mathcal{W},\mathcal{R},\{\mathcal{S}_i\}_{i\in I})$,
where each $\mathcal{S}_i$ is
this time the set $\{\ S\ |\ S\subseteq P_i\ \}\subseteq \mathcal{W}$.
Following Kripke semantics,
define $(\mathcal{M},X)\models\mathcal{S}_i$ iff $X\in\mathcal{S}_i$,
and also define $(\mathcal{M},X)\models\langle\mathcal{R}\rangle(\varphi_1,...,\varphi_n)$ if
there exists $Y_1,...,Y_n\in\mathcal{W}$ such that 
$(X,Y_1,...Y_n)\in\mathcal{R}$ and $(\mathcal{M},Y_i)\models\varphi_i$ for each $i$.
We have $M,X\models P_i$
iff $(\mathcal{M},X)\models\mathcal{S}_i$
and $M,X\models \langle F \rangle(\varphi_1,...,\varphi_n)$
iff $(\mathcal{M},X)\models\langle\mathcal{R}\rangle(\varphi_1^*,...,\varphi_n^*)$.
Thus we have again lifted the semantics from the general compositional 
treatment to a Kripke-style treatment.
Intuitively, ordinary Kripke-style treatment (of whatever) involves
\emph{searching for witnesses} in order to satisfy a
diamond formula. Much of the fundamentality of the
framework stems from this. The related function based
treatment on the power set level enables an
algebraic approach to the underlying Kripke-style approach.
And of course the power set level treatement can again
be turned into a treatment that resembles the style of Kripke by scanning the function
on subsets backwards (the new accessibility relation), and regarding subsets as points.
Such approaches are interesting even if the power set operator does not arise
from an ordinary accessibility relation, but is an arbitrary function on subsets.
Above, a  natural intuition behind the team level satisfaction of
formulae in the setting \emph{without} the assumption monotonicity
is that the team is \emph{exactly} the set of points that satisfy
the formula. With the monotonicity assumption, a natural
intuition is that a team satisfies a formula if each member of
the team does. In a sense the double team
semantics relates to \emph{both} of these intuitions.
It is of course interesting to add further generalized operators
to the setting we have defined. For example,
as in \cite{kuusisto3},
we can consider operators $\mathcal{F}$
that map any tuple $(\mM,T_1,...,T_n)$,
to a set $S\subseteq \mathcal{W}$;
here $T_i\subseteq\mathcal{W} = \mathcal{P}(W)$ for each $i$.
Of course if $f:W\rightarrow U$
is an isomorphism from $(\mM,T_1,...,T_n)$
to $(\mathfrak{N},S_1,...,S_n)$, then
$f( \mathcal{F}(\mM,T_1,...,T_n) ) = \mathcal{F}(\mathfrak{N},S_1,...,S_n)$;
here
$$f( \mathcal{F}(\mM,T_1,...,T_n) ) = \{\ f(S)\ |\ S\in\mathcal{F}(\mM,T_1,...,T_n)\ \},$$
where $f(S) = \{ \ f(s)\ |\ s\in S\ \}$. 
We can also let $\mathcal{F}$ be nullary.
Then $\mathcal{F}$ simply maps $\mM$ to a subset of $\mathcal{W}$
(and is obviously invariant under isomorphisms).
Consider formulae of the type
$\langle \mathcal{F}\rangle(\varphi_1,...,\varphi_n)$.
Define the semantics such that 
$M,X\models \langle \mathcal{F}\rangle(\varphi_1,...,\varphi_n)$
if and only if $X\in\mathcal{F}(\mM,|||\varphi_1|||^M,...,|||\varphi_n|||^M)$,
where $|||\varphi_i|||^{\mathcal{M}} = \{\ X\subseteq W\ |\ M,X\models\varphi_i\, \}$.
If $\mathcal{F}$ is nullary,
then $M,X\models \langle \mathcal{F}\rangle$
iff $X\in\mathcal{F}(\mM)$.
We can now define the global disjunction $\mathcal{V}$
such that $M,X\models\langle\mathcal{V}\rangle(\varphi,\psi)$
iff $X\in |||\varphi|||^M\cup\, |||\psi|||^M$,
and the global negation $\mathcal{N}$ such that 
$M,X\models\langle\mathcal{N}\rangle\varphi$
iff $X \in \mathcal{W}\setminus |||\varphi|||^M$
(recall that $\mathcal{W} = \mathcal{P}(W)$, and
$W$ is the domain of $\mM$ and $M$).
Of course we can also add higher order propositions
$\mathcal{Q}\subseteq\mathcal{W}$ to $M$,
if we wish, and then obviously $M,X\models\mathcal{Q}$
iff $X \in \mathcal{Q}$. We did not consider generalized operators
of this level in the principal sections of this article, mainly for 
the sake of simplicity, but also because we
wanted to consider systems where one
reasons \emph{with} teams rather
than \emph{about} teams.
A very large class of operators $F$ satisfies
the requirements of the framework we have discussed above, and the
approach is general indeed. Let us consider an example in
the spirit of cylindric set algebras.
Recall that the set of all variable symbols is $\mathrm{VAR} = \{\, v_i\ |\ i\in\mathbb{Z}_+\ \}$.
Let $\mA$ be a first-order model
whose vocabulary $\tau$ consists of relation symbols.
Let $A^{\omega}$ denote the set of all $\omega$-sequences
of elements of $A$; $\omega$ is of course the smallest infinite ordinal.
Let $R\in\tau$ be a $k$-ary relation symbol.\footnote{The equality
symbol can be treated as if it was a relation symbol.}
Let $(v_{i_1},...,v_{i_k})$ be a tuple of variable symbols.
Define $P_{R(v_{i_1},...,v_{i_k})}\subseteq A^{\omega}$
to be the relation $T\subseteq A^{\omega}$ such
that $\overline{a}\in A$ is in $T$ if and only if the following conditions hold.
\begin{enumerate}
\item
There exists an assignment $s$
such that $\mA,s\models_{\mathrm{FO}}R(v_{i_1},...,v_{i_k})$.
\item
We have $a(i) = s(v_i)$ for each $i\in \{i_1,...,i_k\}$.
\end{enumerate}
Let $\mathcal{A}$ be the set of all atomic first-order formulae
of the vocabulary $\tau$.
Define the model $M_{\mathfrak{A}} = 
\bigl(\mA,\{P_{\varphi}\}_{\varphi\in\mathcal{A}}\bigr)$.
Define an operator $F_{\exists v_i}$
of the type $(\tau,\omega,\omega)$ for each variable $v_i$ as follows.
Let $\mB$ be a $\tau$-model.
Let $B = \mathit{Dom}(\mB)$ and $S\subseteq B^{\omega}$.
For an $\omega$-sequence $s\in B^{\omega}$, $i\in\mathbb{N}$ and $b\in B$,
let $s[i\mapsto b]$ denote the $\omega$-sequence $t\in B^{\omega}$
such that $t(i) = b$ and $t(j) = s(j)$ for each $j\in\mathbb{N}\setminus\{i\}$.
Define
$$F_{\exists v_i}(\, (\mB,S)\, ) = \{\ s\in B^{\omega}\ \ |\
\ s[ i\mapsto b ]\in S\text{ for some }b\in B\ \ \}.$$
Define also the operators $F_{\neg}$ and $F_{\vee}$ such
that $F_{\neg}(\, (\mathfrak{B},S)\, ) = B^{\omega}\setminus S$
and $F_{\vee}(\, (\mathfrak{B},S,T)\, ) = S\cup T$.
Translate from first-order logic into modal logic as follows.
\begin{enumerate}
\item
$T(R(x_1,...,x_k) )= P_{R(x_1,...,x_k)}$,
\item
$T(\exists v_i\varphi )= \langle F_{\exists v_i}\rangle T(\varphi)$,
\item
$T(\neg \varphi )= \langle F_{\neg}\rangle T(\varphi)$,
\item
$T(\varphi\vee\psi)= \langle F_{\vee}\rangle( T(\varphi),T(\psi))$.
\end{enumerate}
Let $s$ be an assignment that maps to $A$, and let
$t\in A^{\omega}$. We say that
$t$ encodes $s$ if
for all variables $v_i$ in the domain of $s$,
we have $s(v_i) = t(i)$.
Let $s$ be an assignment and $t\in A^{\omega}$ a
sequence that encodes $s$.
Now of course $$\mathfrak{A},s\models_{\mathrm{FO}}\psi
\ \ \Leftrightarrow\ \ (\mA,\{P_{\varphi}\}_{\varphi\in\mathcal{A}},t)\models T(\psi).$$
We can of course perform lifts and all that for the obtained system.
Double team semantics is compositional, and thus we can of course
similarly modalize it also, if we wish.
For the sake of one more example,
let $W$ be a nonempy set 
and $R\subseteq W\times W$ a binary relation.
Let $P_i\subseteq W$, where $i\in I$,
be unary relations.
Consider the standard Kripke diamond
operation of the type $(\{R\},1,1)$ defined such that
$$F_{\Diamond}\bigl(\ (W,R),S \ \bigr)
= \{\ w\in W\ |\ \exists u\in W\text{ s.t. } w R u\text{ and }
u\in S\ \}.$$
Let $M = ((W,R),\{P_i\}_{i\in I})$.
%
%
%
%
%
%
The monotone lift dictates that $M,X\models \langle F_{\Diamond}\rangle  \varphi$
iff there is some $Y\subseteq W$ such that $X \subseteq F_{\Diamond}((W,R),Y)$
and $M,Y\models\varphi$. 
A whole new range of possibilities arise from considering generalized
operators that modify the underlying models.
Let $\mathcal{C}$  denote the class of isomorphism classes 
of ordinary pointed Kripke models. Let $F$ be a function (too large to be a set)
from $\mathcal{C}$ to $\mathcal{P}(\mathcal{C})$.
We may define $(M,w)\models(F)\varphi$
iff there exists a model $(N,v)\in \bigcup F([(M,w)])$ such that
$(N,v)\models\varphi$. Here $[(M,w)]$ is of course
the isomorphism class of $(M,w)$, and $\bigcup F([(M,w)])$
is the union of the classes in $F([(M,w)])$.
Similar operators can of course be defined for predicate logic.
These kinds of \emph{generalized modifiers} can
be annoyingly strong from the set theoretic perspective.
A rather tame such an operator, i.e., a modifier,
is employed in \cite{kuusistoooo}
in order to obtain a Turing complete logic $\mathcal{L}$ (see \cite{kuusistoooo}
for the syntax and semantics).
A possible reading of an $\mathcal{L}$-formula $\varphi$ 
states that \emph{it is possible to verify $\varphi$.}
Formula $\neg\varphi$ can be considered to 
state that \emph{it is possible to falsify $\varphi$},
or even that \emph{it is possible to disprove $\varphi$}.
Note that negation here is a strong negation; indeed,
since $\mathcal{L}$ captures recursive
enumerability, $\neg$ \emph{cannot be} the
contradictory negation.
We can define recursive readings $r$
of formulae of $\mathcal{L}$ as follows.
For first-order atoms, we let
\emph{$r(\varphi) := ``\varphi$}."
For atom $k$, where $k$ is a
natural number, we let $r(k) := $ \emph{``$k$."}
For $\neg$, we let $r(\neg\varphi) :=$ \emph{``it is 
\emph{falsifiable} that $r(\varphi)$"} or
\emph{``it is \emph{refutable} that $r(\varphi)$."}
For the conjunction, we
let $r(\varphi\wedge\psi)$ := ``\emph{$r(\varphi)$ and $r(\psi)$}."
For $\exists x$, 
we define that $r(\exists x\varphi) := $\emph{``there exists
an $x$ such that $r(\varphi)$."}
For $\mathrm{I} x$, 
we let $r(\mathrm{I} x\varphi) := $\emph{``it is possible to insert a
fresh $x$ such that $r(\varphi)$."}
Operators concerning the insertion and deletion
of tuples of relations can be given a similar reading.
For $k\varphi$, we let $r(k\varphi) :=$ \emph{``\emph{this} 
statement---call it $k$---holds and asserts that $r(\varphi)$,"}
or even \emph{``\emph{this} 
statement $k$ holds and asserts that $r(\varphi)$."} 
(Note that indeed the novel variables $k$ refer
to \emph{formulae} and can can be associated
with the word ``\emph{this}" (\emph{this} formula...).
Now, importantly, we read $\mathfrak{A},f\models^+\varphi$
as ``it is verifiable in $\mathfrak{A},f$ that $r(\varphi)$"
rather than ``it is true in $\mathfrak{A},f$ that $r(\varphi)$." 
Similarly, $\mathfrak{A},f\models^-\varphi$ is read 
as ``it is falsifiable in $\mathfrak{A},f$ that $r(\varphi)$."
On the interpretational level, verifiability of verifiability means verifiability,
and falsifiability of falsifiability means verifiability.
This is dictated by the formal semantics of $\mathcal{L}$. (So
the picture is in a sense more symmetric than with intuitionistic  logic,
and strongly corresponds with Turing computation,
with verification being halting in an accepting state and 
falsifiability halting in a rejecting state; indeterminacy corresponds to
divergence. The logic gives a unified perspective on logic and
computation, using games (game-theoretic semantics).
If considered a \emph{canonical logic}, it shows that the logical 
operators missing from first-order logic (in order to
be Turing-complete and thus capture systematicity) are
recursion and quantification of new points and tuples.)

The reading of the liar's sentence $1\neg 1$ (which is a
formula of $\mathcal{L}$) is
now (in a simplified form according to the reading $r$) ``this statement is
falsifiable," where falsifiability means falsifiability or refutability or even
disprovability (in a model) in the sense of $\mathcal{L}$.
The truth value of $1\neg 1$ in $\mathcal{L}$ is
\emph{indeterminate}. The
statement ``this statement is
falsifiable" is (in the sense of $\mathcal{L}$ and
even in some informal senses)
unproblematic (not paradoxical).
Even informally the sentence ``this
statement is falsifiable" can be considered unproblematic because one
cannot now deduce a contradiction in exactly
the same way as in the case of the classical
reading of the liar's sentence.
Call $\varphi := $ ``\emph{this statement is falsifiable}."
If we assume $\varphi$ is \emph{true}, we
obtain a contradiction, and then if we conclude
that therefore $\varphi$ is \emph{false}, we
\emph{cannot} infer that thus $\varphi$ is true.
We can simply infer that ``it is 
false that this sentence called $\varphi$ is falsifiable."
It does not (have to) imply that $\varphi$ is true.
All this is consistent with $\varphi$ being indeterminate.
So the usual problem does not arise.
Bivalence leads to problems with the liar's paradox.
Consider the rather classical
readings $r'$ for $\mathcal{L}$ such that $r'(\neg\varphi) :=$ ``\emph{not} $r'(\varphi)$,"
$r'(k\varphi) := $ ``\emph{this statement,
call it $k$, holds and asserts that $r'(\varphi)$}," 
and $r'(k) := $ "\emph{$k$ holds}."
With this reading, $1\neg 1$ gives the liar's paradox;
now $\neg$ is ``not" as opposed to asserting falsifiability.
With the reading $r$, the paradox does not really arise,
and the formal semantics of $\mathcal{L}$ dictates that the formula is simply 
indeterminate, as we discussed above. Thus the paradox is resolved
with the falsifiability reading
due to refusing to adopt the bivalent perspective and talking
about verification/falsifiaction rather than simply truth/falsity.
If the truth value of condition $k$ depends solely
on ``the truth value of condition $k$" 
(with $k$ uninterpreted), we do not have to
adopt a reductionist perspective that we can dig the
truth value of the statement $k$ from some foundational 
fully determined atomic layer of bivalent facts.
We have not \emph{defined} $k$,
and undefined statements are \emph{ontologically} indeterminate.
If the meaning of dsfsd is defined simply to
be the meaning of dsfsd, then we can refuse
that it has a conventional meaning that can be
understood directly or in some sufficiently clear reductionist fashion.
We have not defined what it is, and we could define it in any way we wish.
Thus it has no meaning.
Note that $\mathcal{L}$ is still bivalent on fragments of $\mathcal{L}$
where formulae are determinate, such as first-order logic.
Being Turing-complete, $\mathcal{L}$ \emph{captures}
the notion of \emph{systematicity} (which we here
\emph{equate} with the capacity of Turing machines,
partly because we lack any better definition for systematicity).
Arguably, \emph{systematicity} is exactly what \emph{logic} is,
meaning that the word (or the notion of)
logic means exactly systematicity.
But this position of course can be debated.
%

%
%
So, $\mathcal{L}$ is not bivalent, just like Turing machines 
are not.  The logic $\mathcal{L}$
takes seriously the \emph{perspective} that bivalence
breaks (or can naturally be considered to break) in the presence of indeterminacy.
But does the lack of bivalence stem from 
ontic or epistemic indeterminacy? One can  argue 
that $1\neg 1$ is indeterminate in some ontological sense,
whereas a formula of $\mathcal{L}$ that takes as inputs
that represent Turing machines (plus input) and tries
to classify whether the input halts, is indeterminate in an epistemic sense.
(It is easy to write a formula of $\mathcal{L}$ 
which attempts to classify Turing machines according
to halting as follows. The formula approximates the
assertion that the input to the formula (a model
encoding a Turing machine and an input to it) halts.
The formula is sound and complete for halting inputs,
always being verifiable on halting inputs, 
but the evaluation of the formula soundly stops  in a rejecting state only on
\emph{some} of the diverging inputs.
Thus the formula misses some false inputs
and the game goes on without an end there.
The formula is epistemically indeterminate on those models,
at least with respect to the desired meaning of the formula
that would fully classify halting and diverging (which is impossible). But of
course it is debatable whether the desired meaning of the
formula is the true meaning of the formula. We could dictate that it
is not and that the formula is
always indeterminate in an ontic way when it is indeterminate formally.)
Reductionist approaches are fine when we can
always reach bivalent atoms. But who is to say 
that I am forced to admit that either $\exists x(x\in x)$ is
true or that $\exists x(x\in x)$ is false? (Forget about ZFC here.)
Falsity here (in the above sentence) does not refer to the contradictory negation
of truth. As we discussed above, genuine indeterminacy about \emph{definitions}
concerning partially determined notions seems to 
appear here. Let $w$ and $w'$ be two possible worlds such
that $\exists x(x\in x)$ is true in $w$ and
false in $w'$. Then my model $\{w,w'\}$
does not satisfy $\exists x(x\in x)$. Whether
my model satisfies $\neg \exists x(x\in x)$,
depends on my reading of $\neg$. Indeed,
it is neither unnatural nor uncommon to read $\neg$ such that
$\neg\exists x(x\in x)$ means that
$\exists x(x\in x)$ is determinately false,
i.e., that $\exists x(x\in x)$ is
false in both $w$ and $w'$.\footnote{Such readings of $\neg$
seem to occur in contexts where natural language is used.
On the other hand, in informal situations, the possibility of
indeterminacy is very rarely assumed. The \emph{assumption of determinacy}
indeed anyway seems to lead to many kinds of seemingly paradoxical situations.}
Then my model does not satisfy $\neg \exists x(x\in x)$,
and thus there is a truth value gap.
\emph{Strong negation} is indeed a 
natural operator, which has natural
uses in contexts involving indeterminacy.
And indeterminacy itself can appear rather natural.
As already mentioned, we cannot decide the truth value
of condition $k$ if the truth value of condition depends solely on
``the truth value of $k$."
This happens in $k\neg k$ and $k k$. 
The negation in $\mathcal{L}$ is one
kind of a strong negation, and, indeed, it
arises naturally in $\mathcal{L}$ largely due to 
indeterminacy. Furthermore, as already mentioned,
negation in $\mathcal{L}$ cannot be the
contradictory negation.
A further direction that $\mathcal{L}$ could and should be 
developed concerns games. It would be interesting to
extend $\mathcal{L}$ by quantifiers $Q_p x$ for 
$p\in \mathcal{P}$, where $\mathcal{P}$ is a set
of players. This  would result in a multiplayer version of
$\mathcal{L}$, which currently has only the two players $\exists$
and $\forall$. (Indeed, Abramsky has suggested such quantifiers.)
Additionally, the operators that modify the model could also be associated
with more players. Obviously  the looping constructs would be kept in
the system. Each player would be associated with quantifiers that add
and delete domain points and tuples of relations.
Arguably, the system could be considered to
be a framework that (at least in some
reasonable sense) captures \emph{all} perfect information games,
or defines the notion of a game. Turing completeness would be part of
the argument claiming that in some sense \emph{all} games are captured.
Indeed, the system would be all about a group of agents
\emph{modifying a relational structure} together. Quantifiers $Q_ax$ would
simply color individual nodes. Similarly,  quantifiers that add tuples to,
say, a unary relation, would color subsets of the domain.
Other operators would add (delete) domain points and tuples
of relations of higher arities.
A framework
where relational structures are modified,
possibly infinitely long, is a rather general framework.
Our system would (and will) indeed be a very general and
flexible framework for modeling interaction.
Additionally, the system will nicely build on the logic $\mathcal{L}$, being, after all, a
rather direct generalization of $\mathcal{L}$.
The logic $\mathcal{L}$ unifies logic and computation using games;
the extension will unify, in a sense, logic, computation  and games.
To model concurrency, it could be interesting to allow for
simultaneous moves by the players, with conditions dictating what happens with
clashing move attempts. Also imperfect information could be added in one
way or another. 
Generalized modifiers facilitate the definition of the
rather natural and intriguing Turing complete logic $\mathcal{L}$, and surely
they also offer rather interesting and intriguing perspectives on logic.
Further possiblities concerning such operators
should be investigated. For example theories of arithmetic
that talk about classes of finite models, as opposed to
talking about the single model $(\mathbb{N},+,\cdot)$ (possibly  together
with its non-standard variants), would be interesting in this context.
In the spirit of graph theory, one would talk about, e.g., finite models
that represent initial segments of arithmetic. Also finite models encoding finite sets, of course,
would be interesting here.
Let us finish up by considering ordinary modal dependence logic
and its variants. A natural generalized version of the
modal dependence atom $=\hspace{-1pt}(p_1,...,p_k,q)$
is defined as follows. Let $(k_1,...,k_l)$ be a nonempty sequence of
positive integers. (We consistently ignore
the possibility of considering operators without explicit input objects.)
Let $Q$ be a generalized quantifier of the type $(1,...,1)$,
where $1$ is repeated $1+\Sigma_{i\in\{1,...,l\}}k_i$ times.
Consider a formula of the type $A_Q(\overline{p}_1,...,\overline{p}_l)$,
where $\overline{p}_i$ is a sequence of $k_i$ proposition symbols.
Define $M,X\models A_Q(\overline{p}_1,...,\overline{p}_l)$
iff $(W,X,||{\overline{p}_1}_1||^M,...,||{\overline{p}_l}_{k_l}||^M)\in Q$.
Here $W$ is the domain of $M$ and $X\subseteq W$ a team.
Obviously ${\overline{p}_i}_j$ denotes the $j$-th proposition
symbol of proposition symbol sequence $\overline{p}_i$,
and $||{\overline{p}_i}_j||^M\subseteq W$ is the set of possible worlds
where ${\overline{p}_i}_j$ is true in the classical sense.
We could define for example an idependence operation $I(p,q)$
such that $I(p,q)$ holds in a team iff 
the following holds: if $p$ and $q$ get some
truth values in a world $w$ in the team, then there is a world in the team where $p$
gets the same value as in $w$ and $q$ the opposite value as in $w$. And so on.
Ordinary dependence logic currently calls for
further investigation concerning interpretation.
$M,X\models p\vee q$ can be interpreted to state that
the statement $p\vee q$ holds in every possible world in the team $X$.
Also the formula $=\hspace{-4pt}(p,q)$
has a natural interpretation in a team.
%
%
%
%
%
%
But what does $=\hspace{-4pt}(p,q)\, \vee=\hspace{-4pt}(p,q)$
exactly mean? Indeed, it seems that putting \emph{together} dependence
atoms and the splitjunction is problematic. The connective $\vee$ is very intuitive 
in restriction to plain propositional logic, and dependence atoms are intuitive on their
on, but the combination of these is somewhat puzzling.
The formula $=\hspace{-4pt}(p,q)\, \vee=\hspace{-4pt}(p,q)$ is, indeed, a
validity, while its direct translation into natural language (with $\vee$ translated
to the word ``or") seems not to be. With 
$=\hspace{-4pt}(p,q)$ and 
$=\hspace{-4pt}(p,q)\, \vee=\hspace{-4pt}(p,q)$ not being equivalent,
$\varphi\vee\psi$ seems to be in general best translated into a statement that 
the possible situations split or divide  into cases such that in the first scenario
we have $\varphi$ and in the other scenario $\psi$.\footnote{
Ordinary dependence logic and IF logic also have a similar feature. 
The sentences $\forall x\forall y=\hspace{-4pt}(x,y)$
and $\forall x\forall y(=\hspace{-4pt}(x,y)\vee =\hspace{-4pt}(x,y))$
are not equivalent, and in a model
with  two elements, one satisfying $P$ and the other one not,
the formula $\forall x\bigl(\exists y/x (Px\leftrightarrow Py)
\vee \exists y/x (Px\leftrightarrow Py)\bigr)$ is true, while
the formula $\forall x\exists y/x (Px\leftrightarrow Py)$ is not.
From the point of  view of natural language, this  can be puzzling, at least if
$\vee$ is taken to translate into ``or."}
Also, the standard
interpretation of $\neg\hspace{-3pt}=\hspace{-4pt}(p,q)$ is
somewhat odd, being true iff the interpreting team is empty.
Thus it is possible to  construct sensible models and teams
where for example the formula $\neg\hspace{-3pt}=\hspace{-4pt}(P\not=\mathrm{NP},\
\text{It is raining})$ is \emph{not} true. Therefore, it is easy to
see that formulae cannot be directly translated into
natural language
(with $\vee$ translated to ``or" and $\neg$ to ``not", 
and with the atomic proposition symbols being suitably tame)
such that all the resulting natural language statements 
can be given an interpretation that exactly corresponds to
the formal semantics of the untranslated formulae.
Note that of course already in first-order logic, the  symbol $\vee$
is the inclusive disjunction, and thus it is very easy to see that already
the translations of first-order sentences to natural language can
be immediately claimed ambiguous. However,
in the case of first-order logic, the translations
have \emph{some sensible reading} that exactly corresponds to
the  formal semantics of first-order logic.

%
%
Let us consider an alternative approach to dependence
(in modal and propositional contexts) altogether. This approach has the  property
that the natural language translations (with $\vee$ and $\neg$
translated to ``or" and ``not," respectively) have \emph{some sensible reading} that 
corresponds to the formal semantics (as long as the interpretations of
the proposition symbols are
suitably tame).
Let us extend the syntax of ordinary propositional (or modal) logic by the formula 
construction rule $=\hspace{-4pt}(\varphi_1,...,\varphi_k,\psi)$.
Note that here we allow for the arbitrary nesting of the
dependence operator $=$. Let us interpret the logic
using ordinary Kripke models.
Let us define
that $M,w\models\ =\hspace{-4pt}(\varphi_1,...,\varphi_k,\psi)$
iff the set $X := \{\ w'\in W\ |\ wRw'\ \}$ of successors of $w$ satisfies the condition
$$\forall u,v\in X\bigl(\ \forall i(M,u\models \varphi_i\Leftrightarrow M,v\models \varphi_i)\ \Rightarrow\
(M,u\models \psi\Leftrightarrow M,v\models \psi)\ \bigr).$$
Here we simply relativise the old dependence condition
to the set of successors of $w$.
In the similiar spirit, $\Box\varphi$ holds at $w$ iff all
successors of $w$ satisfy $p$. Now $=\hspace{-4pt}(\varphi_1,...,\varphi_k,\psi)$
holds at $w$ if the set of all successors of $w$ satisfies the dependence condition.
The interpretation of this semantics is similar to the interpretation of $\Box$.
The set of \emph{possible worlds}, or situations, or whatever,
must satisfy something. Under this interpretation, dependence is
interpreted with respect to the same \emph{sets of worlds} as necessity and possibility.
Of course a different accessibility relation could be used, if desired, but in
several contexts it is natural to argue that the very same accessiblity
relation is appropriate.
Anyway, a \emph{set} of possible worlds is used for the interpretation,
as in ordinary modal dependence logic, but this time the set involved
is quite explicitly associated with the set of possible alternative situations.
Other operators can of course be
treated similarly. Examples of natural operators
include for example operators corresponding to independence
declarations, and for another example, operators asserting that most successors
(in the finite and in general) satisfy $p$, and thus capturing an approach to 
the notion of likelyhood.
The approach resembles Kripke semantics and talks
about possible worlds (or sets of possible worlds: it is 
natural to define, just like in Kripke semantics,
that $M\models\varphi$ iff $M,w\models\varphi$ for 
all $w$ in the domain of $M$).
Trivially, if we have $\Diamond$ and $=\hspace{-1mm}( )$ in the
language together with the Booleans, we (typically) would like
the formulae $=\hspace{-1mm}(\varphi,\psi)$
and $\neg\bigl(\Diamond(\varphi\wedge\psi)
\wedge\Diamond(\varphi\wedge\neg\psi)\bigr)
\wedge \neg\bigl(\Diamond(\neg\varphi\wedge\psi)
\wedge\Diamond(\neg\varphi\wedge\neg\psi)\bigr)$ to be equivalent.
(Modal dependence logic does not respect this.)
This new approach to modal logic with dependence declarations is
interpreted with pointed models $M,w$; models $M,X$,
where $X$ is a team, are not needed. The Boolean connectives
have their \emph{usual} meaning (no splitjunctions).
Just like ordinary modal logic, this framework is \emph{realist} in spirit (as opposed
to antirealist); the evaluation point $w$ is considered to be the \emph{actual world}.
(For example the formula $p\wedge\Box\neg p$ \emph{can} be interpreted to
mean that $p$ holds while it is \emph{conceived necessary} that $\neg p$.)
Interestingly, it seems rather natural to
interpret $=\hspace{-4pt}(p,q)\vee =\hspace{-4pt}(p,q)$
under this new semantics; in team semantics the formula
has a somewhat less natural meaning.
Recall that $M\models\chi$ iff  $M,w\models\chi$ for every $w$ in
the domain of $M$. Let $W$ be the domain of $M$.
The conditions $M\models\  =\hspace{-1.5mm}(p,q) \vee =\hspace{-1.5mm}(p,q)$
and $M\models \neg =\hspace{-1.5mm}(r,s)$ give a sensible interpretation to
the natural language statements corresponding to
$=\hspace{-1.5mm}(p,q) \vee =\hspace{-1.5mm}(p,q)$
and $\neg =\hspace{-1.5mm}(r,s)$. The interpretation via team
semantics (i.e., $\vee$ being the splitjunction and $\neg =\hspace{-1.5mm}(r,s)$ being
true only in the empty team, and the team under investigation
being the domain of $M$) is not natural.

Now consider the formula
$p\rightarrow\ =\hspace{-1.5mm}(r,s)$.
%
%
Let us consider the corresponding natural language assertions
``\emph{If }$p$ \emph{ then} $r$ \emph{determines whether} $s$."
For example, let $p$ assert that ``\emph{the road
is free of cops}, " $r$ that ``\emph{John takes the motorcycle}, "
and $s$ that ``\emph{John will arrive on time}. "
Consider a suitable Kripke model $M$ with the domain
containg a set of
possible worlds $w_1,...,w_5$ as follows.
\begin{enumerate}
\item
$M,w_1\models p\wedge r\wedge s$.
\item
$M,w_2\models p\wedge \neg r\wedge \neg s$.
\item
$M,w_1\models \neg p\wedge r\wedge s$.
\item
$M,w_1\models \neg p\wedge r\wedge\neg s$.
%
%
\item
$M,w_1\models \neg p\wedge \neg r\wedge \neg s$.
\end{enumerate}
Consider the accessibility relation $R$ to be the total binary relation
(universal relation); this is for many purposes the most natural choice.
Now, it is not the case that $M\models p\rightarrow\ =\hspace{-1.5mm}(r,s)$.
However, the interpretation of $p\rightarrow\ =\hspace{-1.5mm}(r,s)$
via Kripke semantics is still \emph{sensible}; it states that the formula is
valid in a model if  in every possible world where $p$ \emph{actually holds}, it is
\emph{conceived by the observer} that $r$ determines $s$.
However, a more natural reading of $p\rightarrow\ =\hspace{-1.5mm}(r,s)$
would be given if the implication was interpreted to be
the operator $\rightarrow>$ defined such that 
$M,w\models \varphi\rightarrow> \psi$
iff $M',w'\models\psi$ for all $w'\in\mathit{domain}(M')$, where $M'$ is
the submodel of $M$ containing exactly 
the worlds $w''$ of $M$ such that $M,w''\models\varphi$.
(We concentrate here primarily on the case where the accessibility relation is
always the total relation on the domain of the model, and thus not really even needed.
We let natural language statements $\chi$ correspond to $M\models\chi$
rather than $M,w\models\chi$.)
The operator $\rightarrow >$ gives nice interpretations
for some examples with nested implications
and a diamond, for example to 
($p\rightarrow > (q\rightarrow > \neg \Diamond r)$)
where the standard
strict implication ($\Box(p\rightarrow \Box(q\rightarrow \neg \Diamond r))$) can be
weird (even when we use the total accessibility relation).
(For example, if $x$ is greater than 3, then,
if $x$ is a prime, it is not possible that $x$ is even.)
The strict implication (with total accessibility) gives a bad reading.
To give a simpler example, let $s$ state that $x$ is odd 
(and let $p$ state that $x$ is greter than 3 and $q$ that $x$ is a prime number).
The implication chain ($p\rightarrow > (q\rightarrow > s)$) 
is more natural than the one with 
the strict implication ($\Box(p\rightarrow \Box(q\rightarrow s)$).
The intended model here of course corresponds to the collection of all possible
assignments of values to $x$ in $\mathbb{N}$.
(Other examples where standard
Kripke semantics with the total accessibility relation can
be considered unintuitive (with respect to some interpretations of $\Diamond$)
involve for example the possible scenario where $\Diamond(p
\wedge \Diamond \neg p)$ is valid in every point of a Kripke model.)
The implication $\varphi \rightarrow >\psi$ is similar to $\neg\varphi\vee\psi$ in
team semantics (modal dependence logic); now we let $\varphi$ be a standard
propositional logic formula. But with a lax $\vee$ one gets too
many possible worlds to the left hand side team; for example,
when evaluating $\neg p\vee\psi$,
we can split $V$ to $U$ and $U'$ such that $U\models\neg p$
and $U'\models\psi$, with $U'$ still containing some worlds that do not
satisfy $p$. Thus $\neg p$ is still true in some possible worlds in $U'$.
It is easy to construct examples where this does not go
well with the intuition of the statement that ``if $p$, then $\psi$" (use
modal statements in $\psi$, or dependence/independence statements, etc.).
But of course $\rightarrow >$ as such is probably sometimes rather weird as well.
Finally, yet one more alternative for $\rightarrow >$ could be 
defined, now in team semantics, such that $T\models (\varphi\rightarrow\psi)$
iff $S\models\psi$, where $S = \{ u\in T\ |\ \{u\}\models\varphi\}$.
Here $T$, $S$ and $\{u\}$ are teams. This connective is
faithful to the phenomenon that implications relativise (in one way or another) to
all those possible scenarios where the antecedent holds.
Diamonds are of course related to implications (see below), and
the related diamond here would be defined such that $T\models\Diamond\varphi$
iff $\{u\}\models\varphi$ for some $u\in T$.
A slight variant of $\rightarrow$ above could replace $S$
with $\{u\in W\ |\ \{u\}\models\varphi\}$, where $W$ is the
set of all valuations (also ones not in $T$).
This implications works here and there...
Possible worlds provide a fruitful approach to developing the semantics of
implications and other logical operators, but work remains to be done.
A rather natural  possibility is to let the nesting depth of (some) operators
dictate the meaning of a formula in the type hierarchy. 
For example the truth of $\Diamond =\hspace{-1mm}(p,q)$ can be determined
with respect to the set of teams conceived; a suitable team
satistying $=\hspace{-1mm}(p,q)$ is searched from there. 
A formula with nesting depth three involves sets of sets of teams.
(Conjunctions of formulae with different nesting depths may need adjusting,
depending on the formal details desired.)
Many operators, like the implication, possibility, dependencies, the word ``unless,"
etcetera, can be
modelled with more or less success in this way.
%
%
%
%

%
The logic with the Kripke-style 
reading of atoms $=\hspace{-1mm}(r,s)$ where the
accessibility relation is always the total binary relation,
is rather natural. 
It is of course a syntactically
closed logic (free nesting of dependence operators and no negation normal forms).
The  logic can, however, be simulated
(with the same expressivity on the level of models/teams) by a system based on
extended team semantics that has the following grammar.
$$\varphi ::= p\ |\ \neg\varphi\ |\ {\sim}\varphi\ |\
(\varphi_1\vee\varphi_2)\ |\ (\varphi_1\sqcup\varphi_2).$$
Here a team $W$ satisfies $p$ (or $W\models p$) iff $p$
is true at every point $w\in W$. The conncective $\vee$ is
the standard splitjunction and ${\sim}$ is the negation
such that $W\models{\sim}\varphi$ iff $W\not\models\varphi$.
The connective $\sqcup$ is the disjunction such that
$W\models\varphi\sqcup\psi$ iff we have $W\models\varphi$ or $W\models\psi$.
The novel negation $\neg$ can be interpreted such  that
$W\models\neg \varphi$ iff for every $w\in W$ we have $\{w\}\not\models\varphi$.
Intuitively this kind of a negation can be considered to occur for example in
the assertion that ``the days were not rainy," or that ``it was never rainy"
(rather than that ``it is not the case that the days were rainy")
and similar inner negations. Thus both disjunctions and both negations
have natural uses in natural language. (The splitjunction intuitively
states that each world satisfies at least one of  the disjuncts, such as in
``it was raining or shining," rather than ``it was raining or it was shining,"
where in both cases the talk is about multiple days.
The splitjunction corresponds to the the inner reading mode.)
Repeating an argument from \cite{kuusisto3, kuusisto6}, the principle of 
excluded middle can be investigated easily with this logic. Consider the
natural language counterpart of the statement $x=5\ \ \dot{\vee}\ \  \dot{\neg}{x=5}$
where $\dot{\neg} x= 5$ is to be read ``$x$ is not $5$" while $\dot{\vee}$ is to be
read ``or." Now, it is easy to argue that in a reasonable sense ``$\dot{\neg} x = 5$"
does not hold, as we assume $x$ is \emph{not defined} in any
way.\footnote{Indeed, the
excluded middle is very interesting in the context of undefined (or 
not fully defined) statements.} Concluding from
this that ``$x=5$``  holds would of course be unjustified;
the claim ``$x=5$`` can also be argued---rather reasonably---not to hold.
Thus the claim ``$\dot{\neg}\, x = 5\ \dot{\vee}\ x = 5$" is not justified
and arguably does not to hold.
To fully understand the reasoning above, define the set of all 
possible worlds to be the set $W$ of all
functions from an infinite set of variables into $\mathbb{N}$, i.e., $W$ is
the set of variable assignments. We consider
there to be no actual world; $x$ is simply undefined.
We repeat the above reasoning: ``$\dot{\neg} x = 5$" is
interpreted to stand for $W\models\dot{\neg} x = 5$,\ \ \ ``$x = 5$"  is
interpreted to stand for $W\models x = 5$,\ \ \ and 
the claim ``$x=5\ \ \dot{\vee}\ \  \dot{\neg}{x=5}$" is
interpreted to mean $W\models (x=5\ \ \sqcup \ \  \dot{\neg}{x=5})$, which is false.
Note that we can reasonably interpret ``$\dot{\neg} x = 5$" to mean that $W\models \neg x=5$.
Alltogether, ``$x=5\ \ \dot{\vee}\ \  \dot{\neg}{x=5}$" then
claims that $W\models x=5\, \sqcup\, \neg x= 5$, which is indeed false despite the
natural language reading we gave it. Of course
the reading $W\models x=5 \vee \neg x= 5$ gives a true statement,
and this reading also makes good sense.
Other readings of the natural language statement are
also possible, such as $W\models x=5\, \sqcup\, \sim x= 5$,
which is of course also true because $W\models\ \ \sim x=5$.
There is nothing very
nontrivial about all this (before or after analysis).
If the set of proposition symbols considered is
finite, this logic can define all sets of teams (when repetitions of equivalent 
propositional assignments are ignored and the empty team is
included in all definable sets), just like S5 (universal modality) or the version of S5 with
the dependence operator instead of a diamond. Dependence and independence (etc.)
operators can be added. (Also the modal logic version can accommodate independence
atoms etc.). If desired, $\neg$ can also be defined such that $W\models\neg\varphi$
iff for all nonempty subsets $U$ of $W$, we have $U\not\models\varphi$; then another
interesting system arises. These negations concern
concepts such as ``never," with possibly a non-temporal reading
corresponding to \emph{impossibility}.
The second reading of  $\neg\varphi$ is related to it never becoming
conceivable that $\varphi$, while the first reading is related to
the assertion  that $\varphi$
simply never can be true. These could
be characterized as the \emph{subjective} and \emph{objective} reading.
It is worth pointing out that $\neg$ (with
both readings) is related to the denial of the possibility that $\varphi$.
Thus one can define, for example, $W\models\Diamond\varphi$
iff for some $w\in W$, we have $\{w\}\models\varphi$.
Notice that now the empty team does not satisfy $\Diamond\varphi$.
The next obviously natural operator has been first published in the 
first-order context in a paper of Galliani. R\"{o}nnholm
has defined and studied the variant  of it 
(also in the propositional and
modal context) that forces
truth in the empty team. We define $W\models\Diamond'\varphi$ 
iff for some nonempty $U\subseteq W$, we have $U\models\varphi$.
We do not bother ourselves now with
the issue of the empty team, even though it is
important. Instead, we note that \emph{implication
and possibility are intuitively interrelated} such that
for the operators $\Diamond$,$\rightarrow$,$\neg$,$\wedge$,$\bot$
(where these operators are not associated with a formal semantics)
one would typically like to have a semantics where
$\varphi\rightarrow\psi$ is equivalent to 
$\neg\Diamond(\varphi\wedge \neg\psi)$,
and furthermore, $\Diamond\varphi$ is 
equivalent to $\neg(\varphi\rightarrow\bot)$.
This of course suggests (for example)
the implication such that $W\models\varphi\rightarrow\psi$
iff ($\{w\}\models\varphi\Rightarrow\{v\}\models\psi)$ for all $u,v\in W$.
Obviously also the necessity operator $W\models\Box\varphi$
iff $\{u\}\models\varphi$ for all $u\in W$ arises naturally.
Finally, it is worth noting that while $\Diamond'$ is cool, it is of course not
without interesting exotic features,
because for example $\Diamond' p \wedge \Diamond' ( \sim \Diamond' p)$ is
easily satisfiable.
Antirealist approaches are often very natural.
A set of possible worlds (a team) can be interpreted to
be a \emph{possible perspective} in the following sense.
Consider a team $\{w,w'\}$ with two worlds satisfying exactly the same
propositions, with the exception of $p$; assume that $w$ satisfies $p$
while $w'$ does not. Assume that we are in some sense genuinely free 
to define whether $p$ holds or not. For example, $p$ could
state that $\exists x(x\in x)$. (Forget about ZFC here.)
Then $\{w,w'\}$, rather than $w$ or $w'$,
corresponds to our intuitive perspective.
One could add to the framework of possible perspectives
also a team of \emph{forbidden worlds}. The
pair of teams containing a set of possible worlds and a
set of impossible worlds would then be some kind of a
perspective on reality. The two sets 
would not have to exhaust the
space of all worlds. (This would depend on further
interpretational issues.)
Impossible worlds can correspond to
paraconsistent belief sets. Extended truth value sets
are not much different from complex numbers, and
thus they are rather tame objects really when correctly 
conseptualised. Just like it is possible to
measure (in a sense) a negative number,
for example the weight of a piece of gold corresponding to debt, it
possible to measure complex numbers corresponding
to for example the length of one side of a rectangular thin plate of owed gold;
if the area is understood to be $-1$ (due to it corresponding to debt),
the side can be understood to correspond to $i$.
Developing this approach further,
one could  begin, for example, with higher order propositions $D(\varphi)$,
where $\varphi$ is a formula of ordinary propositional logic.
A team $X$ would satisfy $D(\varphi)$ iff
$X$ satisfies $\varphi$ in the sense of team semantics.
The formula $D(\varphi)$ would read that $\varphi$ is
\emph{determinately true}. One could
then use ordinary Boolean logic with this set of higher order
 propositions. Of course the higher order propositions would
not have to be where the type hierarchy stops. One could talk not
only about determinacy (etc.) of primitive statements,
but also about determinacy (etc.) of propositions 
talking about determinacy, and so on ad infinitum.
Dependence would be an interesting extra ingredient (possibly a fundamental one)
in this world of different senses of the excluded middle,
different modes of negation, etc.
There is a trivial Galois connection 
between syntax and semantics.
Let us use it to investigate the meaning of
connectives in team semantics
(in a somewhat rough and \emph{sketchy} way for now).
Let $\mathrm{VAR}$ define the countably infinite
set of first-order variables used in first-order logic.
Fix a vocabulary $V$ and let $\mathcal{C}$
be a class pairs $(M,f)$, where $M$ is a $V$-model and $f$ a
function  $f:\mathrm{VAR}\rightarrow\mathit{Dom}(M)$. A
team could be any subclass of $\mathcal{C}$. We could
now define systems of semantics for these kinds of teams
and suitable logics, but let us do something a bit more particular.
Let $L$ be a fragment of first-order logic.
For each class $S\subseteq L$ of formulae, let 
$t(S)$ be the class $\{ \chi\in L\ |\, S\models_{\mathcal{C}}\chi\ \}$. (Note
that $t(S)$ generally contains open formulae as
well as sentences.) For each class $\mathcal{D}\subseteq\mathcal{C}$,
let $c(\mathcal{D})$ be the class $\{\ M\in\mathcal{C}\ |\ M\models\chi
\text{ for all }\chi\in t(\mathcal{D})\ \}$,
where $t(\mathcal{D}) =  \{\ \chi\in L
\ |\ M\models\chi\text{ for all }M\in\mathcal{D}\ \}$.
Define $\mathfrak{U}$ to be the structure  $(U,\subseteq)$, whose domain $U$ is
the class $\{ c(\mathcal{D})\ |\ \mathcal{D}\subseteq\mathcal{C}\ \}$.
Define also the structure $\mathfrak{T} = (U^d,\subseteq)$,
where $U^d$ is the class $\{\ t(S)\ |\ S\subseteq L\ \}$.
(The contradictory theory and empty class are there.)
%

%
%

%
We then define a semantics for first-order formulae (ignoring quantifiers 
and negation for the time being
for the sake of simplicity). Let $u\in U$ and $\chi\in L$.
Define $(U,\subseteq),u\Vdash\chi$
iff we have $M\models\chi$ for all $M\in u$.
Similarly, let $T\in U^d$, and define $(U^d,\subseteq),T\Vdash\chi$
iff we have $T\models\chi$.
We note that $(U,\subseteq),u\Vdash\chi\vee\psi$
iff there exist $u_1,u_2\in U$ such that
$(U,\subseteq),u_1\Vdash\chi$
and $(U,\subseteq),u_2\Vdash\psi$
and $u \subseteq u_1\cup u_2$.
Similarly, we 
have $(U^d,\subseteq),T\Vdash\chi\vee\psi$
iff there are $T_1,T_2\in U^d$
such that $(U^d,\subseteq),T_1\Vdash\chi $
and $(U^d,\subseteq),T_2\Vdash\psi$
and $T\supseteq t(T_1\cap T_2)$.
Thus there is an obvious choice (concerning formalism)
depending on whether it is desired that information sets (theories) grow or
classes of possible models shrink. Nice modal operators
(quite like in intuitionistic logic and S4) can be defined
and (meta)logical connectives (contradictory negation,
and all that discussed above) added. Forcing systems 
and lattice-based approaches indeed
have a lot of (rather simple and obvious) explanatory power here. Generalized atoms
(in addition to basic modalities)
provide an especially interesting (meta)logical ingredient
that could be added. An obvious collection of modalities to be added are
given by operators such as $\langle\psi\rangle$
which adds $\psi$ to the theory considered and moves to the
obtained theory. Any nice new axioms for ZFC? How about 
$\exists x \exists y \exists z( x\in y\wedge y\in z\wedge x\not\in z)$,
justified by $\emptyset$, $\{\emptyset\}$ and $\{\{\emptyset\}\}$ of course,
and why not $\forall x \exists y \exists z(x\in y\wedge y\in z \wedge x\not\in z)$,
justified by $X,\{X\},\{\{X\}\}$.
\section{Concluding remarks}
We have defined the notions of a generalized atom
and minor quantifier, and we have shown how these notions can be used when defining
extensions and variants of dependence logic.
We have seen that double team semantics can accommodate such
extensions and variants under the same umbrella framework in a natural way.
We have established that double team semantics has a natural game-theoretic counterpart
and discussed issues related to interpretation of logics based on team semantics.
We have put double team semantics into use by
defining the extension $\mathrm{DC}^2$ of $\mathrm{D}^2$ with counting
quantifiers. We have shown that the satisfiability and finite satisfiability problems of $\mathrm{DC}^2$
are complete for $\mathrm{NEXPTIME}$.
Obvious interesting future questions involve investigating logics
that mix different minor quantifiers and generalized atoms.
It will also be interesting to see \emph{how natural}
generalized atoms are in logical investigations.
Phenomena that appear strange arise easily in logics
that belong to the family of independence-friendly logic, often 
because technical operators are
%
%
carelessly associated with intuitions
that arise from the use of the same symbols in first-order logic.
Signaling (see \cite{mann}) is an
example of such a phenomenon.
It remains to be investigated what kinds
of systems embeddable in the double team semantics
are natural, and up to what extent.  
%
%
%
For example the notion of negation 
calls for further analysis in this context.
%
%
%
%
%
%

%
%

%
There already exists a wide range of papers 
on logics based on team semantics.
Subtle changes in semantic choices, such as using
the lax existential quantifier instead of the strict one, lead to logics with different
expressivities. To understand 
related phenomena better, it definitely makes sense
to study systems based on team semantics in a
unified framework. The double team semantics
aims to provide such a framework.
%

%

%

%
%

%
\appendix
\section{Formulae for the translation $\mathrm{DC^2}\rightarrow\mathrm{FOC^2}$}
%

%
%

%
%

%
\subsection{Formulae for $\chi = \exists^{\geq k} x\, \alpha$.}
$\mathit{Dom}(\chi) = \{x,y\}$:
%
%
%
%
%
%
\begin{align*}
&\psi_{\chi}^1\ := \ \forall x\forall y\bigl(\ S_{\chi}(x,y)\
\rightarrow\ \exists^{\geq k} x\, E_{\alpha}^{\mathit{Uf}}(x,y)\, \bigr),\\
%
%
%
&\psi_{\chi}^2\ := \ \forall x\forall y\bigl(\, E_{\alpha}^{\Uf}(x,y)\
\rightarrow\ \exists x\, S_{\chi}(x,y)\, \bigr),\\
&\psi_{\chi}^3\ := \ \forall x\forall y\bigl(\ T_{\chi}(x,y)\
\rightarrow\ \neg\exists^{\geq k} x\, \neg E_{\alpha}^{\Vgp}(x,y)\, \bigr),\\
&\psi_{\chi}^4\ := \ \forall x\forall y\bigl(\, E_{\alpha}^{\Vgp}(x,y)\
\rightarrow\ \exists x\, T_{\chi}(x,y)\, \bigr),\\
&\psi_{\chi}^5\ := \forall x \forall y\bigl( S_{\alpha}(x,y)\
\leftrightarrow\ E_{\alpha}^{\Uf}(x,y)\bigr),\\
&\psi_{\chi}^6\ := \forall x \forall y\bigl(\, T_{\alpha}(x,y)\
\leftrightarrow\  E_{\alpha}^{\Vgp}(x,y)\, \bigr).
\end{align*}
$\mathit{Dom}(\chi)$ is either of the sets $\{x\},\ \emptyset$:
%
%
%
%
\begin{align*}
&\psi_{\chi}^1\ := \ \exists x\, S_{\chi}(x)\
\rightarrow\ \exists^{\geq k} x\, E_{\alpha}^{\mathit{Uf}}(x),\\
%
%
%
&\psi_{\chi}^2\ := \ \exists x\, E_{\alpha}^{\Uf}(x)\
\rightarrow\ \exists x\, S_{\chi}(x),\\
&\psi_{\chi}^3\ := \ \exists x\, T_{\chi}(x)\
\rightarrow\ \neg\exists^{\geq k} x\, \neg E_{\alpha}^{\Vgp}(x),\\
&\psi_{\chi}^4\ := \ \exists x E_{\alpha}^{\Vgp}(x)\
\rightarrow\ \exists x\, T_{\chi}(x),\\
&\psi_{\chi}^5\ :=\ \forall x \bigl(\, S_{\alpha}(x)\
\leftrightarrow\ E_{\alpha}^{\Uf}(x)\, \bigr),\\
&\psi_{\chi}^6\ :=\ \forall x \bigl(\, T_{\alpha}(x)\
\leftrightarrow\ E_{\alpha}^{\Vgp}(x)\, \bigr).
\end{align*}
$\mathit{Dom}(\chi)$ is $\{y\}$:
%
%
%
%
\begin{align*}
&\psi_{\chi}^1\ := \ \forall y\bigl(\ S_{\chi}(y)\
\rightarrow\ \exists^{\geq k} x\, E_{\alpha}^{\mathit{Uf}}(x,y)\, \bigr),\\
%
%
%
&\psi_{\chi}^2\ := \ \forall x\forall y\bigl(\, E_{\alpha}^{\Uf}(x,y)\
\rightarrow\ S_{\chi}(y)\, \bigr),\\
&\psi_{\chi}^3\ := \ \forall y \bigl(\ T_{\chi}(y)\
\rightarrow\ \neg\exists^{\geq k} x\, \neg E_{\alpha}^{\Vgp}(x,y)\, \bigr),\\
&\psi_{\chi}^4\ := \ \forall x\forall y\bigl(\, E_{\alpha}^{\Vgp}(x,y)\
\rightarrow\ T_{\chi}(y)\, \bigr),\\
&\psi_{\chi}^5\ := \forall x \forall y\bigl( S_{\alpha}(x,y)\
\leftrightarrow\ E_{\alpha}^{\Uf}(x,y)\, \bigr),\\
&\psi_{\chi}^6\ := \forall x \forall y\bigl(\, T_{\alpha}(x,y)\
\leftrightarrow\  E_{\alpha}^{\Vgp}(x,y)\, \bigr).
\end{align*}
%
%
%
%
%
%
%
%
%
\subsection{Formulae for $\chi = \exists^{\geq k} y\, \alpha$.}
$\mathit{Dom}(\chi) = \{x,y\}$:
%
%
%
%
%
%
\begin{align*}
&\psi_{\chi}^1\ := \ \forall x\forall y\bigl(\ S_{\chi}(x,y)\
\rightarrow\ \exists^{\geq k} y\, E_{\alpha}^{\mathit{Uf}}(x,y)\, \bigr),\\
%
%
%
&\psi_{\chi}^2\ := \ \forall x\forall y\bigl(\, E_{\alpha}^{\Uf}(x,y)\
\rightarrow\ \exists y\, S_{\chi}(x,y)\, \bigr),\\
&\psi_{\chi}^3\ := \ \forall x\forall y\bigl(\ T_{\chi}(x,y)\
\rightarrow\ \neg\exists^{\geq k} y\, \neg E_{\alpha}^{\Vgp}(x,y)\, \bigr),\\
&\psi_{\chi}^4\ := \ \forall x\forall y\bigl(\,  E_{\alpha}^{\Vgp}(x,y)\
\rightarrow\ \exists y\, T_{\chi}(x,y)\, \bigr),\\
&\psi_{\chi}^5\ := \forall x \forall y\bigl(\, S_{\alpha}(x,y)\
\leftrightarrow\ E_{\alpha}^{\Uf}(x,y)\, \bigr),\\
&\psi_{\chi}^6\ := \forall x \forall y\bigl(\, T_{\alpha}(x,y)\
\leftrightarrow\  E_{\alpha}^{\Vgp}(x,y)\, \bigr).
\end{align*}
If $\mathit{Dom}(\chi)$ is either of the sets $\{y\}, \emptyset$, exactly
the same formulae $\psi_{\chi}^1,...,\psi_{\chi}^6$ are used as in the
case where $\chi = \exists^{\geq k} x\, \alpha$ and $\mathit{Dom}(\chi)$ is $\{x\}$ or $\emptyset$.
\vspace{3mm}
\noindent
$\mathit{Dom}(\chi)$ is $\{x\}$:
%
%
%
%
\begin{align*}
&\psi_{\chi}^1\ := \ \forall x\bigl(\ S_{\chi}(x)\
\rightarrow\ \exists^{\geq k} y\, E_{\alpha}^{\mathit{Uf}}(x,y)\, \bigr),\\
&\psi_{\chi}^2\ := \ \forall x\forall y\bigl(\, E_{\alpha}^{\Uf}(x,y)\
\rightarrow\ S_{\chi}(x)\, \bigr),\\
&\psi_{\chi}^3\ := \ \forall x\bigl(\ T_{\chi}(x)\
\rightarrow\ \neg\exists^{\geq k} y\, \neg E_{\alpha}^{\Vgp}(x,y)\, \bigr),\\
&\psi_{\chi}^4\ := \ \forall x\forall y\bigl(\ E_{\alpha}^{\Vgp}(x,y)\
\rightarrow\ T_{\chi}(x)\, \bigr),\\
&\psi_{\chi}^5\ := \forall x \forall y\bigl( S_{\alpha}(x,y)\
\leftrightarrow\ E_{\alpha}^{\Uf}(x,y)\, \bigr),\\
&\psi_{\chi}^6\ := \forall x \forall y\bigl(\, T_{\alpha}(x,y)\
\leftrightarrow\  E_{\alpha}^{\Vgp}(x,y)\, \bigr).
\end{align*}
\subsection{Formulae for $\chi = \exists^{s} x\, \alpha$.}
Below we let $\exists^{= 1} x\ \psi$ denote
the ($\mathrm{FOC}^2$-expressible) condition
that exactly one $x$ satisfies $\psi$.\\
\vspace{1mm}
\noindent
$\mathit{Dom}(\chi) = \{x,y\}$:
%
%
%
%
%
%
\begin{align*}
&\psi_{\chi}^1\ := \ \forall x\forall y\bigl(\ S_{\chi}(x,y)\
\rightarrow\ \exists^{= 1} x\, E_{\alpha}^{\mathit{Uf}}(x,y)\, \bigr),\\
%
%
%
&\psi_{\chi}^2\ := \ \forall x\forall y\bigl(\, E_{\alpha}^{\Uf}(x,y)\
\rightarrow\ \exists x\, S_{\chi}(x,y)\, \bigr),\\
&\psi_{\chi}^3\ :=\ \forall x\forall y\bigl(\, T_{\chi}(x,y)
\rightarrow \forall x E_{\alpha}^{\Vgp}(x,y)\bigr),\\
&\psi_{\chi}^4\ := \ \forall x\forall y\bigl(\, E_{\alpha}^{\Vgp}(x,y)\
\rightarrow\ \exists x\, T_{\chi}(x,y)\, \bigr),\\
&\psi_{\chi}^5\ := \forall x \forall y\bigl(\, S_{\alpha}(x,y)\
\leftrightarrow\ E_{\alpha}^{\Uf}(x,y)\, \bigr),\\
&\psi_{\chi}^6\ := \forall x \forall y\bigl(\, T_{\alpha}(x,y)\
\leftrightarrow\  E_{\alpha}^{\Vgp}(x,y)\, \bigr).
\end{align*}
$\mathit{Dom}(\chi)$ is either of the sets $\{x\},\ \emptyset$:
%
%
%
%
\begin{align*}
&\psi_{\chi}^1\ := \ \exists x\, S_{\chi}(x)\
\rightarrow\ \exists^{= 1} x\, E_{\alpha}^{\mathit{Uf}}(x),\\
&\psi_{\chi}^2\ := \ \exists x\, E_{\alpha}^{\Uf}(x)\
\rightarrow\ \exists x\, S_{\chi}(x),\\
&\psi_{\chi}^3\ :=\ \exists x\, T_{\chi}(x)
\rightarrow \forall x E_{\alpha}^{\Vgp}(x),\\
&\psi_{\chi}^4\ := \ \exists x E_{\alpha}^{\Vgp}(x)\
\rightarrow\ \exists x\, T_{\chi}(x),\\
&\psi_{\chi}^5\ :=\ \forall x \bigl(\, S_{\alpha}(x)\
\leftrightarrow\ E_{\alpha}^{\Uf}(x)\, \bigr),\\
&\psi_{\chi}^6\ :=\ \forall x \bigl(\, T_{\alpha}(x)\
\leftrightarrow\ E_{\alpha}^{\Vgp}(x)\bigr).
\end{align*}
$\mathit{Dom}(\chi)$ is $\{y\}$:
%
%
%
%
\begin{align*}
&\psi_{\chi}^1\ := \ \forall y\bigl(\ S_{\chi}(y)\
\rightarrow\ \exists^{= 1} x\, E_{\alpha}^{\mathit{Uf}}(x,y)\, \bigr),\\
&\psi_{\chi}^2\ := \ \forall x\forall y\bigl(\, E_{\alpha}^{\Uf}(x,y)\
\rightarrow\  S_{\chi}(y)\, \bigr),\\
&\psi_{\chi}^3\ :=\ \forall y\bigl(\, T_{\chi}(y)
\rightarrow \forall x E_{\alpha}^{\Vgp}(x,y)\bigr),\\
&\psi_{\chi}^4\ := \ \forall x\forall y\bigl(\, E_{\alpha}^{\Vgp}(x,y)\
\rightarrow\ T_{\chi}(y)\, \bigr),\\
&\psi_{\chi}^5\ := \forall x \forall y\bigl(\, S_{\alpha}(x,y)\
\leftrightarrow\ E_{\alpha}^{\Uf}(x,y)\, \bigr),\\
&\psi_{\chi}^6\ := \forall x \forall y\bigl(\, T_{\alpha}(x,y)\
\leftrightarrow\  E_{\alpha}^{\Vgp}(x,y)\, \bigr).
\end{align*}
%
%
%
%
%
%
%
%
%
\subsection{Formulae for $\chi = \exists^{s} y\, \alpha$.}
$\mathit{Dom}(\chi) = \{x,y\}$:
%
%
%
%
%
%
\begin{align*}
&\psi_{\chi}^1\ := \ \forall x\forall y\bigl(\ S_{\chi}(x,y)\
\rightarrow\ \exists^{= 1} y\, E_{\alpha}^{\mathit{Uf}}(x,y)\, \bigr),\\
%
%
%
&\psi_{\chi}^2\ := \ \forall x\forall y\bigl(\, E_{\alpha}^{\Uf}(x,y)\
\rightarrow\ \exists y\, S_{\chi}(x,y)\, \bigr),\\
&\psi_{\chi}^3\ :=\ \forall x\forall y\bigl(\, T_{\chi}(x,y)
\rightarrow \forall y E_{\alpha}^{\Vgp}(x,y)\bigr),\\
&\psi_{\chi}^4\ := \ \forall x\forall y\bigl(\, E_{\alpha}^{\Vgp}(x,y)\
\rightarrow\ \exists y\, T_{\chi}(x,y)\, \bigr),\\
&\psi_{\chi}^5\ := \forall x \forall y\bigl(\, S_{\alpha}(x,y)\
\leftrightarrow\ E_{\alpha}^{\Uf}(x,y)\, \bigr),\\
&\psi_{\chi}^6\ := \forall x \forall y\bigl(\, T_{\alpha}(x,y)\
\leftrightarrow\  E_{\alpha}^{\Vgp}(x,y)\, \bigr).
\end{align*}
If $\mathit{Dom}(\chi)$ is either of the sets $\{y\}, \emptyset$, exactly
the same formulae $\psi_{\chi}^1,...,\psi_{\chi}^6$ are used as in the
case where $\chi = \exists^{s} x\, \alpha$ and $\mathit{Dom}(\chi)$ is $\{x\}$ or $\emptyset$.
\vspace{3mm}
\noindent
$\mathit{Dom}(\chi)$ is $\{x\}$:
%
%
%
%
\begin{align*}
&\psi_{\chi}^1\ := \ \forall y\bigl(\ S_{\chi}(x)\
\rightarrow\ \exists^{= 1} y\, E_{\alpha}^{\mathit{Uf}}(x,y)\, \bigr),\\
&\psi_{\chi}^2\ := \ \forall x\forall y\bigl(\, E_{\alpha}^{\Uf}(x,y)\
\rightarrow\  S_{\chi}(x)\, \bigr),\\
&\psi_{\chi}^3\ :=\ \forall x\forall y\bigl(\, T_{\chi}(x,y)
\rightarrow \forall y E_{\alpha}^{\Vgp}(x,y)\bigr),\\
&\psi_{\chi}^4\ := \ \forall x\forall y\bigl(\, E_{\alpha}^{\Vgp}(x,y)\
\rightarrow\ T_{\chi}(x)\, \bigr),\\
&\psi_{\chi}^5\ :=\ \forall x \forall y\bigl(\, S_{\alpha}(x,y)\
\leftrightarrow\ E_{\alpha}^{\Uf}(x,y)\, \bigr),\\
&\psi_{\chi}^6\ :=\ \forall x \forall y\bigl(\, T_{\alpha}(x,y)\
\leftrightarrow\  E_{\alpha}^{\Vgp}(x,y)\, \bigr).
\end{align*}
\subsection{Formulae for $\chi = \chi_1\vee\chi_2$}
%
%
%
%
%
%
$\mathit{Dom}(\chi) = \{x,y\}$:
\begin{align*}
&\psi_{\chi}^1\ := \forall x\forall y \Bigl(\,  S_{\chi}(x,y)\
\leftrightarrow\ \bigl(S_{\chi_1}(x,y)\, \vee\, S_{\chi_2}(x,y)\bigr)\Bigr),\\
&\psi_{\chi}^2\ := \forall x\forall y \Bigl(\,  T_{\chi}(x,y)\
\leftrightarrow\ T_{\chi_1}(x,y)\Bigr),\\
&\psi_{\chi}^3\ := \forall x\forall y \Bigl(\,  T_{\chi}(x,y)\
\leftrightarrow\ T_{\chi_2}(x,y)\Bigr).
\end{align*}
$\mathit{Dom}(\chi)$ is any of the sets $\{x\}, \{y\}, \emptyset$:
\begin{align*}
%
%
%
&\psi_{\chi}^1\ := \forall x \Bigl(\,  S_{\chi}(x)\
\leftrightarrow\ \bigl(S_{\chi_1}(x)\, \vee\, S_{\chi_2}(x)\bigr)\Bigr),\\
&\psi_{\chi}^2\ := \forall x \Bigl(\,  T_{\chi}(x)\
\leftrightarrow\ T_{\chi_1}(x)\Bigr),\\
&\psi_{\chi}^3\ := \forall x \Bigl(\,  T_{\chi}(x)\
\leftrightarrow\ T_{\chi_2}(x)\Bigr).
\end{align*}
%

%
%
\subsection{Formulae for $\chi = \neg \alpha$}
%
%
%
%
$\mathit{Dom}(\chi) = \{x,y\}$:
\begin{align*}
&\psi_{\chi}^1\ := \ \forall x\forall y \bigl(\ S_{\chi}(x,y)\
\leftrightarrow\  T_{\alpha}(x,y)\, \bigr),\\
&\psi_{\chi}^2\ := \ \forall x\forall y\bigl(\ T_{\chi}(x,y)\
\leftrightarrow\  S_{\alpha}(x,y)\, \bigr).
\end{align*}
$\mathit{Dom}(\chi)$ is any of the sets $\{x\}, \{y\},\emptyset$:
\begin{align*}
&\psi_{\chi}^1\ := \ \forall x\bigl(\ S_{\chi}(x)\
\leftrightarrow\  T_{\alpha}(x)\, \bigr),\\
&\psi_{\chi}^2\ := \ \forall x\bigl(\ T_{\chi}(x)\
\leftrightarrow\  S_{\alpha}(x)\, \bigr).
\end{align*}
\subsection{$\chi$ is an atomic formula}
%
%
%
%
%
%
%
$\chi$ is a first-order atom and $\mathit{Dom}(\chi) = \{x,y\}$:
\begin{align*}
&\psi_{\chi}^1\ :=\ \forall x \forall y\bigl(\, S_{\chi}(x,y)\rightarrow \chi\, \bigr),\\
&\psi_{\chi}^2\ :=\ \forall x \forall y\bigl(\, T_{\chi}(x,y)\rightarrow \neg\chi\, \bigr).
\end{align*}
$\chi$ is a first-order atom and $\mathit{Dom}(\chi)$ is $\{x\}$:
\begin{align*}
&\psi_{\chi}^1\ :=\ \forall x \bigl(\, S_{\chi}(x)\rightarrow \chi\, \bigr),\\
&\psi_{\chi}^2\ :=\ \forall x \bigl(\, T_{\chi}(x)\rightarrow \neg\chi\, \bigr).
\end{align*}
$\chi$ is a first-order atom and $\mathit{Dom}(\chi)$ is $\{y\}$:
\begin{align*}
&\psi_{\chi}^1\ :=\ \forall y \bigl(\, S_{\chi}(y)\rightarrow \chi\, \bigr),\\
&\psi_{\chi}^2\ :=\ \forall y \bigl(\, T_{\chi}(y)\rightarrow \neg\chi\, \bigr).
\end{align*}
If $\chi$ is the formula $\depst(x,y)$, then $\mathit{Dom}(\chi) = \{x,y\}$.
We define
\begin{align*}
&\psi_{\chi}^1\ :=\ \neg\exists x\exists^{\geq 2} y\, S_{\chi}(x,y),\\
&\psi_{\chi}^2\ :=\ \neg\exists x\exists y\, T_{\chi}(x,y).
\end{align*}
If $\chi$ is the formula $\depst(y,x)$, then $\mathit{Dom}(\chi) = \{x,y\}$.
We define
\begin{align*}
&\psi_{\chi}^1\ :=\ \neg\exists y\exists^{\geq 2} x\, S_{\chi}(x,y),\\
&\psi_{\chi}^2\ :=\ \neg\exists x\exists y\, T_{\chi}(x,y).
\end{align*}
If $\chi$ is the formula $\depst(x)$ and $\mathit{Dom}(\chi) = \{x,y\}$, we define
\begin{align*}
&\psi_{\chi}^1\ :=\ \neg \exists^{\geq 2} x\, \exists y\, S_{\chi}(x,y),\\
&\psi_{\chi}^2\ :=\ \neg\exists x\exists y\, T_{\chi}(x,y).
\end{align*}
If $\chi$ is the formula $\depst(x)$ and $\mathit{Dom}(\chi) = \{x\}$, we define
\begin{align*}
&\psi_{\chi}^1\ :=\ \neg \exists^{\geq 2} x\, S_{\chi}(x),\\
&\psi_{\chi}^2\ :=\ \neg\exists x\, T_{\chi}(x).
\end{align*}
If $\chi$ is the formula $\depst(y)$ and $\mathit{Dom}(\chi) = \{x,y\}$, we define
\begin{align*}
&\psi_{\chi}^1\ :=\ \neg \exists^{\geq 2} y\, \exists x\, S_{\chi}(x,y),\\
&\psi_{\chi}^2\ :=\ \neg\exists x\exists y\, T_{\chi}(x,y).
\end{align*}
If $\chi$ is the formula $\depst(y)$ and $\mathit{Dom}(\chi) = \{y\}$, we define
\begin{align*}
&\psi_{\chi}^1\ :=\ \neg \exists^{\geq 2} x\, S_{\chi}(x),\\
&\psi_{\chi}^2\ :=\ \neg\exists x\, T_{\chi}(x).
\end{align*}


\begin{thebibliography}{10}\label{bibliography}
%
\addcontentsline{toc}{section}{References}
%
%
%
\bibitem{IEEEonedimensional:kieronski}
S. Benaim, M. Benedikt, W. Charatonik,
E. Kiero\'{n}ski, R. Lenhardt, F. Mazowiecki and J. Worrell.
Complexity of two-variable logic on finite trees.
In \emph{Proceedings of ICALP},
74--88, 2013.
%
%
%
%
%
%
\bibitem{IEEEonedimensional:charatonic}
W. Charatonik and P. Witkowski. 
Two-variable logic with counting and trees. In
\emph{Proceedings of LICS}, 2013.
%
%
%
\bibitem{engstrom}
F. Engstr\"{o}m.
Generalized quantifiers in dependence logic.
\emph{Journal of Logic, Language and Information}, 21(3), 2012.
%
%
%
\bibitem{jufe}
F.  Engstr\"{o}m and
J. Kontinen. Characterizing quantifier extensions of dependence logic.
\emph{Journal of Symbolic Logic}, 78(1): 307-316, 2013.
%
%
%
\bibitem{jufeva}
F. Engstr\"{o}m,
J. Kontinen and J. V\"{a}\"{a}n\"{a}nen.
Dependence logic with generalized quantifiers: axiomatizations.
In \emph{Proceeings of WoLLIC}, 138-152, 2013.
%
%
%
\bibitem{galliani}
P. Galliani.
Inclusion and exclusion dependencies in team semantics - on some logics of imperfect information.
\emph{Annals of Pure and Applied Logic}, 163(1):68-84, 2012.
%
%
%
\bibitem{hannula}
P. Galliani, M. Hannula and J. Kontinen.
Hierarchies in independence logic.
In \emph{Proceedings of CSL}, 263-280, 2013.
%
%
%
\bibitem{gallianihella}
P. Galliani and L. Hella. Inclusion logic and fixed point logic.
In \emph{Proceedings of CSL}, 281-295, 2013.
%
%
%
%
%
%
\bibitem{indep}
E. Gr\"{a}del and J. V\"{a}\"{a}n\"{a}nen.
Dependence and Independence.
\emph{Studia Logica}, 101(2):399-410, 2013.
%
%
%
%
%
%
%
%
%
\bibitem{sandu}
J. Hintikka and G. Sandu.
Informational independence as a semantical phenomenon.
\emph{Logic, Methodology and Philosophy of Science,
Studies in Logic and Foundations of Mathematics}, vol. 126,
571-589, 1989.
%
%
%
\bibitem{hodges}
W. Hodges. Compositional semantics for a langauge of imperfect
information.
\emph{Logic Journal of the IGPL}, 5(4), 1997 (electronic).
%
%
%
\bibitem{IEEEonedimensional:kieronskimichaliszyn}
E. Kiero\'{n}ski and J. Michaliszyn.
Two-variable universal logic with transitive closure.
In \emph{Proceedings of CSL}, 396--410, 2012.
%
%
%
\bibitem{IEEEonedimensional:kieronskitendera}
E.  Kiero\'{n}ski, J. Michaliszyn, I. Pratt-Hartmann and
L. Tendera.
Two-variable first-order logic with equivalence closure.
In \emph{Proceedings of LICS}, 431--440, 2012.
%
%
%
\bibitem{kontinen}
J. Kontinen, A. Kuusisto, P. Lohmann and J. Virtema.
Complexity of two-variable dependence logic and IF-logic.
In \emph{Proceedings of LICS}, 289-298, 2011.
%
%
%
%
%
%
\bibitem{kuusisto2}
A. Kuusisto.
Defining a double team semantics for generalized quantifiers.
Technical report, Tampub 2012.
%
%
%
\bibitem{kuusisto3}
A. Kuusisto.
Defining a double team semantics for generalized quantifiers
%
(extended version).
%
Technical report, Tampub 2013.
%
%
%
\bibitem{kuusistodistributed}
A. Kuusisto.
Infinite networks, halting and local algorithms.
\emph{In Proc. of the Fifth International Symposium on Games, Automata, Logics and
Formal Verification (GandALF)}, 2014.
%
%
%
\bibitem{kuusisto1}
A. Kuusisto.
Logics of imperfect information without identity.
A parallel publication of an article in the 
%
proceedings of the 2010 ESSLLI Workshop on Dependence and Independence in 
%
Logic. TamPub 2011.
%
%
%
\bibitem{kuusistocsl2013}
A. Kuusisto.
Modal logic and distributed message passing automata. In
\emph{Proceedings of CSL}, 2013.
%
%
%
\bibitem{kuusisto6}
A. Kuusisto.
Resource conscious quantification and ontologies 
%
with degrees of significance.
%
Technical report, TamPub 2010.
%
%
%
\bibitem{kuusistoooo}
A. Kuusisto.
Some Turing-complete extensions of
%
first-order logic.
%
CoRR abs/1405.1715 (2014).
%
%
%
\bibitem{lindstrom}
P. Lindstr\"{o}m. First order predicate logic with generalized quantifiers.
\emph{Theoria}, 32, 1966.
%
%
%
\bibitem{mann}
 A. Mann, G. Sandu and M. Sevenster.
 \emph{Independence-friendly Logic - A Game
 Theoretic Approach}. Cambridge University Press, 2011.
%
%
%
\bibitem{IEEEonedimensional:zeume}
A. Manuel and T. Zeume.
Two-variable logic on 2-dimensional structures.
In \emph{Proceedings of CSL}, 2013.
%
%
%
\bibitem{pratthartmann}
I. Pratt-Hartmann.
Complexity of the two-variable fragment with counting quantifiers.
\emph{Journal of Logic, Language and Information}, 14(3): 369-395, 2005.
%
%
%
\bibitem{steels}
L. Steels and F. Kaplan. AIBO's first words. The social learning of language and meaning. In
\emph{Evolution of Communication}, Vol. 4, no. 1,
Amsterdam: John
Benjamins Publishing Company, 2001.
%
%
%
\bibitem{IEEEonedimensional:tendera}
W. Szwast and L. Tendera.
$FO^2$ with one transitive relation is decidable.
In \emph{Proceedings of STACS}, 317-328, 2013.
%
%
%
\bibitem{vaananen}
 J. V\"a\"an\"anen.
 \emph{Dependence Logic}.
 Cambridge University Press, 2007.
%
%
%
\bibitem{wittgenstein}
L. Wittgenstein.
\emph{Philosophical Investigations}. Blackwell, 1953.
%
%
%
%
%
%
%
\end{thebibliography}
\end{document}